\documentclass[a4paper, reqno]{amsart}

\title[{}]{Quasi-lisse extension of affine $\sll_2$ \`{a} la Feigin--Tipunin}

\author{Thomas Creutzig}
\address{Department of Mathematical and Statistical Sciences, University of Alberta, 632 CAB, Edmonton, Alberta, Canada T6G 2G1}
\email{creutzig@ualberta.ca}

\author{Shigenori Nakatsuka}
\address{Department of Mathematical and Statistical Sciences, University of Alberta, 632 CAB, Edmonton, Alberta, Canada T6G 2G1}
\email{shigenori.nakatsuka@ualberta.ca}

\author{Shoma Sugimoto}
\address{Faculty of Mathematics, Kyushu University, Fukuoka 819-0386, Japan}
\email{shomasugimoto361@gmail.com}

\usepackage{latexsym}
\usepackage{amsmath}
\usepackage{amsfonts}
\usepackage{amssymb}
\usepackage{amsxtra}
\usepackage{mathrsfs}
\usepackage{eucal}
\usepackage{tikz} \usetikzlibrary{calc} \tikzset{>=latex} \usetikzlibrary{backgrounds}
\usepackage[all]{xy}
\usepackage{color}
\usepackage{braket}
\usepackage{graphics, setspace}
\usepackage{etoolbox, textcomp}
\usepackage{comment}
\usepackage{changepage}
\usepackage{xcolor}
\definecolor{rouge}{rgb}{0.85,0.1,.4}
\definecolor{bleu}{rgb}{0.1,0.2,0.9}
\definecolor{violet}{rgb}{0.7,0,0.8}
\usepackage[colorlinks=true,linkcolor=bleu,urlcolor=violet,citecolor=rouge]{hyperref}
\usepackage{cleveref}
\usepackage{lscape}

\usepackage{tikz}			
\usetikzlibrary{matrix}
\usetikzlibrary{automata, arrows.meta, positioning,cd}
\usepackage{dynkin-diagrams}
\usetikzlibrary{cd}			

\newtheorem{definition}{Definition}[section]
\newtheorem{proposition}[definition]{Proposition}
\newtheorem{theorem}[definition]{Theorem}
\newtheorem{corollary}[definition]{Corollary}
\newtheorem{lemma}[definition]{Lemma}
\newtheorem{example}[definition]{Example}
\newtheorem{remark}[definition]{Remark}
\newtheorem{conjecture}[definition]{Conjecture}

\newtheorem*{thmA}{Theorem A}
\newtheorem*{thmB}{Theorem B}
\newtheorem*{thmC}{Theorem C}
\newtheorem*{thmD}{Theorem D}

\numberwithin{equation}{section}

\newcommand{\Z}{\mathbb{Z}}

\newcommand{\C}{\mathbb{C}}
\newcommand{\V}{\mathcal{V}}
\newcommand{\as}{\mathbf{a}}
\newcommand{\s}{\mathbf{s}}

\newcommand{\Hom}{\operatorname{Hom}}

\newcommand{\Com}{\operatorname{Com}}

\newcommand{\Ker}{\operatorname{Ker}}

\newcommand{\im}{\operatorname{Im}}

\newcommand{\ch}{\operatorname{ch}}
\newcommand{\tr}{\operatorname{tr}}

\newcommand{\Q}{\mathbb{Q}}

\newcommand{\W}{\mathcal{W}}
\newcommand{\cU}{\mathcal{U}}
\newcommand{\scU}{\mathscr{U}}
\newcommand{\fN}{\mathfrak{N}}

\newcommand{\g}{\mathfrak{g}}

\newcommand{\h}{\mathfrak{h}}
\newcommand{\bo}{\mathfrak{b}}

\newcommand{\gl}{\mathfrak{gl}}
\newcommand{\sll}{\mathfrak{sl}}
\newcommand{\af}{\mathrm{aff}}
\newcommand{\ssqrt}[1]{\operatorname{\sqrt{\smash[b]{#1}}}}

\newcommand{\FT}{\mathrm{FT}_p}
\newcommand{\sFT}{\mathcal{FT}_p}

\newcommand\doi[2]{\href{http://dx.doi.org/#1}{#2}}

\newcommand{\afWb}[3]{\mathcal{W}_{{#1},{#2}}^{[#3]}}
\newcommand{\afWp}[2]{\mathcal{W}_{{#1},{#2}}^{+}}
\newcommand{\afWm}[2]{\mathcal{W}_{{#1},{#2}}^{-}}

\newcommand{\afXp}[2]{\mathcal{X}_{{#1},{#2}}^{+}}
\newcommand{\afXm}[2]{\mathcal{X}_{{#1},{#2}}^{-}}
\newcommand{\afXpm}[2]{\mathcal{X}_{{#1},{#2}}^{\pm}}
\newcommand{\afXmp}[2]{\mathcal{X}_{{#1},{#2}}^{\mp}}
\newcommand{\genafXp}[1]{\mathcal{X}_{#1}^{+}}

\begin{document}
\maketitle

\begin{abstract}
We study the affine analogue $\FT(\sll_2)$ of the triplet algebra.
We show that $\FT(\sll_2)$ is quasi-lisse and the associated variety is the nilpotent cone of $\sll_2$. 
We realize $\FT(\sll_2)$ as the global sections of a sheaf of vertex algebras in the spirit of Feigin--Tipunin and thereby construct infinitely many simple modules and, in particular solve a conjecture by Semikhatov and Tipunin.
We introduce the Kazama--Suzuki dual superalgebra $s\W_p(\sll_{2|1})$ of $\FT(\sll_2)$ and their singlet type subalgebras $s\mathcal{M}_p(\sll_{2|1})$ and $\mathcal{M}_p(\sll_2)$ and show their correspondence of categories. For $p=1$, we show the logarithmic Kazhdan--Lusztig correspondence for these (super)algebras and, in particular, show that the quantum group corresponding to $s\mathcal{M}_1(\sll_{2|1})$ is the unrolled restricted quantum supergroup $u^H_{-1}(\sll_{2|1})$ as suggested by Semikhatov and Tipunin.
\end{abstract}

\section{Introduction}

One of the most important problems in the theory of vertex operator algebras (VOAs) is to establish and understand the structure of braided tensor category on their module categories.
The existence of braided tensor structure is highly nontrivial and was first established by Kazhdan and Lusztig \cite{KL} for the the affine vertex algebra $V^k(\g)$ associated with simple Lie algebra $\g$ at shifted levels $k+h^\vee\notin \Q_{\geq0}$ by the dual Coxeter number $h^\vee$ on the category of $\g$-integrable $V^k(\g)$-modules.
The braided tensor category turns out equivalent to the module category of the quantum group $U_q(\g)$ at $q=e^{2\pi\ssqrt{-1}/r^\vee(k+h^\vee)}$ where $r^\vee$ is the lacity of $\g$.

The general theory on the existence and construction of braided tensor structure on module categories for VOAs has been layed down by Huang--Lepowsky \cite{HL} and Huang--Lepowsky--Zhang \cite{HLZ}, see also \cite{CKM,M} for some recent works. 
The braided tensor categories thus obtained ---although difficult to establish in examples--- are expected to be equivalent to the categories of modules over certain (quasi-)Hopf algebras. This conjectural equivalence is broadly called the logarithmic Kazhdan--Lusztig correspondence, see \cite{CGR,CLR,CRR,FGST1,FGST2,L}.
The term ``logarithmic" means that the representation category is non-semisimple and thus has a rich structure.
Besides being interesting in their own right, the logarithmic vertex algebras and their module characters are also important in the connection to invariants of 3-manifolds \cite{CCFGH,CCFFGHP} and quantum field theories \cite{CDGG,BCDN}.

The affine vertex algebras $V^k(\g)$ at shifted levels $k+h^\vee\notin \Q_{\geq0}$ include the admissible case and are important for the representation theory \cite{KW,Ar4,AvE2}. On the other hand, the quantum groups $U_q(\g)$ are at roots of unity and behave quite differently since the module categories are non-semisimple and the center becomes larger. The latter gives rise to several choices of the quantum groups $u_q(\g)$, called small or restricted quantum groups.
In the light of correspondences, a natural question would be what kind of vertex algebra corresponds to $u_q(\g)$. 
Natural candidates would be $V^k(\g)$ themselves, the $\W$-algebras $\W^k(\g,f)$ \cite{FF3,KRW} obtained from $V^k(\g)$ via the quantum Drinfeld--Sokolov reduction (a.k.a BRST reduction), including the generalization to the superalgebras, and their extensions or quotients.

Feigin--Gainutdinov--Semikhatov--Tipunin \cite{FGST1,FGST2} gave a conjecture for $\g=\sll_2$ that the vertex algebra would be the triplet algebra $\W(p)$ \cite{K} which is an extension of the $(1,p)$ Virasoro vertex algebra $L(c_{1,p},0)$.
The correspondence is established recently \cite{CLR,CLR2, GN} under the quasi-Hopf algebra modification of $u_q(\sll_2)$ after a series of important works \cite{AM,AM2,CGR,KoSa,MY,NT,TW}.

Feigin and Tipunin \cite{FT} gave conjectural VOAs $\W(p)_Q$ corresponding to $u_q(\g)$ as extensions of principal $\W$-algebras appearing in the global sections of sheaves of certain vertex algebras over flag varieties.
Although the geometric construction enables us to obtain some simple modules \cite{Sug,Sug2} and some logarithmic modules are known \cite{AM3}, the representation theory is still unknown and, in particular, the $C_2$-cofiniteness of $\W(p)_Q$ itself, which implies that the representation category is finite, remains as an important conjecture.

Since the $\W$-algebras are defined through the BRST reduction from $V^k(\g)$, it is natural to study first the representation theory of $V^k(\g)$ and then the properties of the BRST reduction to finally understand the representation theory of the $\W$-algebras. 
Indeed, the $C_2$-cofiniteness of the minimal series representation of the principal $\W$-algebras was established by Arawaka in this way \cite{Ar3}.

Since the construction of $\W(p)_Q$ is generalized to all the $\W$-algebras, including the affine case $V^k(\g)$ \cite{CNS}, we expect that we can use this approach for $\W(p)_Q$.

\subsection{Main results}
In this paper, we consider the Feigin--Tipunin type construction for the rank one case, namely $\FT(\sll_2)$ and $\W(p)$ extending $V^k(\sll_2)$ at $k=-2+1/p$ and $L(c_{1,p},0)$ $(p\geq 1)$ and study their relations.
The free field realizations \cite{Wak,FF,FF3} of $V^k(\sll_2)$ and $L(c_{1,p},0)$ enable us to construct equivariant sheaves of vertex algebras over the flag variety $\mathbb{P}^1=\mathrm{SL}_2/\mathrm{B}$. The global sections
\begin{align}\label{intro geometric setting}
    \FT(\sll_2)=H^0(\mathrm{SL}_2\times_\mathrm{B}\beta\gamma\otimes V_{\sqrt{p}A_1}),\quad \FT(\sll_2,f)=H^0(\mathrm{SL}_2\times_\mathrm{B}V_{\sqrt{p}A_1})
\end{align}
are vertex algebras extending $V^k(\sll_2)$ and $L(c_{1,p},0)$, respectively. 
The restriction of these algebras to the fiber at the identity is injective and gives another realization
\begin{align}\label{Screenings}
    \mathrm{FT}_p(\sll_2)\simeq \mathrm{Ker}\ Q_-\mid_{\beta\gamma\otimes V_{\sqrt{p}A_1}},\quad \mathrm{FT}_p(\sll_2,f)\simeq \mathrm{Ker}\ \mathcal{Q}-\mid_{V_{\sqrt{p}A_1}},
\end{align}
by the kernel of screening operators. 
The latter one is the geometric realization of the triplet algebra $\W(p)$ due to Feigin--Tipunin \cite{FT} and the former one is Theorem \ref{identificatio of affine FT} ($p\geq 2$) and Proposition \ref{decomposition for p=1} $(p=1)$ in this paper.
This realization of $\mathrm{FT}_p(\sll_2)$ appears in the works \cite{ST1} on the logarithmic Kazhdan--Lusztig correspondence and the first VOA constructions are \cite{Ad3, ACGY}. 
In physics, these algebras appear  as chiral algebras in four dimensional physics, namely in so-called Argyres--Douglas theories (conjectured in \cite{C} and proven in \cite{ACGY}) as well as 
 large level limits of the kernel VOAs  of $S$-duality \cite{CG, CGL}.

Although $\W(p)$ is a $C_2$-cofinite vertex algebra \cite{AM}, $\mathrm{FT}_p(\sll_2)$ lies in a wider class, called \emph{quasi-lisse} \cite{AK}, which is a common property of vertex algebras appearing in four dimensional physics e.g. \cite{BLLPRvR,Ar2}.
More precisely, our first theorem tells the following.
\begin{thmA}[{Theorem \ref{theorem: associated variety}}] The vertex algebra $\FT(\sll_2)$ has finite strong generators and the associated variety of $\FT(\sll_2)$ is the nilpotent cone of $\sll_2$. Hence, $\FT(\sll_2)$ is quasi-lisse.
\end{thmA}
We construct explicit strong generators for the proof. In \cite{ST1}, the subalgebra generated by them is called $\mathbf{W}(2,(2p)^{\times 3\times 3})$ and thus we prove a conjecture by Semikhatov and Tipunin that $\mathbf{W}(2,(2p)^{\times 3\times 3})$ agrees with $\FT(\sll_2)$.

Since $\FT(\sll_2)$ is not $C_2$-cofinite, there must be infinitely many simple modules. To construct such modules, we replace the fiber $\beta\gamma\otimes V_{\sqrt{p}A_1}$ by their modules or certain $V^k(\sll_2)$-submodules, following the idea of \cite{FT,Sug,Sug2}. 
We parameterize the fibers by
\begin{align}\label{intro: fibers}
    \Pi[b]\otimes \V_{r,s}\supset \tau(\Pi[b]\otimes \V_{r,s}),\quad ([b]\in \C/\Z,\ s=1,2,\ 1\leq r\leq p),
\end{align}
and $\tau(\Pi[b]\otimes \V_{r,s})$ is proper only for special cases $[b]=[0], [\frac{r}{p}]$, see \S \ref{section:FT-modules} for detail.
\begin{thmB}[{Theorem \ref{FT(sl2,0)}/\ref{simplicity of X}}]
The higher cohomology $H^{>0}(\mathrm{SL}_2\times_\mathrm{B}M)$ vanishes for the fiber $M$ as in \eqref{intro: fibers}. Moreover, the global sections
\begin{align*}
    & \afWb{r}{s}{b}=H^0(\mathrm{SL}_2\times_\mathrm{B}\Pi[b]\otimes \V_{r,s}),\quad ([b]\in \C/\Z\ \backslash\{[0],[\tfrac{r}{p}]\}),\\
    &\afXp{r}{s}=H^0(\mathrm{SL}_2\times_\mathrm{B}\tau(\Pi[0]\otimes \V_{r,s})),\\
    &\afXm{r}{s}=H^0(\mathrm{SL}_2\times_\mathrm{B}\tau(\Pi[\tfrac{r}{p}]\otimes \V_{r,s})),\quad (r\neq p),
\end{align*}
and their spectral flow twists are simple $\mathrm{FT}_p(\sll_2)$-modules.
\end{thmB}
These modules are realized inside the fiber $M$ at the identity and the simplicity of $\mathcal{X}_{r,s}^{\pm}$ and the discrete family $\afWb{r}{s}{b}$ ($[b]\in \frac{1}{p}\Z/\Z$) under this realization was conjectured by Semikhatov--Tipunin \cite[\S 7.4]{ST1}.

Next, we consider the relationships of $\FT(\sll_2)$ and $\W(p)$. Indeed, $\W(p)$ itself is obtained as the BRST reduction $H_{DS}^0$ of $\mathrm{FT}_p(\sll_2)$ \cite{ACGY}, which is expected by construction \eqref{intro geometric setting} as the fibers are already related by $H_{DS}^0$. These two facts can be formulated as the commutativity of two functors
\begin{align}\label{intro: commutativity of functors}
        H_{DS}^0(H^0(\mathrm{SL}_2\times_\mathrm{B}M))\simeq H^0(\mathrm{SL}_2\times_\mathrm{B}H_{DS}^0(M))
\end{align}
applied for $M=\beta\gamma\otimes V_{\sqrt{p}A_1}$. Indeed, this property is generalized for modules. 
\begin{thmC}[{Theorem \ref{compatibility of FT algeras for sl2}}] The commutativity \eqref{intro: commutativity of functors} holds for fibers $M$ as in \eqref{intro: fibers}.
Hence, the simple $\W(p)$-modules $\W_{r,s}(=H^0(\mathrm{SL}_2\times_\mathrm{B}\V_{r,s}))$
are obtained from simple $\FT(\sll_2)$-modules by the BRST reduction:
\begin{align*}
    \W_{r,s}\simeq H_{DS}^0(\mathcal{X}_{r,s}^\pm)\ (r\neq p),\quad \W_{p,s}\simeq H_{DS}^0(\mathcal{X}_{p,s}^\pm).
\end{align*}
\end{thmC}
We note that thanks to the isomorphism $\W(p)\simeq H_{DS}^0(\FT(\sll_2))$, Theorem A implies the $C_2$-cofiniteness of $\W(p)$ \cite{AM} (Collorary \ref{C2 cofiniteness of triplet}).

The characterization of $\FT(\sll_2)$ and $\W(p)$ by the screening kernels \eqref{Screenings} inside free field algebras is important to understand their representation category and the logarithmic Kazhdan--Lustzig correspondence as was first noticed by Semikhatov and Tipunin \cite{ST,ST1} in general. 
Let $A$ be a free field algebra and $V$ a vertex subalgebra characterized by the joint kernel of screening operators $\{S_i\}$. Then under a mild assumption, $\{S_i\}$ generate a Nichols algebra $\fN$ inside the category $\mathrm{Rep}^{\mathrm{wt}}(A)$ of weight $A$-modules \cite{L}. Then it is expected \cite{CLR,ST1,L} that the category $\mathrm{Rep}^{\mathrm{wt}}(V)$ of weight $V$-modules has the structure of braided tensor category and is equivalent to the category of $(\fN,\fN)$ Yetter--Drinfeld modules in $\mathrm{Rep}^{\mathrm{wt}}(A)$:
\begin{align}\label{intro: logKL}
    \mathrm{Rep}^{\mathrm{wt}}(V)\simeq {}^{\fN}_\fN\mathcal{YD}(\mathrm{Rep}^{\mathrm{wt}}(A)).
\end{align}
The category $\mathrm{Rep}^{\mathrm{wt}}(A)$ is often realized as a representation category of a commutative (quasi-)Hopf algebra and then ${}^{\fN}_\fN\mathcal{YD}(\mathrm{Rep}^{\mathrm{wt}}(A))$ as the category of the braided Drinfeld double of $\fN$ which is likely to coincide with a quasi-quantum group $\mathscr{U}$ thanks to the Heckenberger's classification \cite{He} of finite dimensional Nichols algebras \cite{GLO,LS}.

In the case of $\W(p)$, $\mathscr{U}$ is realized as a quasi-Hopf algebra modification of the small quantum group $u_q(\sll_2)$ \cite{CGR,CLR,GN}.
It it worth mentioning that the subalgebra $\mathcal{M}(p)\subset \W(p)$, called the singlet algebra \cite{A2}, does not need the modification and indeed $\mathscr{U}$ is realized as the unrolled restricted version $u_q^H(\sll_2)$ of $u_q(\sll_2)$ \cite{CLR}.

In the case of $\FT(\sll_2)$, Semikhatov and Tipunin \cite{ST1} conjectured that $\mathscr{U}$ is a version of the small quantum group $u_q(\sll_{2|1})$. Again, it is convenient to introduce the singlet type subalgebra $\mathcal{M}_p(\sll_2)\subset \FT(\sll_2)$ and, moreover, to switch to the Kazama--Suzuki dual (KS dual) side \cite{KaSu} as the dual of $V^k(\sll_2)$ is the $\mathcal{N}=2$ superconfromal algebra $\W^\ell(\sll_{2|1})$ at dual level $\ell$. 
The relations of these algebras are summarized as follows.
\begin{center}
\begin{tikzpicture}[x=2mm,y=2mm]
\useasboundingbox (0,0) rectangle (40,15);
\node (A1) at (0,11) {$\W(p)$ };
\node (B1) at (0,6) {$\mathcal{M}(p)$};
\node (C1) at (0,1) {$L(c_{1,p},0)$};
\node (A2) at (20,11) {$\FT(\sll_2)$ };
\node (B2) at (20,6) {$\mathcal{M}_p(\sll_2)$};
\node (C2) at (20,1) {$V^k(\sll_2)$};
\node (A3) at (40,11) {$s\W_p(\sll_{2|1})$};
\node (B3) at (40,6) {$s\mathcal{M}_p(\sll_{2|1})$};
\node (C3) at (40,1) {$\W^\ell(\sll_{2|1})$};
\node at (0,3) {$\cup$};
\node at (0,8) {$\cup$};
\node at (20,3) {$\cup$};
\node at (20,8) {$\cup$};
\node at (40,3) {$\cup$};
\node at (40,8) {$\cup$};
\draw[<->] (A2) -- (A3);
\draw[<->] (B2) -- (B3);
\draw[<->] (C2) -- (C3);
\draw[densely dotted, ->] (A2) -- (A1);
\draw[densely dotted,->] (B2) -- (B1);
\draw[densely dotted,->] (C2) -- (C1);
\node at (30,13) {{\small KS dual}};
\node at (10,13) {{\small $H_{DS}$}};
\end{tikzpicture}
\end{center}
The Kamaza-Suzuki dual algebras enjoy the correspondence \cite{FST,CGNS} for the categories of weight modules and the intertwiners for modules (Theorem \ref{Feigin-Semikhatov duality}).
Since all of these algebras are characterized by the screening kernel inside free field algebras, we can use the formulation \eqref{intro: logKL} for the logarithmic Kazhdan--Lusztig correspondence.
In particular, the quasi-quantum group $\mathscr{U}$ for $s\mathcal{M}_p(\sll_{2|1})$ is the honest unrolled restricted quantum supergroup $u_q^H(\sll_{2|1})$.
For $p=1$, these algebras are also obtained as vertex (super)algebra extension of the singlet algebra $\mathcal{M}(2)$ and thus the logarithmic Kazhdan--Lusztig correspondence for $\mathcal{M}(2)$ can be used to show the correspondence for this case. 
\begin{thmD}[{Theorem \ref{log KL theorem}}]
For $p=1$, the logarithmic Kazhdan--Lusztig correspondence \eqref{intro: logKL} holds and, in particular, 
\begin{align*}
    \mathrm{Rep}^{\mathrm{wt}}(s\mathcal{M}_1(\sll_{2|1}))\simeq \mathrm{Rep}^{\mathrm{wt}}(u_{-1}^H(\sll_{2|1})).
\end{align*}
\end{thmD}

\subsection{Outlook}
It is important to study systematically the Feigin--Tipunin type extension of the $\W$-algebras and especially of the affine vertex algebras. We hope that one might be able to understand their representation categories following the ideas of the geometric representation theory at the level of abelian categories and then to study the logarithmic Kazhdan--Lusztig correspondence at the level of braided tensor categories. 
It is plausible that the Feigin--Tipunin type extension of affine vertex algebras are quasi-lisse and their associated varieties likely coincide with the nilpotent cone and that the Feigin--Tipunin type extension of other $\W$-algebras are uniformly obtained with the BRST reduction. It implies that we would obtain a new family of quasi-lisse vertex algebras and especially settle the $C_2$-cofiniteness conjecture of the original algebra $\W(p)_Q$.
For the hook-type $\W$-algebras, we expect that the degenerate case $p=1$ can be understood by using the singlet algebras $\mathcal{M}(p)$.
Another important direction is their possible applications to the invariants of 3-manifolds, namely q-series called the homological blocks \cite{CCFGH,CCFFGHP}. The invariants depend on the choice of gauge group $G$. In the case of $G=\mathrm{SL}_2$, they are expressed in terms of the module characters of the singlet algebras $\mathcal{M}(p)$. 
We hope that the module characters of the Feigin--Tipunin type extension of affine vertex algebra, where the gauge group $G$ is clear, would lead a more direct relationships. 

\subsection*{Organization of the paper} 
In \S \ref{section: basics}, we introduce $\FT(\sll_2)$ and derive the presentation as screening kernel. 
In \S\ref{section: triplet}, we review the triplet algebra $\W(p)$.  In \S\ref{section: weight V-modules}, we study free field representations of $V^k(\sll_2)$ appearing in the construction of $\FT(\sll_2)$-modules. 
In \S \ref{section:FT-modules}, we introduce $\afWb{r}{s}{b}$ and $\afXpm{r}{s}$ as $(\mathrm{B},V^k(\sll_2))$-bimodules and study their properties based on \S\ref{section: triplet}-\ref{section: weight V-modules}. 
In \S\ref{section: main results}, we show Theorem B and C.
In \S\ref{section: associated variety}, we show Theorem A and propose a conjecture on the relation of the relative semi-infinite cohomology for $\gl_1$ and associated varieties, which would lead to yet another proof of the $C_2$-cofiniteness of $\W(p)$ and some new $C_2$-cofinite VOAs.
In \S\ref{section: the case p=1}, we treat the degenerate case $p=1$. 
In \S\ref{section: log KL}, we study the logarithmic Kazhdan--Lusztig correspondence for $\FT(\sll_2)$ and related algebras and prove Theorem D. 

\subsection*{Acknowledgements}
We appreciate many fruitful discussions on related topics with Simon Lentner. S.N. wishes to express his gratitude to Jinwei Yang for valuable discussions.
T.C. is supported by NSERC Grant Number RES0048511, S.N by JSPS Overseas Research Fellowships Grant Number 202260077 and S.S. by JSPS
KAKENHI Grant number 22J00951. A part of this work is done while S.S. was staying at University of Alberta. He is grateful to the institute for the hospitality.

\section{extension of affine vertex algebras}\label{section: basics}
\subsection{Equivariant VOA bundles}\label{Sec: construction}
Let $\pi^\alpha$ be the Heisenberg vertex algebra generated by the field $\alpha(z)$ satisfying the OPE 
\begin{align*}
\alpha(z)\alpha(w)\sim \frac{2}{(z-w)^2},
\end{align*}
$\pi^\alpha_{\lambda}$ the Fock module of $\pi^\alpha$ of highest weight $\lambda\in \C\alpha$ and $V_{\sqrt{p}A_1}=\bigoplus_{n\in \Z}\pi^\alpha_{\sqrt{p}n\alpha}$ the lattice vertex algebra associated with the rescaled lattice $\sqrt{p}A_1=\sqrt{p}\Z\alpha$ ($p> 0$). 
By the Wakimoto realization \cite{Wak,FF}, the universal affine vertex algebra $V^k(\sll_2)$ of $\sll_2$ at level $k=-2+\frac{1}{p}$ is realized as 
\begin{align}\label{sl2 wakimoto}
\begin{split}
\mu_{\mathrm{Wak}}\colon &V^k(\sll_2)\hookrightarrow \beta\gamma \otimes \pi^{\alpha}\\
&e\mapsto \beta\otimes 1,\quad h\mapsto -2\gamma\beta\otimes 1-1\otimes \tfrac{1}{\sqrt{p}}\alpha,\\
& f\mapsto -\gamma^2\beta\otimes 1-\gamma\otimes \tfrac{1}{\sqrt{p}}\alpha+(-2+\tfrac{1}{p})\partial \gamma\otimes 1\\
&L_{\mathrm{sug}}\mapsto \beta\partial \gamma\otimes 1+1\otimes (\tfrac{1}{4}\alpha^2+\sqrt{p}\partial \varpi)
\end{split}
\end{align}
where the symbol $:AB:$ of the normally ordered product is omitted, $\beta\gamma$ denotes the $\beta\gamma$-system, $L_{\mathrm{sug}}$ is the conformal vector by the Segal--Sugawara construction and $\varpi=\frac{1}{2}\alpha$ the fundamental weight. 
The image is contained in
\begin{align}\label{sl2 long screening}
V^k(\sll_2)\subset \mathrm{Ker}\left(Q_+\colon \beta\gamma\otimes \pi^\alpha \rightarrow \beta\gamma\otimes \pi^\alpha_{\sqrt{p}\alpha}\right),\quad 
Q_+=\int Y(\beta\otimes e^{\sqrt{p}\alpha},z) dz.
\end{align}
Note that $\beta\otimes e^{\sqrt{p}\alpha}$ lies in $\beta\gamma\otimes V_{\sqrt{p}A_1}$ and thus we have 
\begin{align*}
Q_+\colon \beta\gamma\otimes V_{\sqrt{p}A_1}\rightarrow \beta\gamma\otimes V_{\sqrt{p}A_1}.
\end{align*}
It is straightforward to show that the assignment
\begin{align}\label{f-action}
f \mapsto Q_+,
\quad   h|_{\beta\gamma\otimes \pi^{\alpha}_{\lambda}}=-\tfrac{1}{\sqrt{p}}(h,\lambda),\quad (\lambda \in \sqrt{p}A_1)
\end{align}
defines an action of the lower Borel subalgebra $\bo\subset \sll_2$, which integrates to the action of the Borel subgroup $\mathrm{B}=\mathrm{H}\ltimes \mathrm{N}_-$.
Therefore, we may construct the sheaf of vertex algebras over the flag variety $\mathbb{P}^1=\mathrm{SL_2}/\mathrm{B}$ by the local sections of the equivariant vector bundle 
$$\sFT(\sll_2):=\mathrm{SL}_2\times_\mathrm{B} (\beta\gamma \otimes V_{\sqrt{p}A_1}).$$
The affine Feigin--Tipunin algebra $\mathrm{FT}_p(\sll_2)$ of $\sll_2$ is, by definition, the vertex algebra on the global sections:
\begin{align*}
\mathrm{FT}_p(\sll_2):=H^0(\sFT(\sll_2)).
\end{align*}
It contains $V^k(\sll_2)$ as a vertex subalgebra since the subalgebra of the constant sections recovers \eqref{sl2 long screening}:
\begin{align*}
    H^0(\sFT(\sll_2))^{\mathrm{SL}_2}\simeq (\beta\gamma \otimes V_{\sqrt{p}A_1})^\mathrm{B}\simeq (\beta\gamma \otimes \pi^\alpha)^{\mathrm{N}_-}=
    \mathrm{Ker}_{\beta\gamma\otimes \pi^\alpha}Q_+.
\end{align*}

The construction of $\mathrm{FT}_p(\sll_2)$ is a natural analogue of the celebrated construction due to Feigin and Tipunin \cite{FT} of logarithmic extensions of the simply-laced principal $\W$-algebras. The relevant case is the extension of the Virasoro $(1,p)$-model $L(c_{1,p},0)$ of central charge $c_{1,p}:=1-6(p-1)^2/p$, also known as the triplet algebra $\W(p)$ \cite{K,AM}. 
It is obtained by taking the BRST reduction of the realization \eqref{sl2 wakimoto}, which gives an embedding 
\begin{align}\label{vir wakimoto}
\mu_{\mathrm{Wak}}\colon &L(c_{1,p},0)\hookrightarrow  \pi^{\alpha},\quad 
L\mapsto \omega_{1,p}=\tfrac{1}{4}\alpha^2+\tfrac{p-1}{2\sqrt{p}}\partial \alpha,
\end{align}
whose image is contained in the kernel of the cohomology class of the screening operator $\mathcal{Q}_+=[Q_+]$, namely 
\begin{align}\label{free field realization of vir}
L(c_{1,p},0)\subset \mathrm{Ker}\left(\mathcal{Q}_+\colon \pi^\alpha \rightarrow \pi^\alpha_{\sqrt{p}\alpha}\right),\quad 
\mathcal{Q}_+=\int Y(e^{\sqrt{p}\alpha},z) dz.
\end{align}
By replacing ${Q}_+$ with $\mathcal{Q}_+$ in \eqref{f-action}, we obtain a $\mathrm{B}$-action on $V_{\sqrt{p}A_1}$, giving rise to the vertex algebra
$$\mathrm{FT}_p(\sll_2,f):=H^0(\mathrm{SL}_2\times_\mathrm{B} V_{\sqrt{p}A_1}).$$
By \cite{FT,Sug}, it is known that the global sections $\mathrm{FT}_p(\sll_2,f)$ embeds into the fiber at the identity, which gives rise to another realization 
\begin{align}\label{triplet by short screening}
   \mathrm{FT}_p(\sll_2,f)\simeq \mathrm{Ker}\left(\mathcal{Q}_-\colon \pi^\alpha\rightarrow \pi^\alpha_{-\frac{1}{\sqrt{p}}\alpha}\right),\quad \mathcal{Q}_-=\int Y(e^{-\frac{1}{\sqrt{p}}\alpha},z)dz,
\end{align}
which is the definition of the triplet algebra $\W(p)$ \cite{K,AM}.
We will obtain the analogous realization for $\mathrm{FT}_p(\sll_2,f)$, which appears in \cite{ACGY,ST1}. 

\subsection{Inverse Hamiltonian reduction}\label{sec: iHR}
It is useful to realize $V^k(\sll_2)$ by using $L(c_{1,p},0)$, called the inverse Hamiltonian reduction \cite{Ad}.
Let $V_{\Z u\oplus \Z v}$ denote the lattice vertex superalgebra associated with the lattice $\Z u\oplus \Z v$ equipped with bilinear form defined by $(u,u)=1=-(v,v)$. 
By using the vertex subalgebra 
$$\Pi[0]:=\bigoplus_{n\in \Z}\pi^{u,v}_{n(u+v)}\subset V_{\Z u\oplus \Z v},$$
we may realize the $\beta\gamma$-system as follows \cite{FMS}: 
\begin{align}\label{FMS}
\begin{split}
\mu_{\mathrm{FMS}}\colon &\beta\gamma \xrightarrow{\simeq} \mathrm{Ker} \left( Q_{\mathrm{FMS}}\colon \Pi[0]\rightarrow V_{\Z u\oplus \Z v} \right), \quad Q_{\mathrm{FMS}}:=\int Y(e^u,z)dz,\\
&\beta\mapsto e^{u+v},\quad \gamma \mapsto -ue^{-(u+v)}.
\end{split}
\end{align}
By composing it with \eqref{sl2 wakimoto}, we obtain a realization
\begin{align}\label{wakimoto map}
\mu:=\mu_{\mathrm{FMS}}\circ \mu_{\mathrm{Wak}}\colon V^k(\sll_2)\hookrightarrow \Pi[0]\otimes _{\sqrt{p}A_1}    
\end{align}
We note that $L_{\mathrm{sug}}(z)$ gives the following conformal weights:
\begin{align}\label{align: conformal weight of lattice vectors}
\Delta(e^{au+bv})=\frac{1}{2}(a^2-b^2)+
{\frac{1}{2}}
(a+b),\quad \Delta(e^{n\sqrt{p}\alpha})=p(n^2-n).
\end{align} 

The vertex algebra $\Pi[0]\otimes V_{\sqrt{p}A_1}$ has the following particular automorphism, which can be checked by direct calculation:
\begin{align}\label{isom of vertex algebra}
\begin{array}{ccl}
g\colon\Pi[0]\otimes V_{\sqrt{p}A_1} &\xrightarrow{\simeq} & \Pi[0]\otimes V_{\sqrt{p}A_1} \\
1\otimes e^{\sqrt{p}\alpha} &\mapsto &e^{-(u+v)}\otimes e^{\sqrt{p}\alpha}\\
e^{u+v} \otimes 1 & \mapsto & e^{u+v}\otimes 1\\
1\otimes \alpha &\mapsto & 1\otimes \alpha - \tfrac{1}{\sqrt{p}}(u+v)\otimes 1\\
u\otimes 1 &\mapsto & u\otimes 1 -\tfrac{1}{4p}(u+v)\otimes 1+1\otimes\tfrac{1}{2\sqrt{p}}\alpha\\
v\otimes 1 &\mapsto & v\otimes 1 +\tfrac{1}{4p}(u+v)\otimes 1-1\otimes\tfrac{1}{2\sqrt{p}}\alpha.
\end{array}
\end{align}

The inverse Hamiltonian reduction \cite{Ad} states the following.
\begin{proposition}[\cite{Ad}]\label{eq: inverse hamiltonian diagram}\hspace{0mm}\\
\textup{(1)} There is an embedding of vertex algebras
\begin{align*}
\begin{array}{ccl}
\Phi\colon V^k(\sll_2) &\hookrightarrow & \Pi[0]\otimes L(c_{1,p},0) \\
e &\mapsto & e^{u+v} \\
h &\mapsto &-2\mathbf{v}\\
f &\mapsto & \left(\frac{1}{p}\omega_{1,p}-\mathbf{u}^2-(k+1)\partial \mathbf{u} \right)e^{-(u+v)}
\end{array}
\end{align*}
where 
$$\mathbf{u}=u-\frac{1}{4p}(u+v),\quad \mathbf{v}=v-\frac{1}{4p}(u+v).$$
Moreover, $\Phi$ sends
$$L_{\mathrm{sug}} \mapsto\frac{1}{2}(u^2-v^2)+\frac{1}{4p}\partial(u+v)-\partial u+\omega_{1,p}.$$
\textup{(2)} The embedding $\Phi$ makes the following diagram commute:
\begin{align*}
\xymatrix@=18pt{
\beta\gamma\otimes \pi^{\alpha} \ar[rr]_-{\mu_{\mathrm{FMS}}}& & \Pi[0]\otimes V_{\sqrt{p}A_1}\ar[rr]_{g}^{\simeq} && \Pi[0]\otimes V_{\sqrt{p}A_1} \\
V^k(\sll_2)\ar[u]^-{\mu_{\mathrm{Wak}}} \ar[rrrr]^{\Phi} &&&& \Pi[0]\otimes L(c_{1,p},0). \ar[u]_-{\mathrm{id}\otimes\mu_{\mathrm{Wak}}}
}
\end{align*}
\end{proposition}
Note that $g$ sends $Q_+$ and $Q_{\mathrm{FMS}}$ to
\begin{align}\label{modifying the screenings}
\mathcal{Q}_+=\int Y\left(\mathbf{1}\otimes e^{\sqrt{p}\alpha},z\right)dz,\quad \mathcal{Q}:=\int Y\left(e^\mathbf{u}\otimes e^{\frac{1}{2\sqrt{p}}\alpha},z\right)dz,
\end{align}
respectively.
\begin{theorem}\label{identificatio of affine FT}
We have an isomorphism of vertex algebras
\begin{align*}
    \mathrm{FT}_p(\sll_2)\simeq \mathrm{Ker}\ Q_-\mid_{\beta\gamma\otimes V_{\sqrt{p}A_1}},\quad Q_-:=\int Y\left(e^{-\frac{1}{p}(u+v)}\otimes e^{-\tfrac{1}{\sqrt{p}}\alpha},z\right)dz.
\end{align*}
\end{theorem}
\begin{proof}
Let us write the $\mathrm{B}$-action on $\Pi[0]\otimes V_{\sqrt{p}A_1}$ by \eqref{f-action} (resp.\ \eqref{free field realization of vir}) as $\mathrm{B}^{\mathrm{aff}}$ (resp.\ $\mathrm{B}^{\mathrm{vir}}$). Then the automorphism $g^{-1}$ gives an isomorphism of $\mathrm{B}$-modules
\begin{align*}
    g^{-1}\colon (\mathrm{B}^{\mathrm{aff}},\Pi[0]\otimes V_{\sqrt{p}A_1})\xrightarrow{\simeq} (\mathrm{B}^{\mathrm{vir}},\Pi[0]\otimes V_{\sqrt{p}A_1}).
\end{align*}
Hence we have 
\begin{align}\label{computation by automorphisms}
\begin{split}
    H^n(\mathrm{SL}_2&\times_{\mathrm{B}^{\mathrm{aff}}}(\Pi[0]\otimes V_{\sqrt{p}A_1})) \xrightarrow[g^{-1}]{\simeq} H^n(\mathrm{SL}_2\times_{\mathrm{B}^{\mathrm{vir}}}(\Pi[0]\otimes V_{\sqrt{p}A_1}))\\
    &\simeq \delta_{n,0}\Ker\ \mathcal{Q}_-|_{\Pi[0]\otimes V_{\sqrt{p}A_1}}
    \xrightarrow[g]{\simeq} \delta_{n,0}\Ker\ Q_-|_{\Pi[0]\otimes V_{\sqrt{p}A_1}}
\end{split}
\end{align}
by \eqref{triplet by short screening} and \eqref{modifying the screenings}.   
Since \eqref{triplet by short screening} is induced by the restriction to the fiber at the identity, the left exactness of the global section functor $H^0(\mathbb{P}^1,\text{-})$ implies
\begin{align*}
    \FT(\sll_2)
    &= H^0(\mathrm{SL}_2\times_{\mathrm{B}^{\mathrm{aff}}}(\beta\gamma\otimes V_{\sqrt{p}A_1}))\\
    &\simeq \Ker\ Q_{\mathrm{FMS}}|_{H^0(\mathrm{SL}_2\times_{\mathrm{B}^{\mathrm{aff}}}(\Pi[0]\otimes V_{\sqrt{p}A_1}))}\\
    &\simeq \Ker\ Q_{\mathrm{FMS}} \cap \Ker\ {Q}_-|_{\Pi[0]\otimes V_{\sqrt{p}A_1}}
    \simeq \Ker {Q}_-|_{\beta\gamma\otimes V_{\sqrt{p}A_1}}.
\end{align*}
This completes the proof.
\end{proof}
In \cite{ACGY}, the abelian intertwining algebra $\mathcal{V}^{(p)}$ is introduced as 
$$\mathcal{V}^{(p)}:=\mathrm{Ker}\ Q_-\mid_{\beta\gamma\otimes V_{\sqrt{p}A_1^*}}$$
where $A_1^*=\frac{1}{2}A_1$ is the weight lattice. Since $\Z_2=A_1^*/A_1$ acts as automorphisms of this algebra, the fixed point subalgebra is the vertex algebra
$$(\mathcal{V}^{(p)})^{\Z_2}\simeq \mathrm{Ker}\ Q_-\mid_{\beta\gamma\otimes V_{\sqrt{p}A_1}}$$
originally introduced by Semikhatov and Tipnin \cite{ST1}.
The above theorem immediately implies the following.
\begin{corollary} 
We have an isomorphism of vertex algebras
$$\mathrm{FT}_p(\sll_2)\simeq (\mathcal{V}^{(p)})^{\Z_2}.$$
\end{corollary}
\section{Triplet algebra $\W(p)$}\label{section: triplet}
We review some results of $\W(p)$-modules, following \cite{AM,CRW,FT,Sug}.
\subsection{$(1,p)$-model}\label{subsection: (1,p)-model}
The $C_1$-cofinite simple $L(c_{1,p},0)$-modules are parametrized by the following highest weight modules 
\begin{align}\label{parametrization of h.wt.}
L_{r,s}:=L(c_{1,p},h_{r,s}),\quad h_{r,s}:=\frac{(r-sp)^2-(1-p)^2}{4p},\quad (1\leq r \leq p,\ s\geq1).
\end{align}
As $L_{r,s}=L_{-r,-s},\ L_{r+p,s+1}$, it is also useful to use $L_{r,s}$ ($r,s\in \Z$).
By using $L_{r,s}$'s as a building block, we consider the Fock modules over $\pi^\alpha$, which are $L(c_{1,p},0)$-modules by \eqref{free field realization of vir}. Following \cite{CRW}, let us introduce
\begin{align}\label{parametrization}
\alpha_{r,s}:=-(s-1)\sqrt{p}\varpi+\frac{r-1}{\sqrt{p}}\varpi,\quad (1\leq r\leq p,\quad s\in \Z),
\end{align}
which represent $\alpha_{r,s}$'s for all $r,s\in \Z$ since $\alpha_{r+p,s}=\alpha_{r,s-1}$.
As easily seen from \eqref{vir wakimoto}, $e^{\alpha_{r,s}}$ is a highest weight vector of $L(c_{1,p},0)$ of highest weight $h_{r,s}$, which indeed generates the submodule $L_{r,s}\subset \pi^\alpha_{\alpha_{r,s}}$.
By \cite{TK}, for $r\neq p$, the integrals
\begin{align}\label{screening operators}
\begin{split}
&\mathcal{Q}_-^{[r]}:=\int_{[\Gamma_r]} \prod_{a=1}^r Y(e^{-\frac{1}{\sqrt{p}}\alpha},z_a)dz\colon \pi^{\alpha}_{\alpha_{r,s}}\rightarrow \pi^{\alpha}_{\alpha_{-r,s}}(=\pi^{\alpha}_{\alpha_{p-r,s+1}}),
\end{split}
\end{align}
over certain cycles $[\Gamma_r]$ give $L(c_{1,p},0)$-module homomorphisms, which generalize $\mathcal{Q}_-$ in \eqref{triplet by short screening} for $r=1$. It is convenient to set $\mathcal{Q}_-^{[r]}=0$ for $r=p$.
Then we have the short exact sequence of $L(c_{1,p},0)$-modules associated with $\mathcal{Q}_-^{[r]}$
\begin{align}\label{canonical exact seq}
0\rightarrow \Ker \mathcal{Q}_-^{[r]}|_{\pi^{\alpha}_{\alpha_{r,s}}} \rightarrow \pi^{\alpha}_{\alpha_{r,s}}\rightarrow \mathrm{Coker} \mathcal{Q}_-^{[r]}|_{\pi^{\alpha}_{\alpha_{r,s}}} \rightarrow 0.
\end{align}
Introduce the $L(c_{1,p},0)$-modules
$$\mathcal{M}_{r,s}:=\bigoplus_{n\geq0}L_{r,s+2n},\quad \mathcal{M}_{p-r',s'+1}:=\bigoplus_{n\geq0}L_{r',-s'-2n}$$
for $1\leq r,r'\leq p$ and $s\geq 1$, $s'\geq0$. Note that the notation is consistent, see \eqref{parametrization of h.wt.}.
\begin{proposition}[e.g. \cite{AM,CRW}]\label{description as virasoro}
Set $1\leq r,r' \leq p$ with $r'\neq p$. Then
\begin{align*}
\Ker \mathcal{Q}_-^{[r]}|_{\pi^{\alpha}_{\alpha_{r,s}}}\simeq
\begin{cases}
\mathcal{M}_{r,s} &\hspace{-1mm}(s\geq 1) \\
\mathcal{M}_{r,2-s} &\hspace{-1mm} (s\leq 1),
\end{cases}
\
\mathrm{Coker} \mathcal{Q}_-^{[r']}|_{\pi^{\alpha}_{\alpha_{r',s}}}\simeq
\begin{cases}
\mathcal{M}_{p-r',1+s}&\hspace{-1mm} (s\geq 0) \\
\mathcal{M}_{p-r',1-s}&\hspace{-1mm} (s\leq 0),
\end{cases}
\end{align*}
as $L(c_{1,p},0)$-modules. In particular,
\begin{align*}
\mathrm{Coker} \mathcal{Q}_-^{[r]}|_{\pi^{\alpha}_{\alpha_{r,s}}}\simeq \Ker \mathcal{Q}_-^{[p-r]}|_{\pi^{\alpha}_{\alpha_{p-r,s+1}}}.
\end{align*}
\end{proposition}
The singlet algebra is, by definition, the vertex algebra on the kernel
$$\mathcal{M}(p):= \mathrm{Ker}\mathcal{Q}_-|_{\pi^\alpha},\quad (p\geq 2).$$
The positively graded simple $\mathcal{M}(p)$-modules are classified in \cite{A2} and, in particular, 
$\mathcal{M}_{r,s}:=\Ker \mathcal{Q}_-^{[r]}|_{\pi^{\alpha}_{\alpha_{r,s}}}$ is a simple $\mathcal{M}(p)$-module and \eqref{canonical exact seq} is indeed a socle decomposition of $\pi^\alpha_{r,s}$ as a $\mathcal{M}(p)$-module:
\begin{align}\label{singlet decompsoition}
    0\rightarrow \mathcal{M}_{r,s} \rightarrow \pi^{\alpha}_{\alpha_{r,s}}\rightarrow \mathcal{M}_{p-r,s+1} \rightarrow 0.
\end{align}

\subsection{Simple triplet modules}\label{subsection: simple W(p)-modules}
We parameterize the simple $V_{\sqrt{p}A_1}$-modules by 
$$\mathcal{V}_{r,s}=V_{\sqrt{p}A_1+\alpha_{r,s}}=\bigoplus_{s'\in s+2\Z}\pi^{\alpha}_{\alpha_{r,s'}},\quad ( 1\leq r \leq p,\ 1\leq s\leq 2)$$
according to \eqref{parametrization}.
Introduce the $L(c_{1,p},0)$-submodules 
\begin{align}\label{simple triplet modules}
\W_{r,s}:=\mathrm{Ker}\left(\mathcal{Q}_-^{[r]}\colon \mathcal{V}_{r,s}\rightarrow \mathcal{V}_{p-r,3-s} \right),\ (1\leq r < p),\ \W_{r,s}:=\mathcal{V}_{p,s},\ (r=p).
\end{align}
\begin{theorem}[\cite{AM}] 
The $L(c_{1,p},0)$-modules $\W_{r,s}$, $(1\leq r\leq p,\ s=1,2)$ are $\W(p)$-modules and form a complete set of all the inequivalent simple $\W(p)$-modules.
\end{theorem}

It follows from Proposition \ref{description as virasoro} that 
\begin{align}\label{decomposition of triplet}
\W_{r,s}\simeq \bigoplus_{n\geq 0}\C^{2n+s}\otimes L_{r,2n+s}
\end{align}
as $L(c_{1,p},0)$-modules.
The basis of the multiplicity space $\C^{2n+s}$ is given by the highest weight vectors
\begin{align*}
v_{r,2n+s;a}:=\mathcal{Q}_+^ae^{\alpha_{r,2n+s}},\quad (a=0,\cdots,2n+s-1).
\end{align*}
It follows that $\C^{2n+s}$ is a $\mathrm{B}$-module and, moreover, a simple $\mathrm{SL}_2$-module. Hence $\W_{r,s}\subset \V_{r,s}$ is an $(\mathrm{SL}_2,L(c_{1,p},0))$-submodule.
For $r\neq p$, \eqref{canonical exact seq} implies a short exact sequence of $\W(p)$-modules
\begin{align}\label{virasoro felder complex}
0 \rightarrow \W_{r,s}\rightarrow \mathcal{V}_{r,s} \rightarrow \W_{p-r,3-s}({-\varpi})\rightarrow 0
\end{align}
called the Felder complex. Here $\W_{p-r,3-s}({-\varpi})$ is the tensor product $\W_{p-r,3-s}\otimes \C_{-\varpi}$ twisting the $\mathrm{H}$-action by $-\varpi$, which makes  \eqref{virasoro felder complex} into a complex of $\mathrm{B}$-modules.
Now, we can realize the simple $\W(p)$-modules $\W_{r,s}$ on the cohomologies of the equivariant vector bundles $\mathrm{SL}_2\times_{\mathrm{B}}\V_{r,s}$ where the $\mathrm{B}$-action is defined by the formulas
\begin{align}\label{B-action for vir}
f\mapsto \mathcal{Q}_+,\quad h|_{\pi_{\alpha_{r,s'}}^{\alpha}}=-\frac{1}{\sqrt{p}}h(\alpha_{1,s'}),\quad (s'\in s+2\Z),
\end{align}
see \S \ref{Sec: construction} for the case $\V_{1,1}=V_{\sqrt{p}A_1}$.
\begin{theorem}[\cite{FT,Sug}]\label{main result for the virasoro triplet algebra}
We have the cohomology vanishing
\begin{align}\label{cohomology vanishing}
H^n(\mathrm{SL}_2\times_\mathrm{B} \mathcal{V}_{r,s})\simeq \delta_{n,0}\W_{r,s}.
\end{align}
Moreover, $\W_{r,s}$ are $\W(p)$-modules characterized as the maximal $\mathrm{B}$-submodule of $\mathcal{V}_{r,s}$ which is an $\mathrm{SL}_2$-module.
\end{theorem}
\begin{proof}
We reproduce the proof in \cite{FT,Sug} as we will use the same argument later.
Note that given an $\mathrm{SL}_2$-module $M$ and $\mu\in \Z\varpi$, the $\mathrm{B}$-module $M(\mu)$ satisfies
\begin{align}\label{fact:Sug-114}
H^n(\mathrm{SL}_2\times_\mathrm{B}M(\mu))\simeq H^n(\underline{\C}_\mu)\otimes M
\end{align}
where $\underline{\C}_\mu$ is the line bundle $\underline{\C}_\mu=\mathrm{SL}_2\times_\mathrm{B}\C(\mu)$
and that 
\begin{align}\label{coh for line bundles}
H^n(\underline{\C}_0)\simeq \delta_{n,0}\C,\quad H^n(\underline{\C}_{-\varpi})= 0.    
\end{align}
As $\W(p)=H^0(\mathrm{SL}_2\times_\mathrm{B} V_{\sqrt{p}A_1})$ acts on $H^n(\mathrm{SL}_2\times_\mathrm{B} \mathcal{V}_{r,s})$, the second assertion follows from \eqref{cohomology vanishing} and \eqref{fact:Sug-114}. Hence it suffices to show \eqref{cohomology vanishing}. 
The case $r=p$ is immediate from \eqref{fact:Sug-114} since $\W_{p,s}=\V_{p,s}$ is an $\mathrm{SL}_2$-module by \eqref{decomposition of triplet}.
The case $r\neq p$ follows from the long exact sequence for $H^\bullet(\mathbb{P}^1,?)$ applied to the Felder complex \eqref{virasoro felder complex}:
\begin{align*}
0
&\rightarrow H^0(\underline{\C}_0)\otimes \W_{r,s} \rightarrow H^0(\mathrm{SL}_2\times_\mathrm{B} \mathcal{V}_{r,s})\rightarrow H^0(\underline{\C}_{-\varpi})\otimes \W_{p-r,3-s}\\
&\rightarrow H^1(\underline{\C}_0)\otimes \W_{r,s} \rightarrow H^1(\mathrm{SL}_2\times_\mathrm{B} \mathcal{V}_{r,s})\rightarrow H^1(\underline{\C}_{-\varpi})\otimes \W_{p-r,3-s}\rightarrow 0.
\end{align*}
We obtain the assertion by using \eqref{coh for line bundles}.
\end{proof}

\section{Weight $V^{k}(\sll_2)$-modules}\label{section: weight V-modules}
Here, following \cite[\S 7]{Ad}, we consider weight $V^k(\sll_2)$-modules which play the role of the $L(c_{1,p},0)$-modules $L_{r,s}$ for the triplet algebra $\W(p)$.
\subsection{Weight modules}
Introduce the $\Pi[0]$-modules
\begin{align}\label{Pi[0]-modules}
\Pi_a[b]:=\bigoplus_{b'\in [b]} \pi^{u,v}_{a\mathbf{v}+b'(u+v)},\quad (a\in \Z,\ [b]=b+\Z \subset\C).
\end{align}
The parameters $a\in\Z$ correspond to the spectral flow twists \cite{Li}, i.e. the $\Pi[0]$-modules 
$\mathcal{S}_{x}\Pi_0[b]:=(\Pi_0[b],Y_{x}(\text{-},z))$ for $x=a \mathbf{v}$ where $Y_x(A, z)=Y(\Delta(x,z)A, z)$ with
$\Delta(x,z):=z^{x_{0}}E^-(x,-z)$ (see \eqref{def of Epm} below for notation). Then we have
\begin{align}\label{twists of free field}
  \Pi_a[b]\simeq \mathcal{S}_{a\mathbf{v}}\Pi_0[b] 
\end{align}
as $\Pi[0]$-modules.
Now, we consider the weight $V^k(\sll_2)$-module realized, via the embedding $\Phi$ in Proposition \ref{eq: inverse hamiltonian diagram}, on the $\Pi[0]\otimes L(c_{1,p},0)$-modules
\begin{align}\label{definition of E}
E_{r,s}^{a,[b]}=\Pi_a[b]\otimes L_{r,s},
\end{align}
see \eqref{parametrization of h.wt.} and \eqref{Pi[0]-modules}.
They are positively graded iff $a=1$ since the lattice vectors
\begin{align*}
\theta_{a,b'}:=e^{a\mathbf{v}+b'(u+v)}\otimes v_{r,s}\in E_{r,s}^{a,[b]}
\end{align*}
have the conformal weights $(1-a)b'+\frac{1}{4}ka^2+h_{r,s}$. For $V^k(\sll_2)$-modules $M$, we write the spectral flow twists as $\mathcal{S}^aM=\mathcal{S}_{a \varpi}M$. Since
\begin{align}\label{E=SE}
E^{a,[b]}_{r,s}\simeq \mathcal{S}^{-a}E^{0,[b]}_{r,s}  
\end{align}
as $V^k(\sll_2)$-modules by \eqref{twists of free field}, we restrict ourselves to the case $a=1$. 
To understand the module structure of $E^{1,[b]}_{r,s}$, we begin with the lowest conformal weight subspace 
\begin{align*}
\Omega_{r,s}^{1,[b]}:=\bigoplus_{b'\in [b]} \C \theta_{1,b'}
\subseteq E_{r,s}^{1,[b]}
\end{align*}
where $\sll_2$ acts by the 0-th modes of the fields, or explicitly by
\begin{align}\label{description of sl2 action}
\begin{split}
&e_0\ \theta_{1,b'}=\theta_{1,b'+1},\\
&h_0\ \theta_{1,b'}=(2b'-k)\theta_{1,b'},\\
& f_0\ \theta_{1,b'}=-(b'+a_{r,s})(b'+a_{-r,-s})\theta_{1,b'-1}
\end{split}
\end{align}
and
\begin{align*}
a_{r,s}:=\left(\alpha_{r,s},\frac{1}{\sqrt{p}}\varpi\right)=\frac{(1-s)p+(r-1)}{2p}.
\end{align*}

To describe the $\sll_2$-module structure $\Omega_{r,s}^{1,[b]}$, the following convention is useful: given a module $M$ with composition series $0=M_{n+1}\subset \cdots \subset M_1\subset M_0=M$ with $M_i/M_i+1$ simple, we draw 
$$M_n\dashrightarrow M_{n-1}/M_n \dashrightarrow\cdots \dashrightarrow M_1/M_2 \dashrightarrow M_0/M_1.$$
Let $L^\pm(a)=L^\pm(a\varpi)$ denote the highest (lowest) weight simple $\sll_2$-module of highest (lowest) weight $a \varpi$. It is straightforward to show the following by using \eqref{description of sl2 action}.

\begin{lemma}\label{lem: decomposition}
The $\sll_2$-module $\Omega_{r,s}^{1,[b]}$ decomposes as follows:
\hspace{0mm}\\
\textup{(1)} $(1\leq r <p)$
\begin{enumerate}
\renewcommand{\labelenumii}{(\arabic{enumii})}
\item[(i)] $[b]\neq[-a_{\pm r,\pm s}]\colon$ $\Omega_{r,s}^{1,[b]}$ is simple.
\item[(ii)] $[b]=[-a_{\pm r,\pm s}]\colon$
$L^-(-2a_{\pm r,\pm s}-k) \dashrightarrow L^+(-2a_{\pm r,\pm s}-2-k)$.
\end{enumerate}
\textup{(2)} $(r=p)$
\begin{enumerate}
\item[(i)] $[b]\neq[-a_{p,s}]\colon$ $\Omega_{r,s}^{1,[b]}$ is simple.
\item[(ii)] $[b]=[-a_{p,s}],\ s=1\colon$ $L^-(1)\dashrightarrow L^+(-1)$.
\item[(iii)] $[b]=[-a_{p,s}],\ s\neq 1\colon$ $L^-(s) \dashrightarrow L^+(s-2) \dashrightarrow L^+(-s)$.
\end{enumerate}
\end{lemma}

\subsection{Some Lemmas} We denote by $U(\widehat{\sll}_{2,k})$ the enveloping algebra of the affine Lie algebra $\widehat{\sll}_2$ at level $k$.
\begin{lemma}\label{lem: position of relaxed highest weight vector}
Non-zero $V^k(\sll_2)$-submodules $M\subset E^{1,[b]}_{r,s}$ satisfy $M\cap\Omega_{r,s}^{1,[b]}\neq 0$.
\end{lemma}

\proof
It suffices to show $\theta_{1,b'}\in M$ for some $b'$.
By using the fields $\zeta(z):=\frac{1}{2}(u+v)(z), \xi(z):=\frac{1}{2}(u-v)(z)$,
we have
\begin{align}\label{e and h by zeta and xi}
&e(z)=T_{2\zeta}z^{-1} E^+(2\zeta,z)E^-(2\zeta,z),\quad h(z)=-2v(z)+ \frac{1}{p}\zeta(z),
\end{align}
by Proposition \ref{eq: inverse hamiltonian diagram} where we set $T_{2\zeta}$ to be the shift operator and 
\begin{align}\label{def of Epm}
&E^{\pm}(t,z)
=\mathrm{exp}\left(-\sum_{\pm m<0}\frac{t_{m}}{m}z^{-m} \right),\quad \sum_{n\geq0}s_n(t)z^n=E^{+}(t,z).
\end{align}
Since $h_0$ acts on $\pi^{u,v}_{\mathbf{v}+b'(u+v)}\otimes L_{r,s}$ by 
distinct scalars $2b'-4k$, it follows that
$M \cap \pi^{u,v}_{\mathbf{v}+b'(u+v)}\otimes L_{r,s}\neq 0$ for some $b'\in [b]$.
Hence, we may take a nonzero element 
\begin{align}\label{nonzero element}
w=F(\zeta,\xi)e^{\mathbf{v}+b'(u+v)}\otimes \overline{w} \in M
\end{align}
for some differential polynomial $F(\zeta,\xi)$. 
We may assume $F(\zeta,\xi)=F(\zeta,0)$. 
Indeed, in the case $F(\zeta,\xi)\neq F(\zeta,0)$, we can replace $w$ with another element of the same form \eqref{nonzero element} with $F(\zeta,\xi)=F(\zeta,0)$ by applying $e_n$'s with $n>0$ to $w$ as follows.
By the commutation relations
\begin{align}\label{commutation relation}
[\zeta_{m}, \zeta_{n}]=[\xi_{m}, \xi_{n}]=0,\quad
[\zeta_{m}, \xi_{n}]=\frac{1}{2}m\delta_{m+n,0}
\end{align}
and the Taylor series expansion $g(t+a)=\mathrm{exp}(a\partial_t)g(t)$, we have
\begin{align}\label{applying e(z) to w}
e(z)w
=&z^{-1}E^+(2\zeta,z)F(\zeta,\xi-z^\bullet)e^{\mathbf{v}+(b'+1)(u+v)}\otimes\overline{w}
\end{align}
where $F(\zeta,\xi-z^\bullet)=F(\zeta,\xi)|_{\xi_{m}\rightarrow\xi_{m}-z^{m}}$.
Notice that $F(\zeta,\xi-z^\bullet)=F(\zeta,\xi)$ iff $F(\zeta,\xi)=F(\zeta,0)$ since we can inductively show that if $F(\zeta,\xi-z^\bullet)=F(\zeta,\xi)$ then the coefficient of $z^{-n}$ in $F(\zeta,\xi-z^\bullet)$ agrees with $-\partial_{\xi_{-n}}F(\zeta,\xi)$ starting with $n=1$.

Now, if $F(\zeta,\xi)\neq F(\zeta,0)$, then we have $F(\zeta,\xi-z^\bullet)\neq F(\zeta,\xi)$ and thus $F(\zeta,\xi-z^\bullet)=F^{(1)}(\zeta,\xi)z^{-N}+(\text{higher order terms})$ for some $N>0$. Then we replace $w$ with 
\begin{align}\label{process}
  w^{(1)}:=e_{N}w=F^{(1)}(\zeta,\xi)e^{\mathbf{v}+(b'+1)(u+v)}\otimes\overline{w}.  
\end{align}
Since the degree of $F^{(1)}(\zeta,\xi)$ is strictly lower than that of $F(\zeta,\xi)$, after repeating this procedure for finitely many times, we reach $w^{(n)}=F^{(n)}(\zeta,\xi)e^{\mathbf{v}+(b'+n)(u+v)}\otimes\overline{w}$ such that $F^{(n)}(\zeta,\xi)= F^{(n)}(\zeta,0)$ as desired.

It follows from $F(\zeta,\xi)=F(\zeta,0)$ that $h_nw=\partial_{\zeta_{-n}}w$ ($n>0$) by \eqref{commutation relation}.
Hence, by applying $h_n$'s to $W$, we obtain the element $w':=e^{\mathbf{v}+b'(u+v)}\otimes \overline{w}\in M$. 
Since $L_{r,s}$ is simple, by applying suitable $L_{\mathrm{sug},n}$'s to $w'$, we obtain $\theta_{1,b'}=e^{\mathbf{v}+b'(u+v)}\otimes v_{r,s}\in M$ and thus the assertion.
\endproof

\begin{lemma}\label{lemm-UOmega=E}
$U(\widehat{\sll}_{2,k})\Omega_{r,s}^{1,[b]}=E_{r,s}^{1,[b]}$.
\end{lemma}
\begin{proof}
For $w\in L_{r,s}$, we can show by induction of the conformal weight that $M_w=\Pi_1[b]\otimes w$ is obtained from  $e^{\mathbf{v}+b'(u+v)}\otimes w$ by applying the fields in \eqref{e and h by zeta and xi}. 
Conversely, by using the replacement \eqref{process}, we obtain $e^{\mathbf{v}+b'(u+v)}\otimes w$ from any nonzero element of $M_w$ by using the fields in \eqref{e and h by zeta and xi}. 
Now, we can show by induction of the conformal weight on $L_{r,s}$ that starting with $w=v_{r,s}$, we apply $f(z)$ to the subspace $M_w$ and use the above argument to obtain the subspace $M_{L_nw}$ ($n<0$). Since $E_{r,s}^{1,[b]}=\sum_w M_{w}$, we obtain the assertion.   
\end{proof}

\subsection{Structure of weight modules}

Let us introduce the $V^k(\sll_2)$-module
\begin{align*}
\mathbb{V}^\pm_k(\lambda):=U(\widehat{\sll}_{2,k})\underset{U(\sll_2[\![t]\!])}{\otimes} L^\pm(\lambda)
\end{align*}
and denote by $L^\pm_k(\lambda)$ its unique simple quotient.

\begin{proposition}
\label{decomposition of weight modules}
The $V^k(\sll_2)$-module $E_{r,s}^{1,[b]}$ decomposes as follows:
\hspace{0mm}\\
\textup{(1)} $(1\leq r <p)$
\begin{enumerate}
\renewcommand{\labelenumii}{(\arabic{enumii})}
\item[(i)]\label{decomposition of weight modules 1-1}
$[b]\neq[-a_{\pm r,\pm s}]\colon$ $E_{r,s}^{1,[b]}$ is simple.
\item[(ii)]\label{decomposition of weight modules 1-2}
$[b]=[-a_{\pm r,\pm s}]\colon$
$L^-_k(-2a_{\pm r,\pm s}-k)\dashrightarrow L^+_k(-2a_{\pm r,\pm s}-2-k)$.
\end{enumerate}
\textup{(2)} $(r=p)$
\begin{enumerate}
\renewcommand{\labelenumii}{(\arabic{enumii})}
\item[(i)]\label{decomposition of weight modules 2-1}
$[b]\neq[-a_{p,s}]\colon$ 
$E_{p,s}^{1,[b]}$ is simple.
\item[(ii)]\label{decomposition of weight modules 2-2}
$[b]=[-a_{p,s}],\ s=1\colon$ 
$L^-_k(\varpi)\dashrightarrow L^+_k(\varpi)\dashrightarrow L^+_k(-\varpi)$.
\item[(iii)]\label{decomposition of weight modules 2-3}
$[b]=[-a_{p,s}],\ s\neq 1\colon$
$L^-_k(s\varpi)\dashrightarrow L^+_k((s-2)\varpi)\dashrightarrow 
L^+_k(s\varpi)\dashrightarrow 
L^+_k(-s\varpi)$.
\end{enumerate}
\end{proposition}

\proof
We postpone a part of the proof of \textup{(2)-(ii), (iii)} until \S \ref{section:FT-modules} and show the remaining cases.
We show\textup{(1)}-(i) and \textup{(2)}-(i). 
Take a nonzero $V^k(\sll_2)$-submodule $M\subset E_{r,s}^{1,[b]}$. Then the $\sll_2$-submodule $M\cap \Omega_{r,s}^{1,[b]}$ is nonzero by Lemma \ref{lem: position of relaxed highest weight vector} and thus $\Omega_{r,s}^{1,[b]}$ by Lemma \ref{lem: decomposition}. Hence $M=E_{r,s}^{1,[b]}$ holds by Lemma \ref{lemm-UOmega=E} and thus the assertion.

We show \textup{(1)}-(ii).
Set $\lambda=(-2a_{\pm r,\pm s}-k)\varpi$.
By Lemma \ref{lem: decomposition}, we have an $\sll_2$-submodule $L^-(\lambda)\subset \Omega_{r,s}^{1,[b]}$ and thus a $V^k(\sll_2)$-submodule $M:=U(\widehat{\sll}_{2,k})L^-(\lambda)\subset E_{r,s}^{1,[b]}$.
For a nonzero submodule $N\subset M$, $N\cap \Omega_{r,s}^{1,[b]}\subset L^-(\lambda)$ is nonzero by Lemma \ref{lem: position of relaxed highest weight vector} and coincides with $L^-(\lambda)$ by the simplicity. Hence, $N=M$ and thus $M$ is simple, i.e. $M=L^-_k(\lambda)$.

Set $L:=\{v\in E_{r,s}^{1,[b]} \mid f_0^{\gg0}v=0\}\subset E_{r,s}^{1,[b]}$ where $f_0^{\gg0}$ means $f_0^N$ for large enough $N$. It is a $V^k(\sll_2)$-submodule and contains $L_k^-(\lambda)$. 
We show $L=L_k^-(\lambda)$. Indeed, it follows from the structure of $\Omega_{r,s}^{1,[b]}$ (Lemma \ref{lem: decomposition}) and Lemma \ref{lemm-UOmega=E} that $e_0^{\gg0}v\in L_k^-(\lambda)$ for all $v\in L$.
Therefore, the quotient $L/L_k^-(\lambda)$ is an integrable $\sll_2$-submodule.
Since $-2a_{r,s}-k\notin \Z$, $\h$-weights of $E_{r,s}^{1,[b]}$ are not integral. Thus $L/L_k^-(\lambda)=0$ i.e. $L=L_k^-(\lambda)$. 

We show $E_{r,s}^{1,[b]}/L_k^-(\lambda) \simeq L^+_k(\lambda-\alpha)$ by using $L=L_k^-(\lambda)$. By Lemma \ref{lem: decomposition}, it suffices to show the simplicity. Take $\overline{v}\in E_{r,s}^{1,[b]}/L_k^-(\lambda)$ and a lift $v\in E_{r,s}^{1,[b]}$. Since $f_0^Nv\neq0$ for all $N$, we may take a large enough $N$ so that  $U(\widehat{\sll}_{2,k})v \cap \Omega_{r,s}^{1,[b]}$ intersects non-trivially with  $L^+(\lambda-2)\subset \Omega_{r,s}^{1,[b]}$ by repeating the process in the proof of Lemma \ref{lem: position of relaxed highest weight vector}.
Hence, $U(\widehat{\sll}_{2,k})v= E_{r,s}^{1,[b]}$ by Lemma \ref{lemm-UOmega=E} and thus $E_{r,s}^{1,[b]}/L_k^-(\lambda)$ is simple. 
This completes the proof of \textup{(1)}-(ii). 

Finally we consider the cases \textup{(2)}-(ii), (iii).
For $s\geq 1$, set
\begin{align}\label{socle sequence of E}
\begin{split}
&N_0^s:=E^{1,[a_{p,s}]}_{p,s},~
N_1^s:=\{v\in N_0~|~f_0^{\gg0}v=0\},\\
&N_2^s:=U(V^k(\sll_2))\theta_{1,-s+2},~
N_3^s:=U(V^k(\sll_2))\theta_{1,s}.
\end{split}
\end{align}
Proofs of $N_3^s\simeq L^-_k(s\varpi)$ and $N_0^s/N_1^s\simeq L^+_k(-s\varpi)$ are similar to that of the case \textup{(1)}-(ii).
By $L^+(s-2)\hookrightarrow(N_2^s/N_3^s)^{U(\sll_2[\![t]\!]t)}$
and Frobenius reciprocity, we obtain a surjective $V^k(\sll_2)$-module homomorphism $\mathbb{V}^+_{k}((s-2)\varpi)\rightarrow N_2^s/N_3^s$.
The proof of remaining assertions is postponed to the last of \S\ref{section:FT-modules} below.
\endproof

\section{Free field $V^k(\sll_2)$-modules}\label{section:FT-modules}
Here, we consider the structure of free field  $V^k(\sll_2)$-modules which we will use to construct simple $\FT(\sll_2)$-modules in the next section.

\subsection{Free field modules}\label{subsection: decomposition through Phi}
Let us introduce free field representations 
\begin{align*}
    \Pi_a[b]\otimes \V_{r,s},\quad (a\in \Z, [b]\in \C/\Z,\ 1\leq r\leq p, s=1,2)
\end{align*}
over $V^k(\sll_2)$ through the Wakimoto realization $\mu$ in \eqref{wakimoto map}. 
The assignment
\begin{align*}
&f\mapsto Q_+=\int Y(e^{u+v}\otimes e^{\sqrt{p}\alpha},z)dz,\quad 
h|_{\Pi_a[b]\otimes \pi_{\alpha_{r,s'}}^{\alpha}}=\left\lceil-\frac{1}{\sqrt{p}}h\left(\alpha_{r,s'}-\frac{a}{\sqrt{p}}\varpi\right)\right\rceil,
\end{align*}
where $\lceil t \rceil$ is the maximal integer $t_0\leq t$, induces a $\mathrm{B}$-action, which we denote by $\mathrm{B}^\af$. Then $\Pi_a[b]\otimes \V_{r,s}$ is a $(\mathrm{B}^\af, V^k(\sll_2))$-bimodule.

The automorphism $g$ in \eqref{isom of vertex algebra} induces an isomorphism of $\Pi[0]\otimes V_{\sqrt{p}\mathrm{A}_1}$-modules
\begin{align}\label{pull-back of PiV}
g \colon 
\Pi_a[b]\otimes \mathcal{V}_{r,s} \simeq g^*(\Pi_a[\tfrac{a}{4p}+b-a_{r,s}]\otimes \mathcal{V}_{r-a,s}).
\end{align}
By Proposition \ref{eq: inverse hamiltonian diagram}, it intertwines the $\mathrm{B}^{\af}$-action given by $\mathrm{B}^{\mathrm{vir}}$ and $\mathrm{B}^{\af}$ respectively, and the $V^k(\sll_2)$-action through $\Phi$ and $\mu$ respectively.
Let us first consider the structure of the $(\mathrm{B}^{\mathrm{vir}}, V^k(\sll_2))$-bimodule $\Pi_a[\tfrac{a}{4p}+b-a_{r,s}]\otimes \mathcal{V}_{r-a,s}$.
Since the spectral flow twist $\mathcal{S}^a=\mathcal{S}_{-a\mathbf{v}}$ acts on the first factor and 
\begin{align*}
\Pi_0[b'-a_{r,s}]\otimes \W_{r,s}\simeq\mathcal{S}^{a}(\Pi_a[b'-a_{r,s}]\otimes \W_{r,s}),
\end{align*}
as $(\mathrm{SL}_2^{\mathrm{vir}},V^k(\sll_2))$-bimodules, we may restrict to the case $a=0$ and write $\Pi[b]=\Pi_0[b]$ to simplify the notation.

Let us consider the case $r\neq p$ and take a tensor product of $\Pi[b-a_{r,s}]$ with the Felder complex \eqref{virasoro felder complex}. 
Then we obtain a short exact sequence of $(\mathrm{B}^{\mathrm{vir}},V^k(\sll_2))$-bimodules
\begin{align}\label{vir-aff Felder complex}
\begin{split}
0\rightarrow \Pi[b-a_{r,s}]\otimes \W_{r,s}&\rightarrow \Pi[b-a_{r,s}]\otimes \mathcal{V}_{r,s}\\
&\rightarrow \Pi[b-a_{r,s}]\otimes \W_{p-r,3-s}(-\varpi)\rightarrow 0.
\end{split}
\end{align}
We apply Proposition \ref{decomposition of weight modules} to the first and third term 
by using 
\begin{align*}
[-a_{r,2n+s}]=[-a_{r,s}],\quad 
[-a_{-r,-(2n+s)}]=[-a_{r,s}+\tfrac{r}{p}]
\end{align*}
and the isomorphism of $V^k(\sll_2)$-modules
\begin{align}\label{align L=SL}
\mathcal{S}L^-_k(\lambda)\simeq L^+_k(\lambda+k\varpi)\quad(\lambda\in\mathfrak{h}^{\ast}).
\end{align}
Then it is straightforward to show that $\Pi_0[b-a_{r,s}]\otimes \W_{r,s}$ decomposes into
\begin{align}\label{easy decomposition I}
\Pi[b-a_{r,s}]\otimes \W_{r,s} \simeq \displaystyle \bigoplus_{n\geq0}\C^{2n+s}\otimes E^{0,[b-a_{r,s}]}_{r,2n+s}
\end{align}
as $(\mathrm{SL}_2^{\mathrm{vir}},V^k(\sll_2))$-bimodules and, moreover, 

\begin{align}\label{easy decomposition II}
\begin{cases}
    [b]=[0]\colon & 0\rightarrow \mathcal{X}_{r,s}^{\mathrm{vir},+} \rightarrow \Pi[b-a_{r,s}]\otimes \W_{r,s}
\rightarrow  \mathcal{S}\mathcal{X}_{r+1,s}^{\mathrm{vir},+} \rightarrow 0,\\ 
&\\
    [b]=[\tfrac{r}{p}]\colon & 0\rightarrow \mathcal{X}_{r,s}^{\mathrm{vir},-}\rightarrow \Pi[b-a_{r,s}]\otimes \W_{r,s}
\ \rightarrow \mathcal{S}\mathcal{X}_{r-1,s}^{\mathrm{vir},-}\rightarrow 0,
\end{cases}
\end{align}
where we set the $(\mathrm{SL}_2^{\mathrm{vir}},V^k(\sll_2))$-bimodules
\begin{align}\label{decomposition data}
\mathcal{X}_{r,s}^{\mathrm{vir},\pm}=\bigoplus_{n\geq0}\C^{2n+s}\otimes L_k^+(\lambda_{\pm r,\pm(2n+s)}) 
\end{align}
with
\begin{align*}
\lambda_{r,s}:=-\frac{1}{\sqrt{p}}\alpha_{r,s}=\left(s-1-\frac{r-1}{p}\right)\varpi\quad(r,s\in\Z).
\end{align*}

Now, we return back to the $(\mathrm{B}^{\af},V^k(\sll_2))$-bimodule $\Pi[b]\otimes \V_{r,s}$ through \eqref{pull-back of PiV}. We introduce the following $(\mathrm{SL}_2^{\af},V^k(\sll_2))$-bimodules
\begin{align*}
\begin{split}
\afWb{r}{s}{b}&:=g^\ast\left(\Pi[b-a_{r,s}]\otimes \W_{r,s}\right)\subseteq \Pi[b]\otimes\mathcal{V}_{r,s}\quad (b\in\C),\\
\afXpm{r}{s}&:=g^{\ast}\mathcal{X}^{\mathrm{vir},\pm}_{r,s}\subseteq \mathcal{W}_{r,s}^\pm:=\afWb{r}{s}{b}\quad ([b]=[0],[\tfrac{r}{p}]).
\end{split}
\end{align*}
In parallel to the definition of $\W_{r,s}$ in \eqref{simple triplet modules} as screening kernels, $\afWb{r}{s}{b}$ is also characterized as
\begin{align*}
\afWb{r}{s}{b} \simeq \displaystyle{\mathrm{Ker}\ Q_-^{[r]}|_{\Pi[b]\otimes \V_{r,s}}},\ (1 \leq r <p),\ \afWb{p}{s}{b} \simeq \Pi[b]\otimes \V_{p,s}.
\end{align*}
through the screening operator 
\begin{align*}
Q_{-}^{[r]}=\int_{[\Gamma_r]} \prod_{a=1}^r Y(e^{-\frac{1}{p}(u+v)}\otimes e^{-\frac{1}{\sqrt{p}}\alpha},z_a)dz,
\end{align*}
which is the pull-back of $\mathcal{Q}_-^{[r]}$ in \eqref{screening operators}.
It follows from \eqref{easy decomposition I} and \eqref{decomposition data} that
\begin{align}\label{eq: X-decomposition}
\afWb{r}{s}{b}\simeq \bigoplus_{n\geq0} \C^{2n+s}\otimes g^*\left(E^{0,[b-a_{r,s}]}_{r,2n+s}\right),\
\afXpm{r}{s}\simeq  \bigoplus_{n\geq0} \C^{2n+s}\otimes L^+_k(\lambda_{\pm r,\pm(2n+s)})
\end{align}
as $(\mathrm{SL}_2^{\af},V^k(\sll_2))$-bimodules and, from \eqref{vir-aff Felder complex} and \eqref{easy decomposition II} that
\begin{align}
&0\rightarrow \afWb{r}{s}{b} 
\rightarrow \Pi[b]\otimes \V_{r,s}
\rightarrow \afWb{p-r}{3-s}{b+\tfrac{p-r}{p}}(-\varpi)
\rightarrow 0,\label{Felder complex for PiV}
\\
&0 \rightarrow \afXp{r}{s} 
\rightarrow \mathcal{W}_{r,s}^{+} 
\rightarrow \mathcal{S}\afXp{r+1}{s} 
\rightarrow 0,
\label{sequence by affine representation 1}
\\
&0\rightarrow \afXm{r}{s}  \rightarrow \W_{r,s}^{-}\rightarrow \mathcal{S}\afXm{r-1}{s} \rightarrow 0.
\label{sequence by affine representation 2}
\end{align}
since $[-a_{r,s}]=[\tfrac{p-r}{p}-a_{p-r,3-s}]$.

\subsection{Structure}\label{subsection structure of free field modules}
The following proposition can be proven by the same manner as Proposition \ref{decomposition of weight modules} by using $\beta(z)=Y(e^{u+v},z)$ and $Y(\beta\gamma,z)=v(z)$ through \eqref{FMS}.
\begin{proposition}[\cite{RW}]
\label{prop: digression betagamma}\hspace{0mm}\\
\textup{(1)} 
$\Pi_a[b]$ for $[b]\neq [\tfrac{a}{4p}]$ is a simple $\beta\gamma$-module.\\
\textup{(2)} 
$\Pi_a[\tfrac{a}{4p}]$ is indecomposable and admits a short exact sequence 
\begin{align*}
0\rightarrow \mathcal{S}_{av}\beta\gamma \rightarrow \Pi_a[\tfrac{a}{4p}]\rightarrow \mathcal{S}_{(a-1)v}\beta\gamma \rightarrow 0.
\end{align*} 
Moreover, $\mathcal{S}_{av}\beta\gamma$ is generated by $e^{av}$ and $\mathcal{S}_{(a-1)v}\beta\gamma$ by the image of $e^{av-(u+v)}$.
\end{proposition}

Let us introduce some notation.
We denote by $V^k(\sll_2)$-wmod the category of weak $V^k(\sll_2)$-modules, or equivalently, the category of smooth $\widehat{\sll}_{2}$-module of level $k$, $V^k(\sll_2)\text{-bmod}$ the full subcategory consisting of objects whose conformal grading (in other words, $L_{\mathrm{sug},0}$-eigenvalues) are bounded from below, and $V^k(\sll_2)\text{-mod}^{\geq 0}$ the full subcategory consisting of \emph{ind-objects} of $V^k(\sll_2)\text{-bmod}$, namely, objects in $V^k(\sll_2)$-wmod which are isomorphic to directs limits of objects of  $V^k(\sll_2)\text{-bmod}$. 
Obviously, these categories are related as 
\begin{align*}
V^k(\sll_2)\text{-bmod} \subset V^k(\sll_2)\text{-mod}^{\geq 0} \subset V^k(\sll_2)\text{-wmod}.
\end{align*}
Note that for $M\in V^k(\sll_2)\text{-wmod}$, there exists a unique maximal $V^k(\sll_2)$-submodule $\tau(M)$ lying in $V^k(\sll_2)\text{-mod}^{\geq 0}$.  Indeed, it is the union of all the $V^k(\sll_2)$-submodules which are objects in $V^k(\sll_2)\text{-bmod}$.
For $\lambda\in \mathfrak{h}^\ast$ and $r,s\in\Z$, we have 
\begin{align}\label{lemm:bounded L+ and unbounded SL+}
L^+_k(\lambda)\in V^k(\sll_2)\text{-bmod},\quad
\mathcal{S}L^+_k(\lambda_{r,s})\in V^k(\sll_2)\text{-bmod}
\iff 
\lambda_{r,s}=\lambda_{1,s'\geq 1}
\end{align}

We start with $\Pi[0]\otimes \V_{r,s}$. We have the $\beta\gamma\otimes V_{\sqrt{p}\mathrm{A}_1}$-module isomorphisms
\begin{align}\label{align: digression betagamma 1}
\begin{split}
\mathcal{S}_{-v-\frac{1}{\sqrt{p}}\varpi}(\beta\gamma\otimes\mathcal{V}_{r+1,s})
&\simeq 
\mathcal{S}_{-v}(\beta\gamma)\otimes \mathcal{V}_{r,s}
\simeq
\Pi[0]\otimes\mathcal{V}_{r,s}/\beta\gamma\otimes\mathcal{V}_{r,s},\\
\bold{1}\otimes e^{\alpha_{r+1,s}}
&\mapsto
\bold{1}\otimes e^{\alpha_{r,s}}
\mapsto
[e^{-(u+v)}\otimes e^{\alpha_{r,s}}]
\end{split}
\end{align} 
and thus obtain the short exact sequence
\begin{align}\label{short exact sequence for PiV and betagamma}
0\rightarrow
\beta\gamma\otimes\mathcal{V}_{r,s}
\rightarrow 
\Pi[0]\otimes\mathcal{V}_{r,s}
\rightarrow
\mathcal{S}(\beta\gamma\otimes\mathcal{V}_{r+1,s})(\delta_{r,p}\varpi)
\rightarrow
0.
\end{align}

\begin{lemma}\label{prop: analysis on X}
For $1\leq r< p$,\\
\textup{(1)} 
$\afXp{r}{s}=\beta\gamma\otimes\mathcal{V}_{r,s}\cap \afWp{r}{s}$.\\
\textup{(2)}
The restriction of \eqref{short exact sequence for PiV and betagamma} to $\mathcal{W}_{r,s}^{+}\subset \Pi[0]\otimes \V_{r,s}$ is isomorphic to \eqref{sequence by affine representation 1}.
\end{lemma}

\proof
(1) 
It follows from Proposition \ref{decomposition of weight modules}, \eqref{pull-back of PiV} and \eqref{eq: X-decomposition} that 
highest weight vectors of copies of $L_k^+(\lambda_{r,2n+s})$ (resp. generator for $\mathcal{S}L_k^+(\lambda_{r+1,2n+s})$) inside $\afWp{r}{s}\subset \Pi[0]\otimes \V_{r,s}$ are spanned by 
\begin{align}\label{shape of (co)singular vectors}
v_{r,2n+s;a}^{\af}=Q_+^a (\mathbf{1}\otimes e^{\alpha_{r,2n+s}}),\quad 
\overline{v}_{r,2n+s;a}^{\af}=Q_+^a (e^{-(u+v)}\otimes e^{\alpha_{r,2n+s}}),
\end{align}
respectively $(0 \leq a < 2n+s-1)$.
As $[Q_{\mathrm{FMS}},Q_+]=0$, we have
\begin{align*}
Q_{\mathrm{FMS}}\ v_{r,2n+s;a}^{\af}=0,\quad 
Q_{\mathrm{FMS}}\ \overline{v}_{r,2n+s;a}^{\af}\not= 0.
\end{align*}
As $\beta\gamma \otimes \V_{r,s} =\mathrm{Ker} Q_{\mathrm{FMS}}|_{\Pi[0]\otimes \V_{r,s}}$, the assertion follows.
(2)  Clear from (1).
\endproof

To understand the $(B^{\af},V^k(\sll_2))$-module structure of $M_0:=\Pi[0]\otimes\V_{p,s}$, set
\begin{align}\label{socle sequence of PiVp}
\begin{split}
M_1:=\{v\in M_0~|~f_{1}^{\gg0}v=0\},\quad
M_2:=\beta\gamma\otimes\mathcal{V}_{p,s}, \quad
M_3:=\afXp{p}{s}
\end{split}
\end{align}
Obviously, these $(B^{\af},V^k(\sll_2))$-submodules of $M_0$ are related as $M_i\subsetneq M_{i+1}$.
The last paragraph of the proof of Proposition \ref{decomposition of weight modules} implies that
\begin{align*}
M_3\simeq\bigoplus_{n\geq0}\C^{2n+s}\otimes g^\ast\mathcal{S}N_3^{2n+s}\simeq\afXp{p}{s},~
M_0/M_1\simeq\bigoplus_{n\geq0}\C^{2n+s}\otimes g^\ast\mathcal{S}(\tfrac{N_0^{2n+s}}{N_1^{2n+s}})\simeq\mathcal{S}\afXm{p-1}{s}.
\end{align*}
Let us give evaluations of ${M_2}/{M_3}$ and ${M_1}/{M_2}$ from below, which would later turn out to be isomorphic to ${M_2}/{M_3}$ and ${M_1}/{M_2}$, respectively.
Note that we have
\begin{align*}
Q_+^{2n+s-1}(e^{N(u+v)}\otimes e^{\alpha_{p,2n+s}})\in\C^{\times}e^{(N+2n+s-1)(u+v)}\otimes e^{(2n+s-1)\sqrt{p}\alpha+\alpha_{p,2n+s}}
\end{align*}
for $N\in\Z$.
Set $v_{2n+s}=e^{-(2n+s-1)(u+v)}\otimes e^{\alpha_{p,s+2n}}$.
As $Q_+^j(\mathbf{1}\otimes e^{\alpha_{p+1,2n+s}})=0$ iff $j\geq s+2n-1$, the image $[Q_+^j(e^{-(u+v)}\otimes e^{\alpha_{p+1,2n+s}})]$ under the isomorphism \eqref{align: digression betagamma 1} so is.
In other words, $Q_+^j(e^{-(u+v)}\otimes e^{\alpha_{p+1,2n+s}})$ (and hence $Q_+^jv_{2n+s}$) is in $\beta\gamma\otimes\V_{p,s}$ iff $j\geq s+2n-1$.
On the other hand, the isomorphism \eqref{align: digression betagamma 1} for $r=p-1$ sends $Q_+^{2n+s-1}v_{2n+s}$ to $[e^{-(u+v)}\otimes e^{(2n+s-1)\sqrt{p}\alpha+\alpha_{p+1,2n+s}}]$.
Since the cosingular vector in \cite[Theorem 1.1,1.2]{AM}
gives the highest weight vector
\footnote{For $p=1$ it is given by $e^{-(2n+s)(u+v)}\otimes e^{\alpha_{p-1,2n+s}}$.} 
of the $(2n+s-1)$-dimensional $\mathrm{SL}_2^{\af}$-module with the lowest weight vector $e^{-(u+v)}\otimes e^{(2n+s-1)\sqrt{p}\alpha+\alpha_{p+1,2n+s}}$, there exists a highest weight vector $y_{2n+s}$ of the $(\mathrm{SL}_2^{\af},V^k(\sll_2))$-bimodule $\C^{2n+s}\otimes N_1^{2n+s}$ such that 
$Q_+v_{2n+s}+y_{2n+s-2}\in M_2$,
Therefore, we have
\begin{align}\label{lemm lower bound of $M_i/M_{i+1}$}
\begin{split}
&{M_2}/{M_3}\supseteq\bigoplus_{n\geq 0}\C^{2n+s-1}\otimes g^\ast\mathcal{S}(\tfrac{N_2^{2n+s}}{N_3^{2n+s}})\twoheadrightarrow\afXm{0}{3-s}(-\varpi),\\
&{M_1}/{M_2}\supseteq[\bigoplus_{n\geq 0}\C^{2n+s-1}\otimes U(V^k(\sll_2))v_{2n+s}]\twoheadrightarrow\afXm{0}{3-s}(\varpi).
\end{split}
\end{align}
In particular, at the level of Grothendieck group, we have
\begin{align}\label{lemm lower bound of betagamma}
[\beta\gamma\otimes\V_{1,3-s}]=[\mathcal{S}^{-1}\tfrac{M_0}{M_2}]=[\mathcal{S}^{-1}\tfrac{M_0}{M_1}]+[\mathcal{S}^{-1}\tfrac{M_1}{M_2}]\geq
[\afXp{1}{3-s}]+[\afXm{p-1}{s}].
\end{align}

\begin{proposition}\label{cor: Felder complex for sl2 I}\hspace{0mm}
For $r\not= p$, \\
\textup{(1)} $\tau(\Pi[0]\otimes \V_{r,s})=\beta\gamma\otimes \V_{r,s}$ and $\tau(\afWp{r}{s})=\afXp{r}{s}$,\\
\textup{(2)} we have the following commutative diagram of $(B^{\af},V^k(\sll_2))$-bimodules.
\begin{align*}
\xymatrix{
& 0 \ar[d]& 0 \ar[d] &0 \ar[d] &\\
0 \ar[r] & \afXp{r}{s} \ar[r] \ar[d] & \beta\gamma\otimes \V_{r,s} \ar[r] \ar[d] & \afXm{p-r}{3-s}(-\varpi) \ar[r] \ar[d] & 0  \\
0 \ar[r] & \afWp{r}{s} \ar[r] \ar[d] &  \Pi[0]\otimes \V_{r,s} \ar[r]\ar[d] &  \afWm{p-r}{3-s} (-\varpi) \ar[r] \ar[d] & 0  \\
0 \ar[r] & \mathcal{S}\afXp{r+1}{s} \ar[r] \ar[d] & \mathcal{S}(\beta\gamma\otimes \V_{r+1,s}) \ar[r] \ar[d] &
\mathcal{S}\afXm{p-r-1}{3-s}(-\varpi) \ar[r] \ar[d] & 0  \\
& 0 & 0 &0  &
}
\end{align*}
\end{proposition}

\proof
Note that the second row is \eqref{Felder complex for PiV} and the columns are \eqref{sequence by affine representation 1}, \eqref{short exact sequence for PiV and betagamma}, \eqref{sequence by affine representation 2} respectively from the left, and that the left two columns commute by Lemma \ref{prop: analysis on X}.
Therefore, we have to show that the right two columns commute. 
We denote by $X_{ij}^r$ the term in the $(i,j)$-entry of the diagram.
It is clear from \eqref{lemm:bounded L+ and unbounded SL+} or the conformal gradings that
$$\tau(X_{21}^r)=X_{11}^r,\quad \tau(X_{23}^r)=X_{13}^r,\quad \tau(X_{12}^r)=X_{12}^r.$$
Then $X_{22}^r\twoheadrightarrow X_{23}^r$ induces 
\begin{align}\label{induced injection and surjection}
    X_{12}^r/X_{11}^r\hookrightarrow X_{13}^r,\quad X_{32}^{r}/X_{31}^r\twoheadrightarrow X_{33}^{r}.
\end{align}
Therefore, at the level of Grothendieck group, we have 
$[X_{12}^r]=[X_{11}^r]+[\im(X_{12}^r\rightarrow X_{13}^r)]\leq [X_{11}^r]+[X_{13}^r]$ for $r\not= p$ and similarly
$[X_{32}^{r}]\leq [X_{31}^{r}]+[X_{33}^{r}]$ for $r<p-1$.
Thus we have $\tau(X_{22}^{r})=X_{12}^{r}$.
Since
\begin{align*}
[X_{22}^r]&=[X_{21}^r]+[X_{23}^r]=[X_{11}^r]+[X_{13}^r]+[X_{31}^r]+[X_{33}^r]\\
&=[X_{12}^r]+[X_{32}^r]=[X_{11}^r]+[\im(X_{12}^r\rightarrow X_{13}^r)]+[\Ker(X_{32}^r\rightarrow X_{33}^{r})]+[X_{33}^{r}],
\end{align*}
by \eqref{lemm:bounded L+ and unbounded SL+}, we have $[\im(X_{12}^r\rightarrow X_{13}^r)]=[X_{13}^r]$ and $[\Ker(X_{32}^r\rightarrow X_{33}^{r})]=[X_{31}^r]$.
It completes the proof for $r<p-1$.
Let us consider the case of $r=p-1$.
For $p\geq 3$ (resp. for $p=2$), the first row and $\tau(X_{22}^{p-1})=X_{12}^{p-1}$ are obtained by applying $\mathcal{S}_{-1}$ to the third row of the case $r=p-2$ (resp. \eqref{lemm lower bound of betagamma} and \eqref{induced injection and surjection}).
One can show the remaining similarly.
\endproof

To understand the structure of $\Pi[\tfrac{r}{p}]\otimes \V_{r,s}$, we make a preparation.
\begin{lemma}[{\cite{ACK}}]\label{Ext vanishing}
For $\lambda,\mu\in \mathfrak{h}^*$ such that $\lambda-\mu> \alpha$ and the $\sll_2$-module $M^-(\lambda)$ is $U(\mathfrak{n}_+)$-free and  $M^+(\mu)$ is $U(\mathfrak{n}_-)$-free, then
$$\mathrm{Ext}^1(L_k^-(\lambda),L_k^+(\mu))=0,\quad \mathrm{Ext}^1(L_k^+(\mu),L_k^-(\lambda))=0.$$
\end{lemma}
\proof
We reproduce a proof for the convenience of the readers.
For the first assertion, it suffices to show that every short exact sequence 
\begin{align*}
0\rightarrow L_k^+(\mu) \rightarrow M \rightarrow L_k^-(\lambda) \rightarrow 0
\end{align*}
splits. It is useful to see the shape of the sets of $(L_0,h)$-eigenvalues of the modules $L_k^+(\mu)$ and $L_k^-(\lambda)$. 
Indeed they are supported in the following shape of regions.
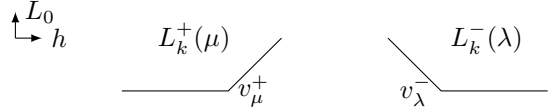
\begin{align*}
\begin{tikzpicture}[x=7mm,y=7mm]
\draw[->] (0,1) -- (0.5,1);
\draw[->] (0,1)-- (0,1.5);
\draw (0.5,1) node[right]{$h$};
\draw (0,1.5) node[right]{$L_0$};
\draw (2,0) --(4,0)--(5,1);
\draw (4,0) node[right]{$v_\mu^+$};
\draw (2.5,1) node[right]{$L_k^+(\mu)$};
\draw (7,1) --(8,0)--(10,0);
\draw (8,0) node[left]{$v_\lambda^-$};
\draw (8,1) node[right]{$L_k^-(\lambda)$};
\end{tikzpicture}
\end{align*}
Since $M^-(\lambda)=U(\mathfrak{n}_+)v_\lambda^-$, we can find an element $v$ in the copy $M^-(\lambda)$ inside $M$ such that the set of $(L_0,h)$-eigenvalues of $U(\sll_2[t^{-1}]t^{-1})v$ has no intersection with that of $L_k^+(\mu)$.
Then we consider the $V^k(\sll_2)$-submodule $N$ generated by $v$. 
By the PBW base theorem, $N=U(\sll_2[t])v$. Then we have the following three possibilities of the relative position of the set of $(L_0,h)$-eigenvalues according to the conformal weights $\Delta_\lambda$ and  $\Delta_\mu$ of $v_\lambda^-$ and $v_\lambda^+$.
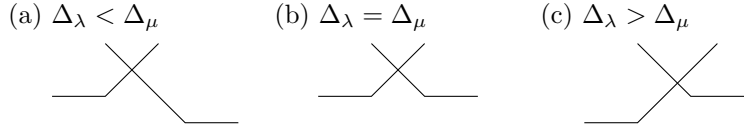
\begin{align*}
\begin{tikzpicture}[x=7mm,y=7mm]
\draw (0,2) node[right]{(a) $\Delta_\lambda< \Delta_\mu$};
\draw (1,0.5) --(2,0.5)--(3,1.5);
\draw (4.5,0) --(3.5,0)--(2,1.5);
\draw (5,2) node[right]{(b) $\Delta_\lambda= \Delta_\mu$};
\draw (6,0.5) --(7,0.5)--(8,1.5);
\draw (7,1.5) --(8,0.5)--(9,0.5);
\draw (10,2) node[right]{(c) $\Delta_\lambda> \Delta_\mu$};
\draw (11,0) --(12,0)--(13.5,1.5);
\draw (12,1.5) --(13,0.5)--(14,0.5);
\end{tikzpicture}
\end{align*}
We show case by case that $N$ is a proper submodule and thus is isomorphic to $L_k^-(\mu)$.
(a) The subspace of lowest conformal weight of $M$ is $M^-(\lambda)$. Since $M^+(\mu)$ is $U(\mathfrak{n}_-)$-free, $M^+(\mu)$ contains a vector lying outside of $N$ and thus $N$ is proper. (b) The subspace of lowest conformal weight of $M$ is an extension of $M^-(\lambda)$ by $M^+(\mu)$. By $\lambda-\mu> \alpha$, this extension splits and thus is isomorphic to $M^-(\lambda)\bigoplus M^+(\mu)$. Therefore, $N$ is proper. (c) $N$ does not contain $M^+(\mu)$ and thus proper. Hence, we have shown that $N\simeq L_k^-(\mu)$ and thus $M$ splits. The proof for the second assertion is similar and so we omit it. This completes the proof.
\endproof

\begin{proposition}\label{cor: Felder complex for sl2 II}\hspace{0mm}
For $r\not= p$, we have $\tau(\afWm{r}{s})=\afXm{r}{s}$.
Furthermore, we have the commutative diagram of $(\mathrm{B}^{\af},V^k(\sll_2))$-bimodules

\begin{align*}
\xymatrix{
& 0 \ar[d]& 0 \ar[d] &0 \ar[d] &\\
0 \ar[r] & \afXm{r}{s} \ar[r] \ar[d] & \tau(\Pi[\tfrac{r}{p}]\otimes \V_{r,s}) \ar[r] \ar[d] & \afXp{p-r}{3-s}(-\varpi) \ar[r] \ar[d] & 0  \\
0 \ar[r] & \afWm{r}{s} \ar[r] \ar[d] &  \Pi[\tfrac{r}{p}]\otimes \V_{r,s} \ar[r]\ar[d] &  \afWp{p-r}{3-s}(-\varpi) \ar[r] \ar[d] & 0  \\
0 \ar[r] & \mathcal{S}\afXm{r-1}{s} \ar[r] \ar[d] & 
N_{r,s}
\ar[r] \ar[d] & \mathcal{S}\afXp{p-r+1}{3-s}(-\varpi) \ar[r] \ar[d] & 0  \\
& 0 & 0 &0  &
}
\end{align*}
where $N_{r,s}:=\Pi[\tfrac{r}{p}]\otimes\mathcal{V}_{r,s}/\tau(\Pi[\tfrac{r}{p}]\otimes\mathcal{V}_{r,s})$.
\end{proposition}

\proof
The first and third columns are the third and first columns of Proposition \ref{cor: Felder complex for sl2 I}, respectively.
The second row is \eqref{Felder complex for PiV}.

In particular, at the level of Grothendieck group of $V^k(\sll_2)\textup{-wmod}$, we have
\begin{align*}
[\Pi[\tfrac{r}{p}]\otimes\V_{r,s}]=[\afXm{r}{s}]+[\afXp{p-r}{3-s}]+[\mathcal{S}\afXm{r-1}{s}]+[\mathcal{S}\afXp{p-r+1}{3-s}].
\end{align*}
To construct the first row, we consider the quotient $V^k(\sll_2)$-module $ \Pi[\tfrac{r}{p}]\otimes \V_{r,s}/\afXm{r}{s}$. 
Thanks to \eqref{sequence by affine representation 2}, it contains a submodule, say $M$, which fits into a short exact sequence
\begin{align*}
0\rightarrow \mathcal{S}\afXm{r-1}{s}\rightarrow M \rightarrow \afXp{p-r}{3-s}(-\varpi) \rightarrow 0.
\end{align*} 
We show that it splits. Indeed, we apply the exact functor $\mathcal{S}^{-1}$. By the bifunctoriality of $\mathrm{Ext}^1(\bullet,\bullet)$, it suffices to show
\begin{align}\label{check for ext vanishing}
\mathrm{Ext}^1(L_k^-(\lambda_{p-r+1,2n+5-s}),L_k^+(\lambda_{-r+1,-2m-s}))=0,\quad (n,m\geq 0),
\end{align}
by \eqref{eq: X-decomposition}. 
As $\lambda_{p-r+1,2n+5-s}\notin \Z_{\leq0}\varpi$ and $\lambda_{-r+1,-2m-s}\notin \Z_{\geq0}\varpi$, $M^-(\lambda_{p-r+1,2n+5-s})$ is $U(\mathfrak{n}_+)$-free and $M^+(\lambda_{-r+1,-2m-s})$ is $U(\mathfrak{n}_-)$-free.
Since 
\begin{align*}
\lambda_{p-r+1,2n+5-s}-\lambda_{-r+1,-2m-s}=(n+m+1)\alpha+\varpi>\alpha,
\end{align*}
Lemma \ref{Ext vanishing} implies \eqref{check for ext vanishing}. 
Therefore, we have a $V^k(\sll_2)$-submodule $M\subset \Pi[\frac{r}{p}]\otimes \V_{r,s}$ satisfying the short exact sequence
$$0\rightarrow \afXm{r}{s} \rightarrow M \rightarrow \afXp{p-r}{3-s}(-\varpi)  \rightarrow 0.$$
As the composition factors of $M$ exhaust all the composition factors of $\Pi[\tfrac{r}{p}]\otimes \V_{r,s}$ lying in $V^k(\sll_2)$-bmod, we conclude $M=\tau(\Pi[\tfrac{r}{p}]\otimes \V_{r,s})$.
Let us show the remaining claims.
For $1\leq i,j\leq 3$, let $X_{ij}$ denotes the $(i,j)$-th entry of the diagram.
We have $X_{11}=X_{12}\cap X_{21}$ by the first row/column and
$
[X_{12}\cap X_{21}]\leq [X_{12}]\cap [X_{21}]=[X_{11}]
$.
As $\tau(X_{22}/X_{12})=0$ and $X_{11}=\tau(X_{21})$, $X_{21}\hookrightarrow X_{22}\twoheadrightarrow X_{22}/X_{12}$ factors through $X_{31}\rightarrow X_{22}/X_{12}$.
This map is injective because
\begin{align*}
X_{31}\simeq X_{21}/X_{11}\simeq X_{21}/X_{12}\cap X_{21}\simeq (X_{21}+X_{12})/X_{12}\subseteq X_{22}/X_{12}.
\end{align*}
Similarly,
$X_{22}\twoheadrightarrow X_{23}\twoheadrightarrow X_{33}$ factors through
$X_{22}/X_{12}\twoheadrightarrow X_{33}$.
\endproof

\begin{remark}
We expect $N_{r,s}\simeq \mathcal{S}\tau(\Pi[\tfrac{r-1}{p}]\otimes \V_{r-1,s})$ as $(B^{\af},V^k(\sll_2))$-bimodules.
\end{remark}
\begin{proof}[Proof of Proposition \ref{decomposition of weight modules}]
Finally, we complete the proof of Proposition \ref{decomposition of weight modules} \textup{(2)}-(ii), (iii).
Let us recall the $(\mathrm{B}^{\af},V^k(\sll_2))$-submodules \eqref{socle sequence of PiVp} of $M_0=\Pi[0]\otimes\V_{p,s}$.

First we consider the case $p\geq 2$.
By applying $\mathcal{S}^{-1}$ to the third row in Proposition \ref{cor: Felder complex for sl2 I} for $r=p-1$, we have $M_2/M_3\simeq\afXm{0}{3-s}$.
On the other hand, by applying $\mathcal{S}$ to the first row in Proposition \ref{cor: Felder complex for sl2 I} for $r=1$, we obtain that $M_1/M_2\simeq\afXm{0}{3-s}$.
By \eqref{lemm lower bound of $M_i/M_{i+1}$}, we have ${N_2^{s'}}/{N_3^{s'}}\simeq\delta_{s',1}L^+_k(s'-2)$ for $s'\geq 1$. 
By process of elimination, we obtain that ${N_1^{s'}}/{N_2^{s'}}\simeq L^+_k(s')$, which is generated by $y_{s'}$.

Let us consider the case of $p=1$.
In this case, the Weyl modules $\mathbb{V}^k(n\varpi)$ ($n\in\Z_{\geq 0}$) are known to be simple.
By comparing the characters of $M_2$ and $M_3$ (see \S\ref{subsection character formulae} below), we obtain that $N_2^{s'}/N_3^{s'}\simeq\delta_{s',1}L^+_k(s'-2)$.
One can show the remaining similarly.
\end{proof}

\section{Main results for $\mathrm{FT}_p(\sll_2)$}\label{section: main results}

\subsection{$\mathrm{FT}_p(\sll_2)$-modules}
Let us consider the equivariant bundles
\begin{align*}
\mathrm{SL}_2\times_{\mathrm{B}}(\Pi[b]\otimes\mathcal{V}_{r,s}),\quad
\mathrm{SL}_2\times_{\mathrm{B}}\tau(\Pi[b]\otimes\mathcal{V}_{r,s})
\end{align*}
and take their global sections
\begin{align*}
H^0(\mathrm{SL}_2\times_{\mathrm{B}}(\Pi[b]\otimes\mathcal{V}_{r,s})),\quad
H^0(\mathrm{SL}_2\times_{\mathrm{B}}\tau(\Pi[b]\otimes\mathcal{V}_{r,s})),
\end{align*}
for $1\leq r\leq p$, $s=1,2$ and $[b]\in \C/\Z$.
Note that sheaves of local sections for equivalent bundles are (sheaves of) $V^k(\sll_2)$-modules as
\begin{align*}
  \mathrm{SL}_2\times_\mathrm{B}V^k(\sll_2) \subset \mathrm{SL}_2\times_\mathrm{B}(\beta\gamma\otimes V_{\sqrt{p}A_1})
\end{align*}
is a sub-bundle. In particular, the global sections are $V^k(\sll_2)$-modules. 
Moreover, the first sheaf is a $\sFT(\sll_2)$-module and thus $H^0(\mathrm{SL}_2\times_{\mathrm{B}}(\Pi[b]\otimes\mathcal{V}_{r,s}))$ is a $\sFT(\sll_2)$-modules by construction.

\begin{theorem}\label{FT(sl2,0)}\hspace{0cm}\\
\textup{(1)} 
We have 
\begin{align*}
H^{n}(\mathrm{SL}_2\times_{\mathrm{B}}(\Pi[b]\otimes \V_{r,s}))&\simeq \delta_{n,0} \afWb{r}{s}{b},\quad (b\in\C,~1\leq r\leq p),\\
H^{n}(\mathrm{SL}_2\times_{\mathrm{B}} \tau(\Pi[b]\otimes \V_{r,s}))&\simeq \delta_{n,0} \afXpm{r}{s},\quad ([b]=[0], [\tfrac{r}{p}],~1\leq r< p),\\
H^{n}(\mathrm{SL}_2\times_{\mathrm{B}} \beta\gamma\otimes \V_{p,s}))&\simeq \delta_{n,0} \afXp{p}{s}.
\end{align*}
\textup{(2)} 
The $(\mathrm{SL}_2^{\af},V^k(\sll_2))$-bimodules $\afWb{r}{s}{b}$ and $\afXpm{r}{s}$ are $\mathrm{FT}_p(\sll_2)$-modules.
\end{theorem}
\proof
(1) The first case is shown in the same way as \eqref{computation by automorphisms}. The remaining case is shown in the same way as Theorem \ref{main result for the virasoro triplet algebra} by replacing \eqref{virasoro felder complex} with \eqref{Felder complex for PiV}, the first row in Proposition \ref{cor: Felder complex for sl2 I}, and the third row in Proposition \ref{cor: Felder complex for sl2 II}, respectively. 
The last case follows from the third row of Proposition \ref{cor: Felder complex for sl2 I} for $r=p-1$. 
(2) The cases $\afWb{r}{s}{b}$ and $\afXp{r}{s}$ follow from (1) and the fact that the sheaves of local sections for the equivalent bundles are $\sFT(\sll_2)$-modules by construction. 
Then the first row for the diagram in Proposition \ref{cor: Felder complex for sl2 I} (2) is a short exact sequence of $\FT(\sll_2)$-modules. Hence, $\afXm{r}{s}$ is a $\FT(\sll_2)$-module. This completes the proof.
\endproof
\begin{remark}
\textup{ 
It follows from Theorem \ref{FT(sl2,0)} that the short exact sequence \eqref{sequence by affine representation 2}
is indeed a complex of $\FT(\sll_2)$-modules. In particular, the quotients $\mathcal{S}\afXm{0}{s}$ and thus $\afXm{0}{s}$ ($s=1,2$) are also $\FT(\sll_2)$-modules.
}
\end{remark}
\begin{remark}\label{remark FT-module hom among nine terms}
\textup{The short exact sequences \eqref{Felder complex for PiV} and those in Proposition \ref{cor: Felder complex for sl2 I}, \ref{cor: Felder complex for sl2 II} are as $\mathrm{FT}_p(\sll_2)$-modules.
In fact, since the second rows/columns are short exact sequences of $\mathrm{FT}_p(\sll_2)$-modules, the first ones so are.
By applying Proposition \ref{prop:dual of PiV} below to Proposition \ref{cor: Felder complex for sl2 I}, \ref{cor: Felder complex for sl2 II}, also the third ones so are.}
\end{remark}

\begin{corollary}\label{cor: FT(sl2,0)}
For $1\leq r<p$, we have $\mathrm{FT}_p(\sll_2)$-module isomorphisms
\begin{align*}
&H^n(\mathrm{SL}_2\times_{\mathrm{B}}(\Pi[b]\otimes \V_{r,s}))\simeq H^{n+1}(\mathrm{SL}_2\times_{\mathrm{B}}\Pi[b+\tfrac{p-r}{p}]\otimes \V_{p-r,3-s})(-\varpi)),\\
&H^n(\mathrm{SL}_2\times_{\mathrm{B}} \tau(\Pi[0]\otimes \V_{r,s}))\simeq H^{n+1}(\mathrm{SL}_2\times_{\mathrm{B}}\tau(\Pi[\tfrac{p-r}{p}]\otimes \V_{p-r,3-s})(-\varpi)),\\
&H^n(\mathrm{SL}_2\times_{\mathrm{B}} \tau(\Pi[\tfrac{p-r}{p}]\otimes \V_{r,s}))\simeq H^{n+1}(\mathrm{SL}_2\times_{\mathrm{B}}\tau(\Pi[0]\otimes \V_{p-r,3-s})(-\varpi)),\\
&H^n(\mathrm{SL}_2\times_{\mathrm{B}}\beta\gamma\otimes\mathcal{V}_{p,s})
\simeq H^{n+1}(\mathrm{SL}_2\times_{\mathrm{B}}\mathcal{S}^{-1}N_{1,3-s}(-\varpi)).
\end{align*}
\end{corollary}

\subsection{Character formulae.}\label{subsection character formulae}
Let us introduce the Pochhammer symbol
\begin{align*}
(z_1,\cdots,z_m;q)_\infty:=\prod_{n=0}^{\infty}(1-z_1q^n)\cdots (1-z_mq^n).
\end{align*}
For a $(\mathrm{H}^{\af},V^k(\sll_2))$-bimodule $M$, we set the characters
\begin{align*}
    \ch_{M}(w;z,q):=\tr_{M}w^hz^{h_{0}}q^{L_0},\quad \ch_{M}(z,q):=\tr_{M}z^{h_{0}}q^{L_0}
\end{align*}
and use $z^{\lambda}$ instead of $z^{(\lambda,h)}$ $(\lambda\in \h^*)$.
Note that the conformal weight of the affine Verma module $\mathbb{M}_k^+(\lambda_{r,s})$ is 
$\Delta_{r,s}=\frac{(sp-(r-1))^2-p^2}{4p}$.
\begin{corollary}\label{character formula for sl2}
For $\lambda_{r,s}$ such that $r\not=0$ or $s>0$, the character of the $V^k(\sll_2)$-module $L_k^+(\lambda_{r,s})$ is given by 
\begin{align}\label{character formula}
\ch_{L_k^+(\lambda_{r,s})}(z,q)=
\frac{z^{\lambda_{r,s}}q^{\Delta_{r,s}}-z^{\lambda_{r,-s}}q^{\Delta_{r,-s}} }{(z^\alpha q,z^{-\alpha},q;q)_\infty}.
\end{align}
\end{corollary}
\proof
We combine Theorem \ref{FT(sl2,0)} and 
the Atiyah--Bott fixed point formula \cite{AB}
\begin{align*}
\sum_{n\geq 0}(-1)^n\ch_{H^n(\mathrm{G}\times_{\mathrm{B}}V)}(w;z,q)
=
\sum_{\lambda\in P_+}\ch_{L^+(\lambda)}(w)\sum_{\sigma\in W}(-1)^{l(\sigma)}\ch_{V^{h=\sigma\circ\lambda}}(z,q)
\end{align*}
as \cite{FT,Sug}.
By comparing this formula with 
\eqref{eq: X-decomposition}, 
we have
\begin{align*}
\ch_{L_k^+(\lambda_{r,2n+s})}(z,q)
&=\ch_{\beta\gamma\otimes \pi^\alpha_{\alpha_{r,2n+s}}}(z,q)-\ch_{\beta\gamma\otimes \pi^\alpha_{\alpha_{r,-(2n+s)}}}(z,q).
\end{align*}
Then it follows from \eqref{sl2 wakimoto} that
\begin{align*}
\ch_{\beta\gamma}(z,q)=\frac{1}{(z^\alpha q,z^{-\alpha};q)_\infty},\quad \ch_{\pi^\alpha_{\alpha_{r,s}}}(z,q)=\frac{z^{\lambda_{r,s}}q^{\Delta_{{r,s}}}}{(q;q)_\infty}
\end{align*}
and thus the assertion for $s\geq 1$ and $1\leq r\leq p$. 
Next, we compute the characters $\ch_{L_k^+(\lambda_{-r,-2n-s})}$ appearing in $\afXm{r}{s}$.
By the first row in Proposition \ref{cor: Felder complex for sl2 I},
\begin{align*}
\ch_{\afXm{r}{s}}(w;z,q)=w(\ch_{\beta\gamma \otimes \V_{p-r,3-s}}(w;z,q)-\ch_{\afXp{p-r}{3-s}}(w;z,q)).
\end{align*}
Let us introduce 
$f_{r,s}=z^{\lambda_{r,s}}q^{\Delta_{r,s}}$.
Since
\begin{align*}
&A_{r,s}:=\ch_{\beta\gamma \otimes \V_{r,s}}(w;z,q)=\frac{\sum_{n\in\Z}w^{(s+2n-1)}f_{r,2n+s}}{(z^\alpha q,z^{-\alpha},q;q)_\infty},\\
&B_{r,s}:=\ch_{\afXp{r}{s}}(w;z,q)=\sum_{n\geq 0}\frac{w^{(2n+s)}-w^{-(2n+s)}}{w-w^{-1}}\frac{(f_{r,2n+s}-f_{r,-(2n+s)})}{(z^\alpha q,z^{-\alpha},q;q)_\infty},
\end{align*}
by \eqref{eq: X-decomposition}
, we have
\begin{align*}
\ch_{L_k^+(\lambda_{-r,-2n-s})}&(z,q)
=\mathrm{CT}_{w}\left[\left(w^{-(2n+s-1)}-w^{-(2n+s+1)} \right)w(A_{p-r,3-s}-B_{p-r,3-s})  \right]\\
&=\frac{(f_{p-r,2n+s-1}-f_{p-r,2n+s+1})-(f_{p-r,2n+s-1}-f_{p-r,-(2n+s-1)})}{(z^\alpha q,z^{-\alpha},q;q)_\infty}\\
&=\frac{f_{p-r,-(2n+s-1)}-f_{p-r,2n+s+1}}{(z^\alpha q,z^{-\alpha},q;q)_\infty}
\end{align*}
where for $f(w,z,q)=\sum_{n\in\Z}w^nf_n(z,q)$, $\mathrm{CT}_w[f(w,z,q)]$ denotes the constant term $f_{0}(z,q)$.
Then we obtain the assertion for $s\leq 0$ and $1\leq r<p$. 
\endproof
\begin{corollary}\label{cor: inverse Feigin-Fuks short exact sequence}
For $\lambda_{r,s}$ such that $r\not =0$ or $s>0$, we have the following exact sequences of $V^k(\sll_2)$-modules
\begin{align}\label{BGG resolution}
0\rightarrow\mathbb{M}^+_k(\lambda_{r,-s})\rightarrow \mathbb{M}_k^+(\lambda_{r,s})\rightarrow L_k^+(\lambda_{r,s})\rightarrow 0.
\end{align}
\end{corollary}
\proof
By Corollary \ref{character formula for sl2}, the surjection $\mathbb{M}^+_k(\lambda_{r,s})\twoheadrightarrow L_k^+(\lambda_{r,s})$ has a singular vector of weight $\lambda_{r,-s}$. 
Thus we have a homomorphism $\mathbb{M}^+_k(\lambda_{r,-s})\rightarrow \mathbb{M}^+_k(\lambda_{r,s})$, which is injective since affine Verma modules are free $U(\widehat{\mathfrak{n}}_-)$-modules of rank one where $\widehat{\mathfrak{n}}_-:=\mathfrak{n}_-\oplus \sll_2[t^{-1}]t^{-1}$. Hence, we have the complex in the assertion, which is exact by Corollary \ref{character formula for sl2}.
\endproof

\begin{corollary}\label{Property of Weyl modules}\hspace{0mm}\\
\textup{(1)} 
The Weyl modules $\mathbb{V}^k(n\varpi)\ (n \in \Z_{\geq0})$ are simple.\\
\textup{(2)} 
We have isomorphisms of $(\mathrm{SL}_2^{\af},V^k(\sll_2))$-bimodules
\begin{align*}
\mathrm{FT}_p(\sll_2)=H^0(\mathrm{SL}_2\times _{\mathrm{B}} (\beta\gamma \otimes \V_{1,1}))&\simeq \bigoplus_{n\colon \mathrm{even}}L(n\varpi)\otimes \mathbb{V}^k(n\varpi),\\
H^0(\mathrm{SL}_2\times _{\mathrm{B}} (\beta\gamma \otimes \V_{1,2}))&\simeq \bigoplus_{n\colon \mathrm{odd}}L(n\varpi)\otimes \mathbb{V}^k(n\varpi).
\end{align*}
\end{corollary}
\proof
\textup{(1)} Since $n\varpi=\lambda_{1,n+1}$, we have
\begin{align*}
\ch_{L_k^+(n\varpi)}(z,q)
&=\frac{z^{\lambda_{1,n+1}}q^{\Delta_{1,n+1}}(1-z^{-(n+1)\alpha})}{(z^\alpha q,z^{-\alpha},q;q)_\infty}\\
&=\frac{q^{\Delta_{1,n+1}}\ch_{L(n\varpi)}(z)}{(z^\alpha q,z^{-\alpha}q,q;q)_\infty}=\ch_{\mathbb{V}^k(n\varpi)}(z,q).
\end{align*}
As we have a natural surjection $\mathbb{V}^k(n\varpi)\twoheadrightarrow L_k^+(n\varpi)$, the equality implies $\mathbb{V}^k(n\varpi)\simeq L_k^+(n\varpi)$.
\textup{(2)} It is immediate from \textup{(1)} and Theorem \ref{FT(sl2,0)} \textup{(1)}. 
\endproof

\subsection{BRST reduction}
Corollary \ref{cor: inverse Feigin-Fuks short exact sequence}
builds a bridge between $\mathrm{FT}_p(\sll_2)$ and $\W(p)$.
For the simplicity of notation, we write the BRST reduction for the principal nilpotent element $f=f_{\mathrm{prin}}$ as $H_{DS}^\bullet=H^\bullet_{DS,f_{\mathrm{prin}}}$. 
In our case, $H_{DS}^\bullet(N)$ for a $V^k(\sll_2)$-module $N$ is the cohomology of the complex 
\begin{align*}
    C_{DS}^\bullet(N)=N\otimes V_\Z,\quad d=\int Y((e+\mathbf{1})\otimes e^{-x},z)dz
\end{align*}
with the cohomological degree given by the lattice $\Z=\Z x$.
By \cite{Ar}, we apply $H_{DS}^\bullet$ to \eqref{BGG resolution} and obtain the short exact sequence of $L(c_{1,p},0)$-modules, namely the Feigin--Fuchs resolution \cite{FeFu} of the simple $L(c_{1,p},0)$-module $L_{r,s}$ by Virasoro Verma modules $M_{r,s}$'s:
\begin{align}\label{Virasoro module through BRST reduction}
0\rightarrow \underbrace{H_{DS}^0(\mathbb{M}^+_k(\lambda_{r,-s}))}_{M_{p-r,s+1}}\rightarrow \underbrace{H_{DS}^0(\mathbb{M}_k^+(\lambda_{r,s}))}_{M_{r,s}}\rightarrow \underbrace{H_{DS}^0(L_k^+(\lambda_{r,s}))}_{L_{r,s}}\rightarrow 0.
\end{align}
for $1\leq r \leq p$ and $s\in \Z_{\geq 1}$.

\begin{theorem}\label{compatibility of FT algeras for sl2}\hspace{0mm}\\
\textup{(1)} {\cite[Theorem 14]{ACGY}} There is an isomorphism of vertex algebras 
\begin{align*}
H_{DS}^{n}(\mathrm{FT}_{p}(\sll_2))\simeq \delta_{n,0}\W(p).
\end{align*}
\textup{(2)} There is an isomorphism of $\W(p)$-modules
\begin{align*}
\quad H_{DS}^{n}(\afXpm{r}{s})\simeq \delta_{n,0}\W_{r,s},\quad 
H_{DS}^{n}(\afXp{p}{s})\simeq \delta_{n,0}\V_{p,s}\quad
(s=1,2,\ 1\leq r <p).
\end{align*}
Equivalently, the functors $H_{DS}^0(\text{-})$ and $H^0(\mathrm{SL}_2\times_{\mathrm{B}}\text{-})$ commute
\begin{align*}
H_{DS}^0(H^0(\mathrm{SL}_2\times_{\mathrm{B}} M))
&\simeq 
H^0(\mathrm{SL}_2\times_{\mathrm{B}} H_{DS}^0(M))
\end{align*}
for $M=\beta\gamma\otimes \V_{r,s}, \tau(\Pi[\tfrac{r}{p}]\otimes \V_{r,s})$.
\end{theorem}
\proof
(1)
By \eqref{Virasoro module through BRST reduction}, we have isomorphisms of $L(c_{1,p},0)$-modules
\begin{align}\label{isom at the level of Virasoro}
    H_{DS}^{n}(\afXpm{r}{s})\simeq \delta_{n,0}\W_{r,s}.
\end{align}
Since the $V^k(\sll_2)$-modules $\beta\gamma\otimes \V_{r,s}$ are direct sum of Wakimoto representations, we have an isomorphism 
$H_{DS}^n(\beta\gamma\otimes \V_{1,1})\simeq \delta_{n,0} V_{\sqrt{p}\mathrm{A}_1}$ of vertex algebras and $H_{DS}^n(\beta\gamma\otimes \V_{r,s})\simeq \delta_{n,0} \V_{r,s}$ as $V_{\sqrt{p}\mathrm{A}_1}$-modules by \cite{FF2}.
Hence, applying $H_{DS}^\bullet$ to 
the first row in Proposition \ref{cor: Felder complex for sl2 I},
we obtain a short exact sequence of $L(c_{1,p},0)$-modules
\begin{align}\label{cohomological Felder complex}
0\rightarrow H_{DS}^0(\afXp{r}{s})\rightarrow \V_{r,s} \rightarrow H_{DS}^0(\afXm{p-r}{3-s})(-\varpi) \rightarrow 0.
\end{align}
As the simple $L(c_{1,p},0)$-submodules appearing in $\W_{r,s}$ and $\W_{p-r,3-s}$ are different, the $L(c_{1,p},0)$-submodule $\W_{r,s}\subset \V_{r,s}$ exists uniquely. 
Hence, $H_{DS}^0(\mathrm{FT}_{p}(\sll_2))\simeq \W(p)$ as vertex algebras.
(2) The case $\afXpm{r}{s}$ for $r\neq p$ follows from \eqref{cohomological Felder complex} and $\afXp{p}{s}$ from \eqref{isom at the level of Virasoro}, respectively.
\endproof

The compatibility of modules over $\FT(\sll_2)$ and $\W(p)$ under the BRST reduction indeed extends to the Felder complex.
\begin{corollary}\label{compatibility of Felder complex}
The complex \eqref{cohomological Felder complex} is isomorphic to the Felder complex \eqref{virasoro felder complex}.
\end{corollary}
\proof
It is straightforward so show that the $\mathrm{B}$-action on \eqref{cohomological Felder complex} induced by $\mathrm{B}^{\af}$ is expressed as
\begin{align*}
    &[F^{\af}]=\left[\int Y(\beta\otimes e^{\sqrt{p}\alpha},z)dz\right]=\left[\int Y(-e^{\sqrt{p}\alpha},z)dz\right]=-[F^{\mathrm{vir}}],\\
    &[H^{\af}]|_{\pi_{\alpha_{r,2n+s}}}=-\tfrac{1}{\sqrt{p}}(h,\alpha_{1,2n+s})=H^{\mathrm{vir}}|_{\pi_{\alpha_{r,2n+s}}}.
\end{align*}
Hence, the $\mathrm{B}$-action on $\V_{r,s}$ differs from $\mathrm{B}^{\mathrm{vir}}$ by the sign $[F^{\af}]=-[F^{\mathrm{vir}}]$. 
Such a difference by signs can be eliminated by the $\mathrm{B}$-module isomorphism $(-1)^\bullet\colon M_\lambda\xrightarrow{\simeq} M_\lambda$, $(m \mapsto (-1)^{\lambda(h)/2}m)$ for weight $\mathrm{B}$-modules $M=\oplus_{\lambda}M_\lambda$ in general.
Since $(-1)^\bullet$ on $V_{\sqrt{p}A_1}$ is an automorphism of vertex algebras and extends to the $V_{\sqrt{p}A_1}$-modules $\V_{r,s}$, we may use it to fix vertical isomorphisms:
\begin{align*}
\xymatrix@=18pt{
0\ar[r] & H_{DS}^0(\afXp{r}{s}) \ar[rr]\ar[d]^-{\simeq}_{(-1)^\bullet}&& \V_{r,s} \ar[rr]^-{[Q_-^{[r]}]}\ar[d]^-{\simeq}_{(-1)^\bullet}&& H_{DS}^0(\afXm{p-r}{3-s})(-\varpi) \ar[r]\ar[d]^-{\simeq}_{(-1)^\bullet} & 0\\
0\ar[r] & \W_{r,s} \ar[rr]&& \V_{r,s} \ar[rr]^-{\mathcal{Q}_-^{[r]}}&& \W_{p-r,3-s}(-\varpi) \ar[r] & 0.
}
\end{align*}
To show that it is an isomorphism of complexes, it remains to show the commutativity of the left square. 
For this, we compute the differential $[Q_-^{[r]}]$ through the embedding $H_{DS}^0(\afXm{p-r}{3-s})\subset H_{DS}^{0}(\Pi[\tfrac{p-r}{p}]\otimes \V_{r,s})$.
By \cite[Proposition 7]{ACGY}, we have $H_{DS}^{n}(\Pi[b]\otimes \V_{r,s})$=0 for $n\neq 0$ and an isomorphism of $V_{\sqrt{p}\mathrm{A}_1}$-modules
\begin{align*}
    \V_{r,s} \xrightarrow{\simeq} H_{DS}^{0}(\Pi[b]\otimes \V_{r,s}),\quad A\mapsto [e^{(b+m)(u+v)}\otimes A]
\end{align*}
for any fixed $m\in \Z$.
Since
\begin{align*}
    [Q_-^{[r]}]
    &=\left[\int_{[\Gamma_r]}\prod_{a=1}^rY(e^{-\frac{1}{p}(u+v)}\otimes e^{-\frac{1}{\sqrt{p}}\alpha},z_a) dz\right]\\
    &=T_{-\frac{r}{p}(u+v)}\otimes \int_{[\Gamma_r]}\prod_{a=1}^rY(e^{-\frac{1}{\sqrt{p}}\alpha},z_a)dz =T_{-\frac{r}{p}(u+v)}\otimes \mathcal{Q}_-^{[r]},
\end{align*}
$[Q_-^{[r]}]$ is identified with $\mathcal{Q}_-^{[r]}\colon \V_{r,s}\rightarrow \V_{p-r,3-s}$ by the natural choice of $m$'s. 
This completes the proof.
\endproof

\begin{corollary}
There are isomorphisms of $\mathcal{W}(p)$-modules
\begin{align*}
&H^n_{\mathrm{DS}}(\mathcal{S}^a\afWb{r}{s}{b})
\simeq\delta_{a,0}\delta_{n,0}\W_{r,s}
\quad
(b\in\C),\\
&H^{n-1}_{\mathrm{DS}}(\mathcal{S}^{a+1}\afXpm{r\pm 1}{s})
\simeq 
H^{n}_{\mathrm{DS}}(\mathcal{S}^{a}\afXpm{r}{s}), \quad (1\leq r <p,~a\not=0)\\
&H^n_{\mathrm{DS}}(\afXm{0}{s})=H^n_{\mathrm{DS}}(\mathcal{S}\afXp{1}{s})=0,\quad (n\in\Z).
\end{align*}
\end{corollary}
\proof
Note that the automorphism $g$ in \eqref{isom of vertex algebra} fixes $e=e^{u+v}$. Then it follows from \eqref{pull-back of PiV} and \cite[Proposition 7]{ACGY} that we have an isomorphism of $L(c_{1,p},0)$-modules
\begin{align*}
H^n_{\mathrm{DS}}(\mathcal{S}^a\afWb{r}{s}{b})
&\simeq H^n_{\mathrm{DS}}(\mathcal{S}_{-a\mathbf{v}}\Pi[b]\otimes \W_{r,s})\\
&\simeq H^n_{\mathrm{DS}}(\Pi_a[b])\otimes \W_{r,s}\simeq \delta_{a,0}\delta_{n,0}\W_{r,s}.
\end{align*}
It is an isomorphism of vertex algebras for $(r,s)=(1,1)$ and the both-hand sides are isomorphic to $\W(p)$. Then it is straightforward to show that $H^n_{\mathrm{DS}}(\mathcal{S}^a\afWb{r}{s}{b})\simeq \delta_{a,0}\delta_{n,0}\W_{r,s}$ is indeed an isomophism of $\W(p)$-modules.
The remaining assertions follow from the first one by the long exact sequences of $H^\bullet_{\mathrm{DS}}(-)$ applied for \eqref{sequence by affine representation 1} and \eqref{sequence by affine representation 2}, respectively.
\endproof
\subsection{Simplicity}
To prove the simplicity of the $\FT(\sll_2)$-modules $\mathcal{W}^{\af}_{r,s}[b]$ and $\mathcal{X}^{\af}_{r,s}[b]$, we use the contragredient dual of modules, which we recall briefly.
Let $V$ be a vertex operator algebra with the conformal vector $L$ and $M$ be a $V$-module such that for each conformal weight $\Delta\in\C$, $\dim M_\Delta$ is finite. 
Then the contragredient dual $M^\ast_L:=\bigoplus_{\Delta}M_\Delta^\ast$ has a
$V$-module structure by
\begin{align}
\langle Y(A,z)\xi, m \rangle
:=
\langle \xi, Y(e^{zL_1}(-z^{-2})^{L_0}A,z^{-1})m\rangle
(=:\langle \xi, Y(A,z)^\dagger m\rangle).
\end{align}
The following lemma is proved straightforwardly.
\begin{lemma}\label{lemm:dual and spectral flow}\hspace{0mm}\\
\textup{(1)}
For a Heisenberg VOA $V=\pi^Q$ with $L^\mu=L_{\mathrm{sug}}+\partial\mu$,  $(\pi_\lambda^Q)^\ast_{L^\mu}\simeq \pi^{Q}_{-\lambda+2\mu}$.\\
\textup{(2)}For $V=V^k(\sll_2)$, $(L^+_k(\lambda))^\ast_{L_{\mathrm{sug}}}\simeq L_k^-(-\lambda)$.\\
\textup{(3)} Let $(V,L)$ be a conformal vertex algebra and $x(z)$ be a Heisenberg field satisfying 
\begin{align*}
    L(z)x(w)\sim \tfrac{a}{(z-w)^3}+\tfrac{x(w)}{(z-w)^2}+\tfrac{\partial x(w)}{(z-w)},
\end{align*}
such that $x_0$ acts on $V$ semisimplly with $\Z$-eigenvalues. Then for a $V$-module $M$, $(\mathcal{S}_xM)^\ast_L\simeq \mathcal{S}_{-x}(M^\ast_L)$. 
\end{lemma}
In our case $V=V^k(\sll_2)$, we always take the contragredient dual $M^*=M^*_{L_{\mathrm{sug}}}$. If $M$ is further a weight $\mathrm{B}^\af$-module, we equip $M^*$ with a $\mathrm{B}^\af$-action by $\langle h\xi,m\rangle=-\langle \xi,hm\rangle$ and $\langle f\xi,m\rangle=\langle \xi,fm\rangle$.
Note that the action of $f$ is changed by (-1) from the dual representation i.e. $\langle X\xi,m\rangle=-\langle \xi,Xm\rangle$ ($X\in \mathfrak{b}^\af$). But the these two $\mathrm{B}^\af$-module structures are isomorphic since $M$ are weight $\mathrm{B}^\af$-modules in both cases, see the proof of Corollary \ref{compatibility of Felder complex}.
\begin{proposition}\label{prop:dual of PiV}
 We have a
$\Pi[0]\otimes V_{\sqrt{p}\mathrm{A}_1}$-module isomorphism
\begin{align}
\begin{split}
\phi\colon\Pi[b]\otimes\V_{r,s}&\simeq \mathcal{S}^2(\Pi[-b]\otimes\V_{p-r,3-s})^\ast((\delta_{r,p}-1)\varpi),\\
G(u,v,\alpha)e^{b'(u+v)+\alpha_{r,s}}&\mapsto (-1)^{b'}G(u,v,\alpha)(e^{-(b'+1)(u+v)+\alpha_{p-r,1-s}})^{\ast}.
\end{split}
\end{align}
\end{proposition}
\proof
For $m',m\in \Z$ and $b'\in\Z$, set
\begin{align*}
x=m'(u+v)+m\sqrt{p}\alpha, \
y=b'(u+v)+\alpha_{r,s}, \
\phi[y]=-(b'+1)(u+v)+\alpha_{p-r,1-s}.
\end{align*}
We show that $\phi$ is a $\Pi[0]\otimes V_{\sqrt{p}\mathrm{A}_1}$-module homomorphism. 
Since $\phi$ is a $\pi^{u,v,\alpha}$-module homomorphism by construction, it suffices to show 
\begin{align}\label{align: phi is a hom of Pi0V}
\phi(e^x(z)G(u,v,\alpha)e^y)=e^x(z)\phi(G(u,v,\alpha)e^y).
\end{align}
We compute the right-hand side.
By applying Lemma \ref{lemm:dual and spectral flow} \textup{(1)} to our case $\mu=-\frac{1}{2}(u-v+\sqrt{p}\alpha)$, we have 
$\langle \xi,x_{(n)}v\rangle=-\langle x_{(-n)}\xi, v\rangle$
for $n\not =0$ and $\xi\in\mathcal{S}^2(\Pi[-b]\otimes\V_{p-r,3-s})^\ast$.
It follows that
\begin{align}\label{align: move Heisenberg action to dual}
\begin{split}
\langle e^x(z)\xi, v\rangle
&=(-1)^{\Delta}z^{(h,x)-2\Delta}\langle \xi, T_xz^{-(x,-)}E^+(x,z^{-1})E^-(x,z^{-1})v\rangle\\
&=(-1)^{\Delta}z^{(h,x)-2\Delta}\langle E^+(x,z)E^-(x,z)\xi, T_xz^{-(x,-)}v\rangle,
\end{split}
\end{align}
where $\Delta$ is the conformal weight \eqref{align: conformal weight of lattice vectors} of $e^x$.  
Thus,
\begin{align*}
&\langle e^x(z)\phi(G(u,v,\alpha)e^y), v\rangle\\
&=(-1)^{b'+m'}z^{-2m-2p(m^2-m)}\langle E^+(x,z)E^-(x,z)\xi, T_xz^{-(x,\phi[y]-x)}v\rangle\\
&=\langle (-1)^{b'+m'}z^{(x,y)} E^+(x,z)E^-(x,z)G(u,v,\alpha)(e^{\phi[y]-x})^\ast,v\rangle.
\end{align*}
As $\phi[x+y]=\phi[y]-x$, it coincides with the left-hand side of \eqref{align: phi is a hom of Pi0V}.

Next, we show that $\phi$ is a ${\mathrm{B}}^{\mathrm{aff}}$-module homomorphism. 
Since it is clear that the ${\mathrm{H}}^{\af}$-action commutes with $\phi$, it remains to show that ${\mathrm{N}}^{\af}$-action does:
\begin{align}
\phi(fG(u,v,\alpha)e^{y})=f\phi(G(u,v,\alpha)e^{y}).
\label{align: N- commutes with phi}
\end{align}
As the left hand side is easy, let us compute the right hand side.
We have
\begin{align*}
\langle f\xi, v\rangle
=\langle \xi, fv\rangle
\stackrel{\eqref{align: move Heisenberg action to dual}}
{=}\int dz\langle T_{-x}z^{-(x,\phi[y]-x)}E^+(x,z^{-1})E^-(x,z^{-1})\xi, v\rangle,
\end{align*}
where $x=u+v+\sqrt{p}\alpha$
and the first equality follows from the cancellation of $(-1)$s by $(-1)^{\bullet}$.
By changing the variable $w=z^{-1}$
we have
\begin{align*}
f\phi(G(u,v,\alpha)e^{y})=\int dw (-1)^{b'+1}G(u,v,\alpha)E_+(x,w)E_-(x,w)(e^{-(b'+2)(u+v)+\alpha_{p-r,2-s}})^\ast.
\end{align*}
It coincides with the left hand side of \eqref{align: N- commutes with phi}.
\endproof

\begin{corollary}\label{prop: contragredient dual}
We have isomorphisms of 
$\mathrm{FT}_p(\sll_2)$-modules
\begin{align*}
&\afWb{r}{s}{b}\simeq \mathcal{S}^2\left(\afWb{r}{s}{-b+\tfrac{r}{p}}\right)^\ast\quad (b\in\C,~1\leq r\leq p),\\
&\afXp{r}{s}{}^\ast\simeq \mathcal{S}^{-1}\afXm{r-1}{s},\quad
\afXm{r}{s}{}^\ast\simeq\mathcal{S}^{-1}\afXp{r+1}{s}\quad (1\leq r<p).
\end{align*}
\end{corollary}
\proof
The case $\afWb{p}{s}{b}=\Pi_0[b]\otimes\mathcal{V}_{p,s}$ is Proposition \ref{prop:dual of PiV}.
For $r\neq p$, we have
\begin{align}
\begin{split}\label{analog of Sug2 Cor3.3}
\afWb{r}{s}{b}
\stackrel{\text{Theorem \ref{FT(sl2,0)}}}{\simeq}\hspace{2mm}
&H^0(\mathrm{SL}_2\times_{\mathrm{B}}(\Pi[b]\otimes\mathcal{V}_{r,s}))\\
\stackrel{\text{Corollary \ref{cor: FT(sl2,0)}}}{\simeq}\hspace{2mm}
&H^1(\mathrm{SL}_2\times_{\mathrm{B}}(\Pi[b+\tfrac{p-r}{p}]\otimes\mathcal{V}_{p-r,3-s})(-\varpi))\\
\stackrel{\text{Proposition \ref{prop:dual of PiV}}}{\simeq}
&H^1(\mathrm{SL}_2\times_{\mathrm{B}}\mathcal{S}^2(\Pi[-b+\tfrac{r}{p}]\otimes\mathcal{V}_{r,s})^{\ast}(-2\varpi))\\
\simeq\hspace{8.5mm}
&\mathcal{S}^2H^1(\mathrm{SL}_2\times_{\mathrm{B}} (\Pi[-b+\tfrac{r}{p}]\otimes \mathcal{V}_{r,s})^{\ast}(-2\varpi))\\
\stackrel{\text{Serre duality}}{\simeq}\hspace{2mm}
&\mathcal{S}^2H^0(\mathrm{SL}_2\times_{\mathrm{B}} (\Pi[-b+\tfrac{r}{p}]\otimes\mathcal{V}_{r,s}))^{\ast}\\
\stackrel{\text{Theorem \ref{FT(sl2,0)}}}{\simeq}\hspace{2mm}
&\mathcal{S}^2\left(\afWb{r}{s}{-b+\tfrac{r}{p}}\right)^\ast,
\end{split}
\end{align}
see \cite[Corollary 3.3]{Sug2} for a similar proof for $\W(p)$.
The remaining cases follow from Lemma \ref{lemm:dual and spectral flow} \textup{(2)}, \textup{(3)} by applying the contragredient dual to \eqref{sequence by affine representation 1} and \eqref{sequence by affine representation 2}.
\endproof

\begin{corollary}\label{dual of betagamma}
For $1\leq r<p$,
we have isomorphisms of 
$\mathrm{FT}_p(\sll_2)$-modules
\begin{align*}
\beta\gamma\otimes \V_{r,s}&\simeq \mathcal{S}^2(\mathcal{S}_{-v}(\beta\gamma)\otimes \V_{p-r,3-s})^\ast
(-\varpi),
\\
\tau(\Pi[\tfrac{r}{p}]\otimes\mathcal{V}_{r,s})&\simeq\mathcal{S}^2N_{p-r,3-s}^\ast
(-\varpi).
\end{align*}
\end{corollary}
\proof
We show the first case. Note that $N':=\phi^{-1}(\mathcal{S}^2(S_{-v}(\beta\gamma)\otimes \V_{p-r,3-s})^*)$ is a $\FT(\sll_2)$-submodule of $\Pi[0]\otimes \V_{r,s}$ by Proposition \ref{prop:dual of PiV}. 
We apply $\mathcal{S}^2(\text{-})^*$ to the horizontal sequence in the bottom the diagram of Proposition \ref{cor: Felder complex for sl2 I} and change the index $(r,s)\mapsto (p-r,3-s)$. Then $N'$ satisfies the short exact sequence 
$$0 \rightarrow \afXp{r}{s} \rightarrow N' \rightarrow \afXm{p-r}{3-s} \rightarrow 0. $$
By Proposition \ref{cor: Felder complex for sl2 I} again, such an extension exists uniquely inside $\Pi[0]\otimes \V_{r,s}$, namely, $\beta\gamma\otimes \V_{r,s}$, we conclude $N=N'$ and thus the first assertion.
One can show the second case similarly.
\endproof
\begin{theorem}
\label{simplicity of X}
The $\mathrm{FT}_p(\sll_2)$-modules $\afWb{r}{s}{b}$ $([b]\not=[0],[\tfrac{r}{p}], 1\leq r\leq p)$, $\afXp{r}{s}$ $(1\leq r\leq p)$, and $\afXm{r}{s}$ $(0\leq r< p)$ are simple.
\end{theorem}
We note that the cases $\afXp{1}{s}$ ($s=1,2$) are proved in \cite[Proposition 3]{ACGY}.
\proof
As the proof is similar, we omit the case $\afWb{r}{s}{b}$.
By Corollary \ref{prop: contragredient dual}, it suffices to consider the case $[b]=[0]$. 
Let $\genafXp{1}$ denote the generalized vertex operator algebra $H^0({\mathrm{SL}_2}\times_{\mathrm{B}}\beta\gamma\otimes(\V_{1,1}\oplus\V_{1,2}))\simeq
\afXp{1}{1}\oplus\afXp{1}{2}$ and $\genafXp{r}$ the $\genafXp{1}$-module $H^0({\mathrm{SL}_2}\times_{\mathrm{B}}\beta\gamma\otimes(\V_{r,1}\oplus\V_{r,2}))\simeq
\afXp{r}{1}\oplus\afXp{r}{2}$.
By \eqref{shape of (co)singular vectors}, $\genafXp{r}$ is generated by $\mathbf{1}\otimes e^{\alpha_{r,1}}$ as $\genafXp{1}$-module.
Furthermore, $\genafXp{r}$ is simple as $\genafXp{1}$-module.
In fact, Corollory \ref{prop: contragredient dual} implies that $\genafXp{r}{}^\ast$ is also generated by $(\mathbf{1}\otimes e^{\alpha_{r,1}})^\ast$, and hence for any $m\in \genafXp{r}$, there exists $F\in U(\genafXp{1})$ such that
\begin{align*}
0
\not=\langle m^\ast,m\rangle
=\langle F(\mathbf{1}\otimes e^{\alpha_{r,1}})^\ast, m\rangle
=\langle(\mathbf{1}\otimes e^{\alpha_{r,1}})^\ast, F^{\dagger} m\rangle.
\end{align*}
Let $M$ be a nonzero $\mathrm{FT}_p(\sll_2)$-submodule of $\afXp{r}{s}$.
In the same manner as above, $\afXp{r}{s}$ is generated by $\Omega_s:=L^+(s-1)\otimes(\mathbf{1}\otimes e^{\alpha_{r,s}})$ as $\mathrm{FT}_p(\sll_2)$-module and we may take a nonzero $m\in M\cap \Omega_s$.
By \cite[Corollary 4.2]{DM} for $\genafXp{1}$ and the fact that $h(x)-h(y)\in2\Z$ for $x,y\in L^+(s-1)$, we have $U(\mathrm{FT}_p(\sll_2))m\supseteq\Omega_s$.
\endproof

\begin{corollary}
The Loewy diagrams of $\Pi[b]\otimes\V_{r,s}$ $([b]=[0],[\tfrac{r}{p}],~1\leq r\leq p)$ are given by Figure \ref{fig:Loewy diagram of PiV}.
\end{corollary}
\proof
Let us recall the notation used in the discussion just before Proposition \ref{cor: Felder complex for sl2 I}.
It suffices to show that $M_1$ is a $\FT(\sll_2)$-module and $M_1/M_3$ splits.
For $m\in M_1$ and $a\in \mathrm{FT}_p(\sll_2)$, we have
\begin{align*}
f_{1}(a_{(n)}m)=(f_{0}a)_{(n+1)}m+a_{(n)}f_{1}m.
\end{align*}
As $L^+_k(n\varpi)\in \operatorname{KL}_{k}(\sll_2)$ (in particular, locally nilpotent with respect to $f_{0}$) by Proposition \ref{decomposition of weight modules}, we have $f_{1}^{\gg 0}(a_{(n)}m)=0$ and the first assertion is true.

Let us show that $M_1/M_3$ (and $N_1^{s+2n}/N_3^{s+2n}$) split.
From the proof of Proposition \ref{decomposition of weight modules} \textup{(2)}-(ii), (iii), $N_1^{s+2n}/N_2^{s+2n}$ and $N_2^{s+2n}/N_3^{s+2n}$ are generated by $y_{s+2n}$ and $v_{s+2n}$ as $V^k(\sll_2)$-modules, respectively.
However, $M_1/M_2$ is generated by $Q_+^jv_{s+2n}$ as $V^k(\sll_2)$-modules by \eqref{lemm lower bound of $M_i/M_{i+1}$}.
If $M_1/M_3$ (resp. $N_1^{s+2n}/N_3^{s+2n}$) does not split, then it immediately leads the contradiction $N_1^{s+2n}=N_2^{s+2n}$ (resp. $M_1=M_2$).
\endproof

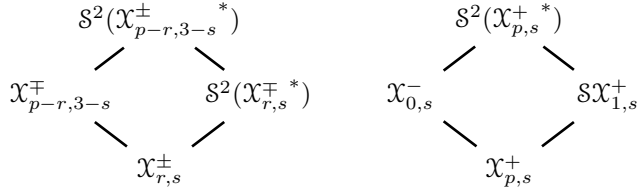
\begin{figure}
\centering
\begin{tikzpicture}[x=2mm,y=2mm]

\node (B3) at (6.5,0) {$\afXpm{r}{s}$};
\node (B2) at (0,5) {$\afXmp{p-r}{3-s}$};
\node (B1) at (13,5) {$\mathcal{S}^2({\afXmp{r}{s}}^\ast)$};
\node (B0) at (6.5,10) {$\mathcal{S}^2({\afXpm{p-r}{3-s}}^\ast)$};
\draw [line width=1pt] (B0) -- (B1);
\draw [line width=1pt] (B2) -- (B3);
\draw [line width=1pt] (B1) -- (B3);
\draw [line width=1pt] (B2) -- (B0);

\node (D3) at (29.5,0) {$\afXp{p}{s}$};
\node (D2) at (23,5) {$\afXm{0}{s}$};
\node (D1) at (36,5) {$\mathcal{S}\afXp{1}{s}$};
\node (D0) at (29.5,10) {$\mathcal{S}^2({\afXp{p}{s}}^\ast)$};
\draw [line width=1pt] (D0) -- (D1);
\draw [line width=1pt] (D2) -- (D3);
\draw [line width=1pt] (D1) -- (D3);
\draw [line width=1pt] (D2) -- (D0);
\end{tikzpicture}
\caption{Loewy diagrams of 
$\Pi[b]\otimes\mathcal{V}_{r,s}$ ($[b]=[0],[\tfrac{r}{p}]$).}
\label{fig:Loewy diagram of PiV}
\end{figure}

\section{Associated variety}\label{section: associated variety}
\subsection{Nilpotent cone}
Given a vertex algebra $V$, the Zhu's $C_2$-algebra is defined as the quotient $R_V=V/C_2(V)$ with $C_2(V)=\mathrm{span}\{a_{(-2)}b\mid a,b\in V \}$, which is a Poisson algebra by 
$$a\cdot b=a_{(-1)}b,\quad \{a,b\}=a_{(0)}b.$$
The reduced scheme $X_V=\mathrm{Specm}(R_V)$ is called the associated variety of $V$. 
Note that we may extend the definition to the super-setting. 
For $V=V^k(\sll_2)$, $R_{V^(\sll_2)}$ is naturally identified as Poisson algebras with the symmetric algebra $\mathrm{S}(\sll_2)$ equipped with Kostant--Kirillov Poisson bracket defined by $\{a,b\}=[a,b]$ ($a,b\in \sll_2$) and thus $X_{V^k(\sll_2)}\simeq \sll_2$ by using the normalized invariant form. 
At $X=x e+ y f+ z h\in \sll_2$, we have $e(X)=y$, $h(X)=2z$, $f(X)=x$, respectively. Then 
the zero locus of the Casimir element $\Omega:=h^2+4ef\in \mathrm{S}(\sll_2)$, namely the closed subvariety $\{z^2+xy=0\}$ is the nilpotent cone 
\begin{align*}
\mathcal{N}:=\{X\in \sll_2\mid \mathrm{ad}_X\colon \sll_2\rightarrow \sll_2\ \text{is nilpotent}\}\subset \sll_2.
\end{align*}
Given a homomophism $\eta\colon V \rightarrow W$ of vertex algebras, we have induced maps
\begin{align*}
    \overline{\eta}\colon R_V\rightarrow R_W,\quad \eta^\sharp \colon X_W\rightarrow X_V.
\end{align*}
for their Zhu's $C_2$-algebras and associated varieties.

\begin{theorem}\label{theorem: associated variety}\hspace{0mm}\\
\textup{(1)} 
The Feigin--Tipunin algebra $\FT(\sll_2)$ is strongly generated by 
\begin{align}\label{strong generators}
    e,\ h,\ f,\ x_{ij}=f_0^iQ_+^je^{-\sqrt{p}\alpha},\quad (0\leq i,j \leq 2).
\end{align}
Moreover, $x_{ij}$'s are all nilpotent in $R_{\FT(\sll_2)}$.\\
\textup{(2)} The embedding $\iota \colon V^k(\sll_2)\hookrightarrow \FT(\sll_2)$ induces an isomorphism 
\begin{align*}
    \iota^\sharp\colon X_{\mathrm{FT}_p(\sll_2)}\simeq \mathcal{N}\subset \sll_2.
\end{align*}
Therefore, $\FT(\sll_2)$ is quasi-lisse.
\end{theorem}
\proof
(1)
Note that by Corollary \ref{Property of Weyl modules} (2), $\FT(\sll_2)$ is strongly generated by the union of the basis of the subspaces
\begin{align*}
    \g\subset V^k(\sll_2),\quad L(n\alpha)\otimes L(n\alpha)\subset L(n\alpha)\otimes \mathbb{V}^k(n\alpha),\quad (n\geq 1).
\end{align*}
and that $L(n\alpha)\otimes L(n\alpha)$ is spanned by $x_{n,ij}=f_0^iQ_+^je^{-n\sqrt{p}\alpha}$ ($0\leq i,j\leq 2n$). 
Recall the Leibnitz rule of the $0^{\mathrm{th}}$ mode 
$$A_{(0)}(a_{(n)}b)=(A_{(0)}a)_{(n)}b+a_{(n)}(A_{(0)}b),$$
for fields $A(z)=\sum_{n\in \Z}A_{(n)}z^{-n-1}$ in general. 
Then 
\begin{align}\label{inductive construction}
   &x_{n+1,00}= x_{1,00}{}_{(-2pn-1)}x_{n,00}
\end{align}
implies by induction that $x_{n,ij}$ are expressed by sums of nomally ordered products of $x_{1,ij}$ ($0\leq i,j\leq 2$) and their derivatives. Hence, $\FT(\sll_2)$ is strongly generated by \eqref{strong generators}.
Next, we show that $x_{ij}$ are all nilpotent in $R_{\FT(\sll_2)}$. By \eqref{inductive construction},  
\begin{align}\label{fundamental nilpotency}
    x_{00}{}_{(-1)}x_{00}=0
\end{align}
in $\FT(\sll_2)$ and thus $x_{00}^2=0$ in $R_{\FT(\sll_2)}$. Hence, $x_{ij}$ are also nilpotent by Lemma \ref{nilpotency check} below.
(2) By the isomorphisms of Poisson algebras  
\begin{align*}
    R_{V^k(\sll_2)}\simeq \mathrm{S}(\sll_2),\quad R_{\beta\gamma}\simeq \C[\beta,\gamma],\quad R_{\pi^\alpha}\simeq \C[\alpha],
\end{align*}
where the Poisson bracket on $R_{\beta\gamma}$ is determined by $\{\beta,\gamma\}=1$, 
the Wakikimoto realization \eqref{sl2 wakimoto} induces the homomorphism
\begin{align}\label{C2 hom}
\begin{array}{cccc}
   \overline{\mu}_{\mathrm{Wak}}\colon  & \mathrm{S}(\sll_2) & \rightarrow & \C[\beta,\gamma,\alpha]\\
     & e &\mapsto & \beta\\
     & h &\mapsto & -2\beta\gamma-\tfrac{1}{\sqrt{p}}\alpha\\
     & f &\mapsto & -\beta\gamma^2-\tfrac{1}{\sqrt{p}}\gamma\alpha.
\end{array}
\end{align}
Since it is injective \cite[Proposition 5.2]{F}, we identify the elements of $\mathrm{S}(\sll_2)$ with their image.
The elementes $x_{i1}\in \FT(\sll_2)\cap \beta\gamma\otimes \pi^\alpha$ are expressed as 
\begin{align*}
    &x_{01}=\sum_{n+m=2p-1}\beta_{-n-1}\otimes s_m(\sqrt{p}\alpha)\\
    &x_{11}=s_{2p}(\sqrt{p}\alpha)+\sum_{n+m=2p-1}\beta_{-n-1}\gamma_0\otimes s_m(\sqrt{p}\alpha)\\
    &x_{21}=2\gamma_0\otimes s_{2p}(\sqrt{p}\alpha)+\sum_{n+m=2p-1}\beta_{-n-1}\gamma_0^2\otimes s_m(\sqrt{p}\alpha),
\end{align*}
where the summation is taken over $n,m\geq0$ and $s_m(x)$ is the Schur polynomial in \eqref{def of Epm}.
Then it follows that $x_{01},x_{11},x_{21}$ in $R_{\beta\gamma\otimes \pi^\alpha}$ are 
\begin{align}\label{align: explicit xi1}
   x_{01}= \varepsilon e \alpha^{2p-1},\quad x_{11}= -\varepsilon h \alpha^{2p-1},\quad x_{21}= -2\varepsilon f \alpha^{2p-1}, 
\end{align}
where $\varepsilon=\tfrac{p^{(2p-1)/2}}{(2p-1)!}$. 
By applying $Q_+^3$ to \eqref{fundamental nilpotency}, we obtain 
$$Q_+(x_{01}{}_{(-1)}x_{01})=0$$
and thus $x_{01}{}_{(-1)}x_{01}\in V^k(\sll_2)$ by Corollary \ref{Property of Weyl modules}. 
By applying the derivation $f_0$ successively, we obtain  
\begin{align}\label{align: nilpotent elements in S(Sl2)}
    x_{01}^2,\ x_{01}x_{11},\ x_{11}^2+x_{01}x_{21},\ x_{11}x_{21},\ x_{21}^2
\end{align}
which lie in $V^k(\sll_2)\subset \FT(\sll_2)$ and are nilpotent by (1).
By using the Gr\"{o}bner basis, we find that the elements \eqref{align: nilpotent elements in S(Sl2)} in $R_{V^k(\sll_2)}$ generate the ideal
$$I=(\alpha^{4p},\alpha^{4p-1}\beta, \alpha^{4p-2} \beta^2 ).$$
Since $\alpha^2=p\ \Omega$, it follows that $\iota^\sharp\colon X_{\FT(\sll_2)}\hookrightarrow \mathcal{N}$.
As $X_{\FT(\sll_2)}$ is stable under the Adjoint action of $\mathrm{SL}_2$, we have either $X_{\FT(\sll_2)}\simeq \mathcal{N}$ or $X_{\FT(\sll_2)}=\{0\}$. 
If $X_{\FT(\sll_2)}=\{0\}$, then $\FT(\sll_2)$ were $C_2$-cofinite and thus the simple modules are finitely many, a contradiction to Theorem \ref{simplicity of X}. Hence, $X_{\FT(\sll_2)}\simeq \mathcal{N}$.
Since, $\mathcal{N}$ has only finitely many symplectic leaves, namely, two nilpotent orbits, $\FT(\sll_2)$ is quasi-lisse \cite{AK} by definition.
\endproof
\begin{lemma}\label{nilpotency check}
Let $R$ be a commutative ring and $D$ a derivation. If $a\in R$ is nilpotent, then so is $Da$. 
\end{lemma}
\proof Let $a^N=0$. Since $0=D^N(a^N)=N! (Da)^N+ab$ for some $b\in R$, $(Da)^{N^2}=(-\frac{1}{N!}ab)^N=0$.
\endproof 
\begin{remark}
One can show that 
\begin{align*}
\begin{array}{lll}
(X_{00})^2&(X_{01})^4&(X_{02})^2\\
(X_{10})^4&(X_{11})^{24}&(X_{12})^4\\
(X_{20})^2&(X_{21})^4&(X_{22})^2
\end{array}
\end{align*}
are all zero in $R_{\FT(\sll_2)}$ by using \eqref{fundamental nilpotency} and $D^3X_{00}=0$ for $D=f_0, Q_+$.
\end{remark}
\begin{remark}
\textup{
The embedding $\mu\colon \FT(\sll_2)\hookrightarrow \beta\gamma\otimes V_{\sqrt{p}A_1}$ induces
\begin{align*}
   \mu^\sharp\colon \C^2\rightarrow \mathcal{N}\subset \sll_2,\quad (a,b)\mapsto \left(\begin{array}{cc} -ab&-ab^2\\ a &ab \end{array} \right)
\end{align*}
for their associated varieties. 
Since $\mu^\sharp$ restricts to an isomorphism $\C^2\supset \{a\neq 0\}\simeq \{z \neq 0\}\subset \mathcal{N}$, $\mu^\sharp$ is birational and thus dominant. It provides a non-semisimple example of a conjectured by Arakawa--van Ekeren--Moreau \cite[Conjecture 1.3]{AvEM} in the semisimple setting.}
\end{remark}
By Corollary \ref{Property of Weyl modules}, $\mathrm{FT}_p(\sll_2)$ is a vertex algebra object in the Kazhdan--Lusztig category $\mathrm{KL}_k$. Then Theorem \ref{compatibility of FT algeras for sl2} implies the isomorphism 
\begin{align}\label{assocaited variety of triplet}
    X_{\W(p)}\simeq X_{\mathrm{FT}_p(\sll_2)}\underset{\sll_2}{\times}\mathscr{S}_f\simeq \{\mathrm{pt}\}
\end{align}
by \cite[Theorem 6.1.]{Ar2} where $\mathscr{S}_f=f+\C e\subset \sll_2$ denotes the Slodowy slice. 
By Theorem \ref{theorem: associated variety}, it follows that $\W(p)$ is strongly generated by $\omega_{1,p}$ and the cohomology classes associated with $x_{ij}$ under the BRST reduction. Thus $\W(p)$ is also finitely strongly generated \cite{AM}. Hence \eqref{assocaited variety of triplet} implies the following.
\begin{corollary}[\cite{AM}]\label{C2 cofiniteness of triplet}
The triplet algebra $\mathcal{W}(p)$ is $C_2$-cofinite.
\end{corollary}

\subsection{Relative semi-infinite cohomology}
Corollary \ref{C2 cofiniteness of triplet} also follows from a conjecture on relative semi-infinite cohomology \cite{Fe,FGZ}, which is of independent interest.  
    Let $V$, $W$ be simple vertex superalgebras of CFT type which contains the rank one Heisenberg vertex algebra of opposite levels, namely,
    \begin{align}\label{decomposition into Focks}
      V\simeq \bigoplus_{n\in \Z}V_n\otimes \pi^{A}_n,\quad W\simeq \bigoplus_{n\in \Z}W_n\otimes \pi^{A^\dagger}_n.  
    \end{align}
Here $A(z)$, $A^\dagger(z)$ are Heisenberg fields with OPEs
\begin{align*}
    A(z)A(w)\sim \frac{a}{(z-w)^2},\quad A^\dagger(z)A^\dagger(w)\sim \frac{-a}{(z-w)^2}
\end{align*}
for some $a\in \C^\times$ and $V_n$, $W_n$ are multiplicity spaces, which are naturally modules over the Heisenberg coset $\Com(\pi^A,V)$ and $\Com(\pi^{A^\dagger},W)$, respectively.
As the diagonal Heisenberg field $A(z)+A^\dagger(z)$ on $V\otimes W$ is commutative, we may take the relative semi-infinite cohomology 
$$V\circ W:=\mathrm{H}_{\mathrm{rel}}^0(\gl_1,V\otimes W)\simeq \bigoplus_{n+m=0}V_n\otimes W_m,$$
which is an extension of $\Com(\pi^A\otimes \pi^{A^\dagger},V\otimes W)$.
Motivated by the results \cite{Ar2} for semisimple Lie algebras, we expect the following. 
\begin{conjecture}\label{rel semi-infinite conjecture}
If $V$, $W$ are strognly finitely generated, then so is $V\circ W$. 
The associated (super)variety $X_{V\circ W}$ of $V\circ W$ is isomorphic to the Hamiltonian reduction 
$$X_{V\circ W}\simeq (X_V\times X_W)/\!/\!/\mathrm{GL}_1$$
of $X_{V\otimes W}\simeq X_V\times X_W$ for the moment map 
$X_V,X_W\rightarrow \gl_1^*$ induced by the embedding $\pi\hookrightarrow V,W$. In particular the super-dimension of $X_{V\circ W}$ is
\begin{align}\label{expected dimension}
    \mathrm{sdim}X_{V\circ W}=\mathrm{sdim}X_{V}+\mathrm{sdim}X_{W}-(2,0).
\end{align}
\end{conjecture}
\begin{example}
\textup{
    By \cite{CGNS}, the $\mathcal{N}=2$ superconformal algebra $\W^\ell(\sll_{2|1})$ can be constructed from $V^k(\sll_2)$ as 
    \begin{align*}
        \W^\ell(\sll_{2|1})\simeq \mathrm{H}_{\mathrm{rel}}^0(\gl_1,V^k(\sll_2)\otimes V_\Z\otimes \pi),
    \end{align*}
    under the relation $(k+2)(\ell+1)=1$. It is a reformulation of Kazama--Suzuki coset construction \cite{KaSu,DVPYZ}
    \begin{align}\label{Kazama-Suzuki}
        \W^\ell(\sll_{2|1})\simeq \Com(\pi,V^k(\sll_2)\otimes V_\Z).
    \end{align}
In this case, we have 
\begin{align*}
    X_{V^k(\sll_2)}\simeq \sll_2,\quad X_{\W^\ell(\sll_{2|1})}\simeq \mathscr{S}_{f}=\left\{\left(\begin{array}{cc|c}a & b & c \\ 1 & a &0 \\\hline 0 & d &2a \\ \end{array} \right) \right\}\subset \sll_{2|1}.
\end{align*}
Then it is straightforward to show 
$$(X_{V^k(\sll_2)}\times X_{V_\Z\otimes \pi})/\!/\!/\mathrm{GL}_1\simeq (\sll_2\times \C^{1|2})/\!/\!/\mathrm{GL}_1\simeq \C^{2|2}\simeq \mathscr{S}_{f}$$
as Poisson (super)varieties.}
\end{example}
Similarly, the conjecture recovers Corollary \ref{C2 cofiniteness of triplet}.
\begin{example}
\textup{By \cite{ACGY}, the simple subregular $\W$-algebra $\W_\ell(\sll_{p-1},f_{\mathrm{sub}})$ at the boundary admissible level $\ell=-(p-1)+\frac{p-1}{p}$ has a realization 
    $$\W_\ell(\sll_{p-1},f_{\mathrm{sub}})\simeq (\W(p)\otimes V_{\sqrt{-2p}\Z})^{\mathrm{U}(1)}.$$
    Then $\W(p)$ can be reconstructed from $\W_\ell(\sll_{p-1},f_{\mathrm{sub}})$ as 
    \begin{align*}
        \W(p)\simeq \mathrm{H}_{\mathrm{rel}}^0(\gl_1,\W_\ell(\sll_{p-1},f_{\mathrm{sub}})\otimes V_{\sqrt{2p}\Z}).
    \end{align*}
    By \cite{Ar3,AvE2}, the associated variety $X_{\W_\ell(\sll_{p-1},f_{\mathrm{sub}})}$ is the nilpotent Slodowy slice
$$ \mathcal{N}\cap \mathscr{S}_{f_{\mathrm{sub}}}\simeq \{(x,y,z)\in \C^3\mid xy+z^{p-1}=0\}\subset \sll_{p-1}.$$
Since $V_{\sqrt{2p}\Z}$ is $C_2$-cofinite and thus $X_{V_{\sqrt{2p}\Z}}=\{\mathrm{pt}\}$, we obtain 
\begin{align}\label{computation of associated varieties}
(X_{\W_\ell(\sll_{p-1},f_{\mathrm{sub}})}\times X_{V_{\sqrt{2p}\Z}})/\!/\!/\mathrm{GL}_1\simeq (\mathcal{N}\cap \mathscr{S}_{f_{\mathrm{sub}}} \times \{\mathrm{pt}\})/\!/\!/\mathrm{GL}_1\simeq \{\mathrm{pt}\}.    
\end{align}
Hence, Conjecture \ref{rel semi-infinite conjecture} says that $\W(p)$ is strongly finitely generated and thus $C_2$-cofinite by \eqref{computation of associated varieties}.}
\end{example}

Conjecture \ref{rel semi-infinite conjecture} produces a concrete one. 
Let $L_k(\sll_n)$ be the simple affine vertex algebra for $\sll_n$ at the admissible level $k=-n+\frac{p}{q}$.
Then $X_{L_k(\sll_n)}$ is a certain nilpotent orbit closure $\overline{\mathbb{O}}_q$ depending on $q$ \cite{Ar3}. If $\overline{\mathbb{O}}_q$ is the minimal nilpotent orbit closure, then $\dim\overline{\mathbb{O}}_q=2(n-1)=2 \dim \h$ \cite{W}. On the other hand, when $k<0$, the decomposition \eqref{decomposition into Focks}: $L_k(\sll_n)\simeq \bigoplus_{\lambda\in Q} L_k(\sll_n)_\lambda\otimes \pi^k_{\h,\lambda}$ gives rise to a negative-definite rational lattice $V_{\ssqrt{1/k}Q}\simeq \bigoplus_{\lambda\in Q} \pi^k_{\h,\lambda}$. Therefore, by using $V_{\ssqrt{-1/k}N Q}$ and the successive $\mathrm{H}_{\mathrm{rel}}^0(\gl_1,\text{-})$ for $\gl_1\subset \h$, we obtain $V:=\mathrm{H}_{\mathrm{rel}}^0(\h,L_k(\sll_n)\otimes V_{\ssqrt{-1/k}N Q})$, which is in the setting of Conjecture \ref{rel semi-infinite conjecture}. Then we expect that $\dim X_V=0$ by \eqref{expected dimension} and thus $X_V=\{\mathrm{pt}\}$.
\begin{conjecture}
The following vertex algebras are $C_2$-cofinite:\\
\textup{(1)}
$\mathrm{H}_{\mathrm{rel}}^0(\h,L_k(\sll_2)\otimes V_{\sqrt{q(2q-p)}A_1})$ $(2q> p \geq 2,\ q\geq 2)$,\\
\textup{(2)}
$\mathrm{H}_{\mathrm{rel}}^0(\h,L_k(\sll_3)\otimes V_{\sqrt{2(6-p)}A_2})$ $(6> p \geq 3,\  q=3)$.
\end{conjecture}

\section{The case $p=1$}\label{section: the case p=1}
We derive screening operators for $\mathrm{FT}_p(\sll_2)$ in the case $p=1$ to complete the excluded case in Theorem \ref{identificatio of affine FT}.
By Theorem \ref{FT(sl2,0)} and Corollary \ref{Property of Weyl modules}, we have
\begin{align}\label{p=1 sl2 FT algebra}
\mathrm{FT}_{1}(\sll_2)\simeq \bigoplus_{n\geq0}\C^{2n+1}\otimes \mathbb{V}^{-1}(n\alpha)\subset \beta\gamma\otimes V_{A_1} \subset \Pi[0]\otimes V_{A_1}
\end{align}
as $(\mathrm{SL}_2,V^{-1}(\sll_2))$-bimodules.
By setting $u(z)=\frac{1}{\sqrt{2}}\alpha(z)$, $\mathcal{M}(2)$ is defined as
\begin{align}\label{p=2 singlet}
    \mathcal{M}(2)=\mathrm{Ker}_{\pi^u}\int Y(e^{-u},z)dz
\end{align}
and is strongly generated by 
\begin{align*}
    L:=\omega_{1,2}=\tfrac{1}{2}u^2+\tfrac{1}{2}\partial u,\quad W=\tfrac{1}{3}u^3+\tfrac{1}{2}u \partial u +\tfrac{1}{12}\partial^2 u
\end{align*}
(\cite{A2}). 
The element $W$ is primary of conformal weight $3$ and satisfies the OPE
\begin{align*}
    W(z)W(w)\sim& \frac{-1}{(z-w)^6}+\frac{3L(w)}{(z-w)^4}+\frac{\tfrac{3}{2}\partial L(w)}{(z-w)^3}\\
    &+\frac{-\frac{3}{4}\partial^2L(w)+4:L(w)^2:}{(z-w)^2}+\frac{-\tfrac{1}{6}\partial^3 L(w)+4 :\partial L(w)\cdot L(w):}{(z-w)}.
\end{align*}
We change the basis of the Cartan subspace of $\Pi[0]\otimes V_{A_1}$ as 
\begin{align*}
    u_1=-u,\quad u_2=v+\alpha,\quad h=-2v-\alpha.
\end{align*}
Then we have an isomorphism 
\begin{align}\label{free field decomposition}
    \Pi[0]\otimes V_{A_1}\simeq \bigoplus_{n+m\equiv 0} \pi^{u_1}_{-n}\otimes \pi^{u_2}_{-m}\otimes \pi^{h}_{(n+m)\varpi}
\end{align}
where the summation runs over $n,m\in \Z$ with $n+m\equiv 0$ modulo $2$.
It is straightforward to see we have the sequence of embeddings
\begin{align*}
    \mathcal{M}(2)\otimes \mathcal{M}(2)\otimes \pi^h\subset \mathrm{FT}_{1}(\sll_2)\subset \Pi[0]\otimes V_{A_1}.
\end{align*}
As elements in $\mathrm{FT}_{1}(\sll_2)$, the generators $L_i,\ W_i$ ($i=1,2$) of the $i$-th factor of $\mathcal{M}(2)$ are expressed as
\begin{align*}
    &L_1=\tfrac{1}{2}\left(\tfrac{1}{2}h^2+ef-\tfrac{1}{2}\partial h \right)-\tfrac{1}{4}X_{11},\quad W_1=\tfrac{1}{12}A-\tfrac{1}{24} B,\\
    &L_2=\tfrac{1}{2}\left(\tfrac{1}{2}h^2+ef-\tfrac{1}{2}\partial h \right)+\tfrac{1}{4}X_{11},\quad W_2=\tfrac{1}{12}A+ \tfrac{1}{24} B,\\
\end{align*}
where we set
\begin{align*}
    A=2h^3+h \partial h+7\partial^2 h-15(\partial e\cdot f- \partial f\cdot e) -6 e f h ,\quad B=2f x_{01}-4 h x_{11}-e x_{21}.
\end{align*}

We decompose $\mathrm{FT}_1(\sll_2)$ as a $\mathcal{M}(2)\otimes \mathcal{M}(2)\otimes \pi^h$-module. 
Recall that the simple $\mathcal{M}(2)$-modules $\mathcal{M}_{s}:=\mathcal{M}_{1,s+1}$ in \S \ref{subsection: (1,p)-model} are simple currents with the fusion rules
\begin{align}\label{fusion of singlet modules}
    \mathcal{M}_n\boxtimes \mathcal{M}_m \simeq \mathcal{M}_{n+m}
\end{align}
by \cite[Theorem 1.1]{CMY} and \eqref{singlet decompsoition} is equivalent to 
\begin{align}\label{Short exact sequence for P=1}
    0\rightarrow \mathcal{M}_n\rightarrow \pi^u_{-n}\rightarrow \mathcal{M}_{n+1}\rightarrow 0.
\end{align}

\begin{proposition}\label{decomposition for p=1} \hspace{0mm}\\
\textup{(1)} There is an isomorphism of $\mathcal{M}(2)\otimes \mathcal{M}(2)\otimes \pi^h$-modules
\begin{align*}
    &\mathrm{FT}_1(\sll_2)\simeq \bigoplus_{n+m\equiv 0} \mathcal{M}_n\otimes \mathcal{M}_m\otimes \pi^{h}_{(n+m)\varpi}.
\end{align*}
\textup{(2)} There is an isomorphism of vertex algebras
\begin{align*}
    &\mathrm{FT}_1(\sll_2)\simeq \bigcap_{i=1,2} \mathrm{Ker} S_i\subset \Pi[0]\otimes V_{A_1},
\end{align*}
where we set 
\begin{align*}
   S_1=\int Y(e^{u},z)dz,\quad S_2=\int Y(e^{-(v+\alpha)},z)dz.
\end{align*}
\end{proposition}
\proof
(1) Set $N_{n,m}:=\mathcal{M}_n\otimes \mathcal{M}_m\otimes \pi^{h}_{(n+m)\varpi}$. It is straightforward to show that the strong generators of $\mathrm{FT}_1(\sll_2)$ satisfy
$$e\in N_{1,1},\quad f\in N_{-1,-1},\quad x_{ij}\in N_{2-i-j,i-j},\quad (0\leq i,j\leq 2).$$
Hence, \eqref{fusion of singlet modules} implies the assertion.
(2) By (1) and \eqref{Short exact sequence for P=1}, we have
\begin{align*}
    &\mathrm{FT}_1(\sll_2)\simeq \bigcap_{i=1,2} \mathrm{Ker} \widetilde{S}_i\subset \Pi[0]\otimes V_{A_1},\quad \widetilde{S}_i=\int Y(e^{-u_i},z)dz.
\end{align*}
Then we obtain the assertion by using \eqref{free field decomposition}.
\endproof
Since the simplicity of vertex algebras is preserved under simple current extensions, the above proposition gives an alternative proof for the simplicity of $\mathrm{FT}_1(\sll_2)$. 

Similarly, by comparing \eqref{FMS} with \eqref{p=2 singlet}, we find a conformal embedding $\mathcal{M}(2)\otimes \pi^{v}\hookrightarrow \beta\gamma$. Since $\beta$ generates $\mathcal{M}_1\otimes \pi^v_{-1}$ and $\gamma$ does $\mathcal{M}_{-1}\otimes \pi^v_{1}$, we obtain the decomposition 
\begin{align}\label{decomoposition of betagamma}
    \beta\gamma\simeq \bigoplus_{n\in \Z} \mathcal{M}_n\otimes \pi^v_{-n}
\end{align}
as $\mathcal{M}(2)\otimes \pi^v$-modules.
Let us decompose \eqref{p=1 sl2 FT algebra} in terms of the weight of $\mathrm{B}^{\mathrm{aff}}$-action (i.e. $(-\alpha_0)$-grading) on $\C^{2n+1}$. Then we have
\begin{align*}
    \mathrm{FT}_1(\sll_2)=\bigoplus_{n\in \Z} \mathrm{FT}_1^{n\alpha}(\sll_2),\quad \mathrm{FT}_1^{0}(\sll_2)\simeq\bigoplus_{n\geq0}\mathbb{V}_{n\alpha}^{-1}.
\end{align*}
Since the restriction of \eqref{free field decomposition} for the trivial $(-\alpha_0)$-grading is exactly the part $n=m$, Proposition \ref{decomposition for p=1} implies 
\begin{align}\label{decomposition of affine singlet for P=1}
    \mathrm{FT}_1^{0}(\sll_2)\simeq \bigoplus_{n\in \Z} \mathcal{M}_n\otimes \mathcal{M}_n\otimes \pi^{h}_{n\alpha}.
\end{align}
The descriptions of $\mathrm{FT}_1(\sll_2), \mathrm{FT}_1^{0}(\sll_2)$ and $\beta\gamma$ as extensions from $\mathcal{M}(2)$ and Heisenberg vertex algebras give some relations among them.
\begin{corollary}\hspace{0mm}\\
\textup{(1)} 
Set $z=v\otimes 1-1\otimes v$ inside $\beta\gamma^{\otimes 2}$ via \eqref{decomoposition of betagamma}. Then we have an isomorphism of vertex algebras
\begin{align*}
    \mathrm{FT}_1^{0}(\sll_2)\simeq \mathrm{Com}\left(\pi^z, \beta\gamma^{\otimes 2}\right).
\end{align*}
\textup{(2)} 
Set $y=z\otimes 1+1\otimes \alpha$ inside $\beta\gamma^{\otimes 2}\otimes V_{A_1}$. 
Then we have an isomorphism of vertex algebras
\begin{align*}
    \mathrm{FT}_1(\sll_2)\simeq H_{\mathrm{rel}}^{\frac{\infty}{2}+0}(\mathfrak{gl}_1,  \beta\gamma^{\otimes 2}\otimes V_{A_1}).
\end{align*}
\end{corollary}
\proof
Straightforward by using the uniqueness of simple vertex algebra structure of simple current extensions.
\endproof
Let us describe the projective cover $\mathcal{P}_{\mathrm{FT}_1(\sll_2)}$ of $\mathrm{FT}_1(\sll_2)$. 
Let $\mathrm{Rep}^{\mathrm{wt}}(\mathrm{FT}_1(\sll_2))$ be the category of finite-length weight modules over $\mathrm{FT}_1(\sll_2)$, see \S \ref{Correspondence of categories}, for a precise definition.
Recall that $\mathrm{FT}_1(\sll_2)$ is a simple current extension of $\mathcal{M}(2)\otimes \mathcal{M}(2)\otimes \pi^h$ by Proposition \ref{decomposition for p=1}.
Let $\mathcal{O}_{\mathcal{M}(2)}$ denote the category of finite-length grading-restricted generalized $\mathcal{M}(2)$-modules. It is a braided tensor category and agrees with the category of $C_1$-cofinite grading-restricted generalized $\mathcal{M}(2)$-modules. 
The previous simple modules $\mathcal{M}_n$ are objects in $\mathcal{O}_{\mathcal{M}(2)}$ and indeed simple currents satisfying \eqref{fusion of singlet modules}. 
Since $\mathcal{O}_{\mathcal{M}(2)}$ does not have enogh projectives, we restrict to the braided tensor subcategory $\mathcal{O}_{\mathcal{M}(2)}^T$, that is, the full subcategory of $\mathcal{O}_{\mathcal{M}(2)}$ consisting of objects $M$ such that the monodromy operator $\mathscr{M}_{\mathcal{M}_1,M}$ is semisimple on $\mathcal{M}_1\boxtimes M$. 
Then $\mathcal{O}_{\mathcal{M}(2)}^T$ has enough projects, see \cite{CMY,CMY3}. 

Let $\mathrm{Rep}^{\mathrm{wt}}(\pi^h)$ denote the category of Fock modules over $\pi^h$ 
and introduce the category $\mathcal{D}:=\mathcal{D}(\mathcal{M}(2),\mathcal{M}(2), \pi^h)$ consisting of finite-length grading-restricted generalized modules $M$ over $\mathcal{M}(2)\otimes \mathcal{M}(2)\otimes \pi^h$ 
such that the submodules over $\mathcal{M}(2)$ (resp.\ $\pi^h$) generated by $m\in M$ lies in $\mathcal{O}_{\mathcal{M}(2)}^T$ (resp.\ $\mathrm{Rep}^{\mathrm{wt}}(\pi^h)$).
Then $\mathcal{D}$ is a braided tensor category and is naturally equivalent to the Deligne product $\mathcal{O}_{\mathcal{M}(2)}^T\boxtimes \mathcal{O}_{\mathcal{M}(2)}^T \boxtimes \mathrm{Rep}^{\mathrm{wt}}(\pi^h)$ by \cite[Theorem 3.15, 4.11]{M}. 
Now, $\mathrm{FT}_1(\sll_2)=\bigoplus_{n+m0}N_{n,m}$ is an object of the ind-completion $\mathrm{Ind}(\mathcal{D})$ of $\mathcal{D}$.
Since $\mathcal{O}_{\mathcal{M}(2)}^T$ is a subcategory of $\mathcal{O}_{\mathcal{M}(2)}$ consisting of $C_1$-cofinite modules, it is striaghtforward to show that $\mathrm{Ind}(\mathcal{D})$ is a braided tensor category by \cite{CMY2}.
Now, let $\mathrm{Ind}(\mathcal{D})^0\subset \mathrm{Ind}(\mathcal{D})$ denote the full braided tensor subcateogry consisting of modules $M$ satisfying $\mathscr{M}_{N_{n,m},M}=\mathrm{id}_{N{n,m}\boxtimes M}$ 
and $\mathrm{Rep}^\mathcal{D}(\mathrm{FT}_1(\sll_2))$ be the braided tensor category of $\mathrm{FT}_1(\sll_2)$-modules objects in $\mathrm{Ind}(\mathcal{D})$, or equivalently, $\mathrm{FT}_1(\sll_2)$-modules which lie in $\mathrm{Ind}(\mathcal{D})$ as $\mathcal{M}(2)\otimes \mathcal{M}(2)\otimes \pi^h$-modules.
By \cite[Lemma 3.1]{CMY4}, $\mathrm{Ind}(\mathcal{D})^0$ agrees with the ind-completion $\mathrm{Ind}(\mathcal{D}^0)$ of $\mathcal{D}^0=\mathcal{D}\cap \mathrm{Ind}(\mathcal{D})^0$.
Then we have a  braided tensor functor
\begin{align*}
    \mathrm{Ind}(\bullet)= \mathrm{FT}_1(\sll_2)\boxtimes \bullet \colon \mathrm{Ind}(\mathcal{D}^0)\rightarrow \mathrm{Rep}^\mathcal{D}(\mathrm{FT}_1(\sll_2))
\end{align*}
called the induction functor \cite{CKM,CMY2}.
By the same proof as \cite[Lemma 3.12]{CMY3}, one can show that $\mathrm{Rep}^{\mathrm{wt}}(\mathrm{FT}_1(\sll_2))$ is a subcategory of $\mathrm{Rep}^\mathcal{D}(\mathrm{FT}_1(\sll_2))$ and the functor $\mathrm{Ind}$ restricts to 
\begin{align*}
    \mathrm{Ind}(\bullet)
\colon \mathcal{D}^0\rightarrow \mathrm{Rep}^{\mathrm{wt}}(\mathrm{FT}_1(\sll_2)).
\end{align*}
The projective cover $\mathcal{P}_{\mathrm{FT}_1(\sll_2)}$ inside $\mathrm{Rep}^{\mathrm{wt}}(\mathrm{FT}_1(\sll_2))$ can be constructed from the projective cover $\mathcal{P}$ of 
$\mathcal{M}(2)$ in $\mathcal{O}_{\mathcal{M}(2)}^T$ \cite[Theorem 5.1.3]{CMY}, which is an indecomposable module with the following Loewy diagram: 
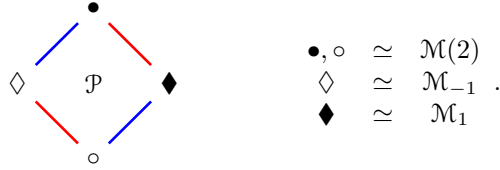
\begin{center}
\begin{tikzpicture}[x=2mm,y=2mm]
\useasboundingbox (0,0) rectangle (30,10);
\node (A) at (5,10) {$\bullet$ };
\node (B1) at (0,5) {$\lozenge$}; \node (B2) at (10,5) {$\blacklozenge$};
\node (C) at (5,0) {$\circ$};
\draw [line width=1pt, blue] (A) -- (B1);
\draw [line width=1pt, red] (A) -- (B2);
\draw [line width=1pt, red] (B1) -- (C); 
\draw [line width=1pt, blue] (B2) -- (C);
\draw (5,5) node {$\mathcal{P}$};
\draw (25,5) node {$\begin{array}{ccc}\bullet, \circ &\simeq& \mathcal{M}(2)\\ \lozenge&\simeq &\mathcal{M}_{-1} \\ \blacklozenge&\simeq &\mathcal{M}_1 \end{array}$.};
\end{tikzpicture}
\end{center}
The diagram reads as follows: the socle of $\mathcal{P}$ is $\circ\simeq \mathcal{M}(2)$, the socle of the quotient $\mathcal{P}/ \circ$ is $\lozenge\oplus \blacklozenge\simeq \mathcal{M}_{-1}\oplus \mathcal{M}_1$, and the top $\bullet\simeq \mathcal{M}(2)$ is the head, i.e. the maximal semisimple quotient of $\mathcal{P}$.
\begin{proposition}
The module $\mathcal{P}\otimes \mathcal{P}\otimes \pi^h$ lies in $\mathcal{D}^0$ and  the induced module $\mathrm{Ind}(\mathcal{P}\otimes \mathcal{P}\otimes \pi^h)$ is a projective cover of $\mathrm{FT}_1(\sll_2)$ in $\mathrm{Rep}^{\mathrm{wt}}(\mathrm{FT}_1(\sll_2))$.
\end{proposition}
\proof 
Since the group of simple currents appearing in $\mathrm{FT}_1(\sll_2)$ is generated by $N_{2,0}$, $N_{1,1}$ (see the proof of Proposition \ref{decomposition for p=1}), it suffices to show $\mathcal{P}\otimes \mathcal{P}\otimes \pi^h$ is monodromy-free for them. Since $\mathcal{P}$ lies in $\mathcal{O}_{\mathcal{M}(2)}^T$, it satisfies 
$\mathscr{M}_{\mathcal{M}_2,\mathcal{P}}=\mathrm{id}_{\mathcal{M}_2\boxtimes \mathcal{P}}$, which implies the monodromy-freeness for $N_{2,0}$. Since $\mathcal{P}$ is indecomposable, $\mathcal{M}_1\boxtimes \mathcal{M}_1\simeq \mathcal{M}_2$ implies $\mathscr{M}_{\mathcal{M}_1,\mathcal{P}}=\pm\mathrm{id}_{\mathcal{M}_\boxtimes \mathcal{P}}$ by \cite[Proposition 7]{CGNS}. In either case, 
$\mathscr{M}_{\mathcal{M}_1\otimes \mathcal{M}_1,\mathcal{P}\otimes \mathcal{P}}=\mathrm{id}_{(\mathcal{M}_1\otimes \mathcal{M}_1)\boxtimes (\mathcal{P}\otimes \mathcal{P})}$ and thus the monodromy-freeness for $N_{1,1}$. Therefore, we have shown the first assertion. 
Since $\mathcal{P}\otimes \mathcal{P}\otimes \pi^h$ is projective in $\mathcal{D}$ and thus in $\mathrm{Ind}(\mathcal{D}^0)$, the induced module $\mathrm{Ind}(\mathcal{P}\otimes \mathcal{P}\otimes \pi^h)$ is also projective in $\mathrm{Rep}^{\mathrm{wt}}(\mathrm{FT}_1(\sll_2))$ by the Frobenius reciprocity. 
Finally, we show that $\mathrm{Ind}(\mathcal{P}\otimes \mathcal{P}\otimes \pi^h)$ is a projective cover of $\mathrm{FT}_1(\sll_2)$.
Since the simplicity is preserved by the functor $\mathrm{Ind}$, the Loewy diagram of $\mathcal{P}\otimes \mathrm{P}\otimes \pi^h$ as a $\mathcal{M}(2)\otimes \mathcal{M}(2)\otimes \pi^h$-modules induces that of  $\mathrm{Ind}(\mathcal{P}\otimes \mathcal{P}\otimes \pi^h)$ as in Figure \ref{fig:Loewy diagram}. The surjection $\mathrm{Ind}(\mathcal{P}\otimes \mathcal{P}\otimes \pi^h)\twoheadrightarrow \mathrm{FT}_1(\sll_2)$ is given by the quotient to the head $\bullet\bullet$. Hence, $\mathrm{Ind}(\mathcal{P}\otimes \mathcal{P}\otimes \pi^h)$ is a projective cover of $\mathrm{FT}_1(\sll_2)$.
\endproof
\begin{figure}[h!]
    \centering
\begin{tikzpicture}[x=1.5mm,y=1.5mm]
\useasboundingbox (0,0) rectangle (60,40);
\node (A) at (25,40) {$\bullet\bullet$ };
\node (B1) at (10,30) {$\lozenge\bullet$}; \node (B2) at (20,30) {$\blacklozenge\bullet$}; \node (B3) at (30,30) {$\bullet\blacklozenge$}; \node (B4) at (40,30) {$\bullet\lozenge$};
\node (C1) at (0,20) {$\circ\bullet$};
\node (C2) at (10,20) {$\lozenge\blacklozenge$};
\node (C3) at (40,20) {$\blacklozenge\lozenge$};
\node (C4) at (30,20) {$\lozenge\lozenge$};
\node (C5) at (20,20) {$\blacklozenge\blacklozenge$};
\node (C6) at (50,20) {$\bullet\circ$};  
\node (D1) at (10,10) {$\circ\blacklozenge$}; \node (D2) at (20,10) {$\circ\lozenge$}; \node (D3) at (30,10) {$\lozenge\circ$}; \node (D4) at (40,10) {$\blacklozenge\circ$}; 
\node (E) at (25,0) {$\circ\circ$};

\draw (50,3) node {{\small$\clubsuit\spadesuit=\mathrm{Ind}(\clubsuit\otimes \spadesuit\otimes \pi^h)$}};
\draw (50,0) node {{\small($\clubsuit,\spadesuit=\bullet, \circ, \lozenge, \blacklozenge$)}};

\draw [line width=1pt, blue] (A) -- (B1);
\draw [line width=1pt, red] (A) -- (B2);
\draw [line width=1pt, green] (A) -- (B3);
\draw [line width=1pt, brown] (A) -- (B4);

\draw [line width=1pt, red] (B1) -- (C1); 
\draw [line width=1pt, green] (B1) -- (C2);
\draw [line width=1pt,brown] (B1) -- (C4);

\draw [line width=1pt, blue] (B2) -- (C1); 
\draw [line width=1pt, brown] (B2) -- (C3);
\draw [line width=1pt, green] (B2) -- (C5);

\draw [line width=1pt, blue] (B3) -- (C2); 
\draw [line width=1pt, red] (B3) -- (C5);
\draw [line width=1pt,brown] (B3) -- (C6);

\draw [line width=1pt, red] (B4) -- (C3); 
\draw [line width=1pt, blue] (B4) -- (C4);
\draw [line width=1pt, green] (B4) -- (C6);

\draw [line width=1pt, green] (C1) -- (D1); 
\draw [line width=1pt, brown] (C1) -- (D2); 

\draw [line width=1pt, red] (C2) -- (D1); 
\draw [line width=1pt,brown] (C2) -- (D3); 

\draw [line width=1pt, blue] (C3) -- (D2); 
\draw [line width=1pt, green] (C3) -- (D4); 

\draw [line width=1pt, red] (C4) -- (D2); 
\draw [line width=1pt, green] (C4) -- (D3); 

\draw [line width=1pt, blue] (C5) -- (D1); 
\draw [line width=1pt, brown] (C5) -- (D4); 

\draw [line width=1pt, blue] (C6) -- (D3); 
\draw [line width=1pt, red] (C6) -- (D4); 

\draw [line width=1pt, brown] (D1) -- (E); 
\draw [line width=1pt, green] (D2) -- (E); 
\draw [line width=1pt, red] (D3) -- (E); 
\draw [line width=1pt, blue] (D4) -- (E); 
\end{tikzpicture}
    \caption{Loewy diagram of $\mathcal{P}(\mathrm{FT}_1(\sll_2))$
    }
    \label{fig:Loewy diagram}
\end{figure}

\section{Logarithmic Kazhdan--Lusztig correspondence}\label{section: log KL}
\subsection{Kazama--Suzuki dual}
We can use the Kazama--Suzuki coset construction \eqref{Kazama-Suzuki} to introduce an extension of $\W^\ell(\sll_{2|1})$ corresponding to $\FT(\sll_2)$, namely 
\begin{align}\label{Kazama-Suzuki dual}
    s\mathcal{W}_p(\sll_{2|1}):=\mathrm{Com}\left(\pi^\Delta,\FT(\sll_2)\otimes V_\Z\right) ,
\end{align}
where 
$\Delta(z)=\varpi(z)\otimes 1-1\otimes x(z)$ and $x(z):=u(z)$ is the Heisenberg field of $V_\Z$, see \S \ref{sec: iHR} for the notation. It is an extension of $\W^\ell(\sll_{2|1})$ at the level $\ell$ satisfying $(k+2)(\ell+1)=1$, i.e. $\ell=-1+p$.
\begin{proposition}\label{super-side}
We have an isomorphism of vertex superalgebras
\begin{align}\label{free field realizations}
     s\mathcal{W}_p(\sll_{2|1})\simeq \bigcap_{i=1,2} \mathrm{Ker}\widehat{S}_i\subset V_{\sqrt{p}A_1}\otimes V_\Z\otimes \pi^{\alpha^\dagger},
\end{align}
where $\alpha^\dagger(z)$ is the Heisenberg field $\alpha^\dagger(z)\alpha^\dagger(w)\sim -2/(z-w)^2$ and  
\begin{align}\label{screening charges}
   \widehat{S}_1=\begin{cases}
   \displaystyle{\int Y(e^{-\frac{1}{\sqrt{p}}\alpha},z)dz} & (p\geq 2),\\
\displaystyle{\int Y(e^{x-\frac{1}{2}(\alpha+\alpha^\dagger)},z)dz} & (p=1),
   \end{cases}
   \quad \widehat{S}_2=\int Y(e^{x+\frac{1}{2\sqrt{p}}(\alpha-\alpha^\dagger)},z)dz. 
\end{align}
\end{proposition}
\proof
By Proposition \ref{decomposition for p=1} for $p=1$ and Theorem \ref{identificatio of affine FT}, \eqref{FMS} for $p\geq2$, $\FT(\sll_2)$ is characterized as the joint kernel of screening operators 
\begin{align}\label{screenings of affien FT}
    \FT(\sll_2)\simeq \bigcap_{i=1,2} \mathrm{Ker} S_i|_{\Pi[0]\otimes V_{\sqrt{p}A_1}}
\end{align}
with 
\begin{align*}
  S_1=
  \begin{cases}
   \displaystyle{\int Y(e^{-\frac{1}{p}(u+v)-\frac{1}{\sqrt{p}}\alpha},z)dz,} & (p\geq 2),\\
   \displaystyle{\int Y(e^{-(v+\alpha)},z)dz} & (p=1),
   \end{cases}
  \quad S_2=\int Y(e^u,z)dz.
\end{align*}
Ii is straightforward to show that there is a unique isomorphism of vertex algebras 
$$\pi^H\otimes \pi^{\alpha} \otimes \pi^x\otimes \pi^{\alpha^\dagger} \simeq \pi^{u,v}\otimes \pi^{\alpha}\otimes \pi^x,$$
which sends
\begin{align}\label{identification of Heisenbergs}
\begin{array}{llllll}
 H &\mapsto& \Delta=-(v+\tfrac{1}{2\sqrt{p}}\alpha+x),&x &\mapsto& x+u+v, \\
 \alpha &\mapsto&  \alpha+\tfrac{1}{\sqrt{p}}(u+v), &\alpha^\dagger &\mapsto &\alpha+2\sqrt{p}(x+v)+\tfrac{1}{\sqrt{p}}(u+v).
\end{array}
\end{align}
It induces an isomorphism of vertex superalgebras 
\renewcommand{\arraystretch}{1.5}
\begin{align*}
\begin{array}{ccc}
V_{\Z}\otimes V_{\sqrt{p}A_1}\otimes \pi^{\alpha^\dagger}&\xrightarrow{\simeq}&\mathrm{Com}(\pi^\Delta, \Pi[0]\otimes V_{\sqrt{p}A_1}\otimes V_\Z)\\
e^{nx}\otimes e^{m\sqrt{p}\alpha} \otimes \mathbf{1} &\mapsto& e^{(n+m)(u+v)}\otimes e^{m\sqrt{p}\alpha}\otimes e^{nx}.
\end{array}
\end{align*}
\renewcommand{\arraystretch}{1}
Then the assertion follows from the isomorphisms
\begin{align*}
    \mathrm{Com}\left(\pi^\Delta,\FT(\sll_2)\otimes V_\Z\right) 
    &\simeq \bigcap_{i=1,2} \mathrm{Ker}S_i|_{\mathrm{Com}\left(\pi^\Delta,\Pi[0]\otimes V_{\sqrt{p}A_1}\otimes V_\Z\right)}\\
    &\simeq \bigcap_{i=1,2} \mathrm{Ker}S_i|_{V_{\Z}\otimes V_{\sqrt{p}A_1}\otimes \pi^{\alpha^\dagger}},
\end{align*}
which identifies $S_i$ with $\widehat{S}_i$ by \eqref{identification of Heisenbergs}. 
This completes the proof.
\endproof
\begin{remark}
\textup{For $p\geq 2$, $\widehat{S}_i$ can also be derived from the Wakimoto modules of $V^\ell(\sll_{2|1})$ which is a super-analogue of \cite{FF,F}. Set $M_{\sll_{2|1}}:=\beta\gamma\otimes bc^{\otimes 2}$, $\pi_\h^{\ell+1}$ the Heisenberg vertex algebra for the Cartan subspace $\h\subset \sll_{2|1}$ and $\mathbb{W}^\ell_{\lambda}:=M_{\sll_{2|1}}\otimes \pi^\ell_{\h,\lambda}$ ($\lambda\in \h^*$).
Then we have an embedding $\mu\colon V^\ell(\sll_{2|1})\hookrightarrow \mathbb{W}^\ell_{0}$ of vertex superalgebras and thus $\mathbb{W}^\ell_{\lambda}$ has a structure of $V^\ell(\sll_{2|1})$-modules, called Wakimoto representation. 
The embedding $\mu$ is extended to a complex
\begin{align*}
    0\rightarrow V^\ell(\sll_{2|1})\xrightarrow{\mu} \mathbb{W}^\ell_{0} \xrightarrow{\oplus \mathbb{S}_i}  \bigoplus_{i=1,2} \mathbb{W}^\ell_{-\alpha_i}
\end{align*}
of $V^\ell(\sll_{2|1})$-modules with the screening operators
$$\mathbb{S}_1=\int Y((\beta+c_1b_2)e^{-\frac{1}{\ssqrt{\ell+1}}\alpha_1},z)dz,\quad \mathbb{S}_2=\int Y(b_1e^{-\frac{1}{\ssqrt{\ell+1}}\alpha_2},z)dz$$
where $\alpha_1,\alpha_2\in \h$ are roots corresponding to $e_{1,2}, e_{2,3}\in\sll_{2|1}$.
Applying the BRST reduction, we obtain a complex of $\W^\ell(\sll_{2|1})$-modules 
\begin{align*}
    0\rightarrow \W^\ell(\sll_{2|1})\xrightarrow{\mu} bc\otimes \pi^{\ell+1}_\h \xrightarrow{\oplus \mathbb{S}_i^W}  \bigoplus_{i=1,2} bc\otimes \pi^{\ell+1}_{\h,-\alpha_i}
\end{align*}
with the screening operators
$$\mathbb{S}_1^W=\int Y(e^{-\frac{1}{\ssqrt{\ell+1}}\alpha_1},z)dz,\quad \mathbb{S}_2^W=\int Y(b_1e^{-\frac{1}{\ssqrt{\ell+1}}\alpha_2},z)dz.$$
By identifying $V_\Z\simeq bc$ and $\alpha=\alpha_1$, $\alpha^\dagger=\alpha_1+2\alpha_2$, $\widehat{S}_i$ coincides with $\mathbb{S}_i^W$.
}
\end{remark}
Let us also introduce the following subalgebras
\begin{align*}
    &\mathcal{M}_p(\sll_2):=\FT(\sll_2)\cap (\Pi[0]\otimes \pi^\alpha),\quad s\mathcal{M}_p(\sll_{2|1}):=s\mathcal{W}_p(\sll_{2|1})\cap V_\Z\otimes \pi^{\alpha,\alpha^\dagger}
\end{align*}
which play the role of the singlet algebra $\mathcal{M}(p)\subset \W(p)$ in the of setting of $\FT(\sll_2)$ and $ s\mathcal{W}_p(\sll_{2|1})$, respectively.
\subsection{Correspondence of categories}\label{Correspondence of categories}
We introduce the category of finite-length weight modules $\mathrm{Rep}^{\mathrm{wt}}(\FT(\sll_2))$ for $\FT(\sll_2)$ as the category of weak finite-length $\FT(\sll_2)$-modules $M$ satisfying 
(i) $M$ admits a simultaneous generalized eigenspace decomposition for $(L_0,J^+_0)$ with $J^+(z)=\frac{1}{2}h$ the form $$M=\bigoplus_{\Delta,\lambda} M_\lambda[\Delta],\quad M_\lambda[\Delta]:=\left\{m\in M \left| \begin{array}{l}(L_0-\Delta)^{>\!\!>0}m=0\\ (J^+_0-\lambda)m=0\end{array}\right.\right\},  (\Delta,\lambda \in \C),$$ 
(ii) $\dim M_\lambda[\Delta]<\infty$ and $M_\lambda[\Delta+N]=0$ as $N\in \Z$ goes to $-\infty$.
Note that $M$ decomposes into a direct sum of Heisenberg Fock modules for $\pi^{J^+}$. 
Since $J^+_0$ acts semisimplly on $\FT(\sll_2)$ with $\Z$-eigenvalues, it decomposes into
\begin{align*}
    \mathrm{Rep}^{\mathrm{wt}}(\FT(\sll_2))=\bigoplus_{[\lambda]\in \C/ \Z}\mathrm{Rep}^{\mathrm{wt}}(\FT(\sll_2))^{[\lambda]}
\end{align*}
in terms of the $J^+_0$-eigenvalues.
The $\W$-superalgebra $\W^\ell(\sll_{2|1})$ also has a Heisenberg field $J^-(z)$ such that  
$J^-(z)J^-(w)\sim -(2\ell+1)/(z-w)^2$ and $J^+_0$ acts semisimplly on $s\W_p(\sll_{2|1})$ with $\Z$-eigenvalues.
The pair of blocks 
$$\mathrm{Rep}^{\mathrm{wt}}(\FT(\sll_2))^{[\lambda]},\quad \mathrm{Rep}^{\mathrm{wt}}(s\W_p(\sll_{2|1}))^{[\mu]}$$
for suitable $([\lambda], [\mu])$ are equivalent as abelian categories.
It is an upgrade of Feigin--Semikhatov duality between the category of weight modules for the subalgebras $V^k(\sll_2)$ and $\W^\ell(\sll_{2|1})$ \cite{CGNS}. 
Since the results/proofs are obtained by word-by-word translation from \cite{CGNS,FN}, we will omit the detail.
We decompose $\FT(\sll_2)$ and $s\W_p(\sll_{2|1})$ as modules over the Heisenberg vertex subalgebras and their coset
\begin{align*}
  \FT(\sll_2)\simeq \bigoplus_{n\in \Z} \mathscr{C}_n\otimes \pi^{J^+}_n,\quad s\W_p(\sll_{2|1})\simeq \bigoplus_{n\in \Z} \mathscr{D}_n\otimes \pi^{J^-}_n.
\end{align*}
By construction \eqref{Kazama-Suzuki dual}, the coset vertex subalgebra $\mathscr{C}_0$ is isomorphisc to $\mathscr{D}_0$ and moreover $\mathscr{C}_n\simeq \mathscr{D}_n$ as their modules. 
We introduce the Heisenberg vertex algebra $\pi^{J^+_*}$ such that $J^+_*(z)J^+_*(w)\sim -J^+(z)J^+(w)$. Then we have an isomorphism 
\begin{align}\label{kernel VOA}
    V_\Z\otimes \pi^v\xrightarrow{\simeq} \bigoplus_{n\in \Z}\pi^{J^+_*}_n\otimes \pi^{J^-}_n.
\end{align}
By using the relative semi-infinite cohomology \cite{Fe,FGZ}, the coset constrcuction \eqref{Kazama-Suzuki dual} is reformulated as an isomorphism of vertex superalgebras 
\begin{align*}
    \mathrm{H}_{\mathrm{rel}}^m(\gl_1,\FT(\sll_2)\otimes V_\Z\otimes \pi^v)\simeq \delta_{m,0}s\W_p(\sll_{2|1})
\end{align*}
thanks to the property
$\mathrm{H}_{\mathrm{rel}}^m(\gl_1,\pi^{J^+}_a\otimes\pi^{J^+_*}_b)\simeq \delta_{m,0} \delta_{a+b,0}\C$.
Replacing \eqref{kernel VOA} with
\begin{align*}
    K^\lambda:=V_\Z\otimes \pi^v_{\lambda v}\simeq \bigoplus_{n\in \Z}\pi^{J^+_*}_{-n-\varepsilon\lambda}\otimes \pi^{J^-}_{n+\frac{1}{\varepsilon}\lambda},\quad (\varepsilon=1/\ssqrt{2p}),
\end{align*}
we obtain a functor 
\renewcommand{\arraystretch}{1.3}
\begin{align*}
\begin{array}{cccc}
    \mathrm{H}_\lambda \colon  & \mathrm{Rep}^{\mathrm{wt}}(\FT(\sll_2))^{[\varepsilon\lambda]} &\rightarrow & \mathrm{Rep}^{\mathrm{wt}}(s\W_p(\sll_{2|1}))^{[\frac{1}{\varepsilon}\lambda]}\\
    & M &\mapsto & \mathrm{H}_{\mathrm{rel}}^0(\gl_1,M \otimes K^\lambda).
\end{array}
\end{align*}
\renewcommand{\arraystretch}{1}
Let us replace $\mathrm{H}_\lambda$ with $\mathrm{H}_{\lambda(\theta)}$ for $\lambda(\theta)=\lambda-\frac{1}{\varepsilon}\theta$ ($\theta\in \Z$). Then the domain $\mathrm{Rep}^{\mathrm{wt}}(\FT(\sll_2))^{[\varepsilon\lambda]}$ remains the same, but the target does change. 
These functors give equivalences of categories and related to each other by spectral flow twists $S_\bullet$ of modules. 

\begin{theorem}\label{Feigin-Semikhatov duality}\hspace{0mm}\\
\textup{(1)} We have equivalences of categories
\begin{align*}
    \mathrm{H}_{\lambda(\theta)} \colon  \mathrm{Rep}^{\mathrm{wt}}(\FT(\sll_2))^{[\varepsilon\lambda]} \xrightarrow{\simeq} \mathrm{Rep}^{\mathrm{wt}}(s\W_p(\sll_{2|1}))^{[\frac{1}{\varepsilon}\lambda(\theta)]},\quad (\theta\in \Z).
\end{align*}
\textup{(2)} The spectral flow twists $S_\bullet$ intertwine $\mathrm{H}_{\lambda(\theta)}$ up to natural isomorphisms
	\begin{center}
		\begin{tikzcd}[row sep=large, column sep = huge]
			\mathrm{Rep}^{\mathrm{wt}}(\FT(\sll_2))^{[\varepsilon\lambda]}
			\arrow[d,"S_{\theta_1}"']
			\arrow[r,"\mathrm{H}_{\lambda(\theta_3)}"]&
			\mathrm{Rep}^{\mathrm{wt}}(s\W_p(\sll_{2|1}))^{[\frac{1}{\varepsilon}\lambda(\theta_3)]}
			\arrow[d,"S_{\theta_2}"]\\
			\mathrm{Rep}^{\mathrm{wt}}(\FT(\sll_2))^{[\varepsilon\lambda']}
			\arrow[r, "\mathrm{H}_{\lambda'(\theta_4)}"]&
			\mathrm{Rep}^{\mathrm{wt}}(s\W_p(\sll_{2|1}))^{[\frac{1}{\varepsilon}\lambda'(\theta_4)]}
		\end{tikzcd}
	\end{center} 
for $\theta_i\in\Z$ and $\lambda,\lambda'\in \C$ satisfying $\theta_4=\theta_2+\theta_3$, $\lambda'=\lambda+\varepsilon \theta_1$.\\
\textup{(3)} The nonzero Fock intertwiners $I(\text{-},z)\colon \pi^v_{\lambda v}\otimes \pi^v_{\mu v}\rightarrow \pi^v_{(\lambda+\mu)v} \{z\}$ induces a natural isomorphism on the spaces of intertwiners for $\FT(\sll_2)$ and $s\W_p(\sll_{2|1})$
\begin{align*}
 I_{\FT(\sll_2)}\binom{M_3}{M_1\ M_2}\xrightarrow{\simeq} I_{s\W_p(\sll_{2|1})}\binom{\mathrm{H}_{\lambda+\mu}(M_3)}{\mathrm{H}_\lambda(M_1)\ \mathrm{H}_\mu(M_2)}.
\end{align*}
\end{theorem}
We note that the theorem holds also for the subalgebras $\mathcal{M}_p(\sll_2)$, $s\mathcal{M}_p(\sll_{2|1})$.
\subsection{Yetter--Drinfeld modules}
Following \cite{CLR}, we formulate the logarithmic Kazhdan--Lusztig correspondence.
By the previous subsection, we have the following vertex (super)algebra extensions $V\subset A$:
\begin{align}\label{extension pairs}
\begin{array}{ll}
    \FT(\sll_2)\subset \Pi[0]\otimes V_{\sqrt{p}A_1}, & s\W_p(\sll_{2|1})\subset V_\Z\otimes V_{\sqrt{p}A_1}\otimes \pi^{\alpha^\dagger},\\
    \mathcal{M}_p(\sll_2)\subset \Pi[0]\otimes \pi^\alpha, & s\mathcal{M}_p(\sll_{2|1})\subset V_\Z\otimes \pi^{\alpha,\alpha^\dagger}.
\end{array}
\end{align}
Suppose that the category of weight modules $\cU:=\mathrm{Rep}^{\mathrm{wt}}(V)$ has the structure of braided tensor categories. Then $A$ is a commutative algebra object of $\mathcal{U}$. We write $\cU_A$ the tensor category consisting of $A$-module objects in $\cU$ and $\cU_A^0$ the braided tensor category consisting of local $A$-module objects.
By \cite{CKM}, we have the following functors 
\begin{center}
\begin{tikzpicture}[x=2mm,y=2mm]
\useasboundingbox (0,0) rectangle (50,13);
\node (A) at (0,10) {$\cU$ };
\node (B) at (15,10) {$\cU_A$};
\node (C) at (15,0) {$\cU_A^0$};
\draw [->] (2,10.5) -- (13,10.5);
\draw [->] (13,9.5) -- (2,9.5);
\draw [->] (14,1) -- (14,9);
\node  at (7.5,12) {$\mathrm{Ind}_A$ };
\node  at (8,8) {$\mathrm{Forget}_A$};
\node  at (13,5) {$\iota$};
\node at (40,5) {$\begin{array}{ll}\mathrm{Ind}_A:=A \boxtimes_\cU\colon& \text{Induction functor}\\ \mathrm{Forget}_A\colon& \text{Forgetful functor} \\ \iota\colon& \text{Tautological embedding}.\end{array}$};
\end{tikzpicture}
\end{center}
The category $\cU_A^0$ is braided tensor equivalent to the category of modules over the vertex algebra $A$ consisting of modules $M$ lies in $\cU$ as $A$-modules. If $\cU$ is good enough, then we would have 
\begin{itemize}
    \item the Schauenburg functor gives the braided tensor equivalence
    \begin{align*}
    \mathcal{U}\xrightarrow{\simeq} \mathcal{Z}_{\cU_A^0}(\cU_A)
\end{align*} 
to the relative Drinfeld center $\mathcal{Z}_{\cU_A^0}(\cU_A)$,
    \item $\cU_A^0$ is identified with the category $\mathrm{Rep}^{\mathrm{wt}}(A)$ of weight modules, i.e. $A$-modules which are direct sums of Fock modules for the Heisenberg fields,
    \item the embedding $\iota\colon \cU_A^0\hookrightarrow \cU_A$ splits, which implies the existence of a Nichols algebra $\fN=\fN(X)$ (associated with an object $X$) in $\mathrm{Rep}^{\mathrm{wt}}(A)$ such that the category $^{\fN}_\fN\mathcal{YD}(\cU_A)$ of $(\mathfrak{N},\mathfrak{N})$ Yetter-Drinfeld module satisfies
        \begin{align*}
         \mathcal{Z}_{\cU_A^0}(\cU_A)\xrightarrow{\simeq} {}^{\fN}_\fN\mathcal{YD}(\cU_A)
\end{align*} 
as braided tensor categories. 
\end{itemize}
Therefore, we would have a braided tensor equivalence 
$$\mathrm{Rep}^{\mathrm{wt}}(V)\xrightarrow{\simeq} {}^{\fN}_\fN\mathcal{YD}(\mathrm{Rep}^{\mathrm{wt}}(A)).$$
Since $\mathrm{Rep}^{\mathrm{wt}}(A)$ is braided tensor equivalent to the category $\mathrm{Vect}_\Gamma^Q$ of graded vector spaces for a quadratic space $(\Gamma,Q)$, the Nichols algebra $\fN$ likely falls into the Heckenberger's classification \cite{He}. 
The general expectation is that the braided vector space $X$, which generates $\fN$, agrees with the space spanned by the highest weights of the screening operators. Then, we obtain a (quasi-)Hopf algebra defined as the braided Drinfeld double
$$\mathscr{U}:=\mathrm{Drin}_{\mathcal{H}}(\mathfrak{N}(X),\mathfrak{N}(X)^*),$$
over the (quasi-)Hopf algebra $\mathcal{H}$ representing $\mathrm{Vect}_\Gamma^Q$ gives the conjectural quantum group side of the logarithmic Kazhdan--Lusztig correspondence
\begin{align}
    \mathrm{Rep}^{\mathrm{wt}}(V)\xrightarrow{\simeq} {}^{\fN}_\fN\mathcal{YD}(\mathrm{Rep}^{\mathrm{wt}}(A)) \xrightarrow{\simeq}\mathrm{Rep}^{\mathrm{wt}}(\mathscr{U}).
\end{align}
The final conjecture for $V=\FT(\sll_2)$ was conjectured by Semikhatov--Tipunin \cite{ST1} and they identified $\mathscr{U}$ with a version of quantum supergroup for $\sll_{2|1}$. 

\subsection{Hopf algebra $\mathcal{H}$}
We describe $\mathscr{U}$ explicitly for $(V,A)=(s\mathcal{M}_p(\sll_{2|1}), V_\Z\otimes \pi^{\alpha,\alpha^\dagger})$. 
To start with, we fix a Hopf algebra $\mathcal{H}$ whose module category is braided tensor equivalent to $\mathrm{Rep}^{\mathrm{wt}}(V_\Z\otimes \pi^{\alpha,\alpha^\dagger})$. 
The category $\mathrm{Rep}^{\mathrm{wt}}(V_\Z\otimes \pi^{\alpha,\alpha^\dagger})$ has, by definition, morphisms of even parity and then is semisimple with simple objects 
\begin{align}\label{Fock type modules}
    V_\Z^{+}\otimes \pi^{\alpha,\alpha^\dagger}_{\xi},\ V_\Z^{-}\otimes \pi^{\alpha,\alpha^\dagger}_{\xi},\quad  (\xi\in \h)
\end{align}
where $\h=\C\alpha\oplus \C \alpha^\dagger$ and $V_\Z^{\pm}$ denotes $V_\Z$ with the vacuum of even (resp.\ odd) parity.
We introduce the notation
\begin{align}\label{highest weights}
\begin{split}
    &
    \widetilde{\beta}_1^\as=\underbrace{-\tfrac{1}{\sqrt{p}}\alpha}_{\widehat{\beta}_1^\as},\hspace{1.4cm}
    \widetilde{\beta}_2^\as=x+\underbrace{\tfrac{1}{2\sqrt{p}}(\alpha-\alpha^\dagger)}_{\widehat{\beta}_2^\as},\\
    &
    \widetilde{\beta}_1^\s=x\underbrace{-\tfrac{1}{2}(\alpha+\alpha^\dagger)}_{\widehat{\beta}_1^\s},\quad 
    \widetilde{\beta}_2^\s=x+\underbrace{\tfrac{1}{2}(\alpha-\alpha^\dagger)}_{\widehat{\beta}_2^\s},
\end{split}
\end{align}
corresponding the highest weights of the screening operators $\widehat{S}_i$ in Proposition \ref{super-side} for $p\geq 2$ and $p=1$, respectively. The simple modules \eqref{Fock type modules} can be expressed as
$$V_\Z^{+}\otimes \pi^{\alpha,\alpha^\dagger}_{\lambda\widehat{\beta}_1^\bullet+\mu \widehat{\beta}_2^\bullet },\quad (\lambda,\mu \in \C)$$
for both cases $\bullet=\as, \s$ and the braiding matrices $\mathcal{B}^\bullet$ for the simple $V_\Z\otimes \pi^{\alpha,\alpha^\dagger}$-modules corresponding to \eqref{highest weights} reads $(e^{\pi\ssqrt{-1}(\widetilde{\beta}_i^\bullet,\widetilde{\beta}_j^\bullet)})_{i,j=1,2}$ and thus
\begin{align}\label{braiding matrix}
    \mathcal{B}^\as=\left(\begin{array}{cc} 1 & 0 \\ 0 & -1 \end{array} \right)
    \left(\begin{array}{cc} q^2 & q^{-1} \\ q^{-1} & 1 \end{array} \right),\quad 
    \mathcal{B}^\s=\left(\begin{array}{cc} -1 & 0 \\ 0 & -1 \end{array} \right)
    \left(\begin{array}{cc} 1 & q^{-1} \\ q^{-1} & 1 \end{array} \right)
\end{align}
where $q=\mathrm{e}^{\pi \ssqrt{-1}/p}$ is a primitive $2p$-th root of unity.
Although $\mathcal{B}^\s$ is introduced only for $p=1$, it is important to see $\mathcal{B}^\s$ defined for $p\geq2$ as well.
The first factor of $\mathcal{B}^\bullet$ gives the parity of the object and the second the root data since the entries are expressed as  $q^{c_{ij}^\bullet}$ through the Cartan matrix
\begin{align}\label{Cartan matrix}
   C^\as=\left(\begin{array}{cc}2 &-1\\ -1&0 \end{array}\right),\quad C^\s=\left(\begin{array}{cc}0 &-1\\ -1&0 \end{array}\right),
\end{align}
respectively.
Recall that both $C^\bullet$ are Cartan matrices $\sll_{2|1}$ realized by
\begin{align*}
    & {}^\as H_1=e_{11}-e_{22},\quad {}^\as H_2=e_{22}+e_{33},\quad E_{\beta_1^\as}=e_{12},\quad E_{\beta_2^\as}=e_{23},\\
    & {}^\s H_1=e_{11}+e_{33},\quad {}^\s H_2=e_{22}+e_{33},\quad E_{\beta_1^s}=e_{31},\quad E_{\beta_2^\s}=e_{23}.
\end{align*}

Now, introduce the Hopf algebra 
\begin{align}\label{Cartan subalgebra}
\mathcal{H}^\bullet=\C[ H_i, K_i^{\pm 1}, K_0\mid i=1,2]/( K_0^2-1),\quad (\bullet=\as, \s),
\end{align}
whose coproduct $\Delta$, counit $\varepsilon$, and the antipode $S$ is given by  
\begin{align}\label{Hopf structre}
\begin{array}{lll}
    \Delta\colon&  H_i\mapsto  H_i\otimes 1+1\otimes  H_i,&  K_j\mapsto  K_j\otimes  K_j,\\
    \varepsilon\colon&  H_i\mapsto 0,&  K_j\mapsto 1,\\
    S\colon &  H_i\mapsto - H_i,&  K_j\mapsto  K_j.
\end{array}
\end{align}
In the below, we write ${}^\bullet H_i$ etc when we specify the algebras $\mathcal{H}^\bullet$.
The $q$-weight $\mathcal{H}{}^\bullet$-modules $M$ is, by definition, the $\mathcal{H}{}^\bullet$-modules $M$ such that $H_i$, $K_j$ act semisimply and $K_i=q^{H_i}$ $(i=1,2)$ hold as operators on $M$.
Then the simple $q$-weight $\mathcal{H}^\bullet$-modules are parameterized by the one dimensional representations
$$\C_{(\pm,\lambda\beta_1^\bullet+\mu\beta_2^\bullet)},\quad (\lambda,\mu\in \C).$$
Here $K_0$ acts by $\pm1$, ${}^\bullet H_i$ by the pairing $({}^\bullet H_i,\beta_j^\bullet)=c_{ij}^\bullet$ through the Cartan matrix \eqref{Cartan matrix}.
We denote by $\mathrm{Rep}^{\mathrm{wt}}(\mathcal{H}^\bullet)$ the braided tensor category of $q$-weight $\mathcal{H}$-modules equipped with the standard universal R-matrix $\mathcal{R}^\bullet=(-1)^{\overline{K}_0\otimes \overline{K}_0}q^{\Omega^\bullet}$ where $\Omega^\bullet$ in $\mathcal{H}^\bullet\otimes \mathcal{H}^\bullet$ is the quadratic Casimir element
\begin{align*}
\Omega^\bullet=\begin{cases}
-(H_1\otimes H_2+H_2\otimes H_1+2 H_2\otimes H_2) & (\bullet=\as),\\
-(H_1\otimes H_2+H_2\otimes H_1) & (\bullet=\s).
\end{cases}
\end{align*}
The following lemma is checked straightforwardly.
\begin{lemma}\label{free field equiv}\hspace{0mm}\\
\textup{(1)} We have an isomorphism of Hopf algebras $ \mathcal{H}^\as\xrightarrow{\simeq} \mathcal{H}^\s$ such that  
\begin{align*}
    \begin{array}{lll}
       {}^\as H_1\mapsto {}^\s H_1-{}^\s H_2,  &  {}^\as H_2 \mapsto {}^\s H_2,  &  K_0\mapsto {}^\s K_0, \\
       {}^\as K_1\mapsto {}^\s K_1 {}^\s K_2^{-1},  & {}^\as K_2 \mapsto{}^\s K_2. &
    \end{array}
\end{align*}
It sends $\mathcal{R}^\as$ to $\mathcal{R}^\s$ and thus induces a braided tensor equivalence
$$\mathrm{Rep}^{\mathrm{wt}}(\mathcal{H}^\as)\xrightarrow{\simeq} \mathrm{Rep}^{\mathrm{wt}}(\mathcal{H}^\s).$$
\textup{(2)} We have braided tensor equivalences
\begin{align*}
\begin{array}{ccc}
\mathrm{Rep}^{\mathrm{wt}}(V_\Z\otimes \pi^{\alpha,\alpha^\dagger})&\xrightarrow{\simeq}&\mathrm{Rep}^{\mathrm{wt}}_q(\mathcal{H}^\bullet)\\
V_\Z^{\pm}\otimes \pi^{\alpha,\alpha^\dagger}_{\lambda\widehat{\beta}_1^\bullet+\mu\widehat{\beta}_2^\bullet}&\mapsto& \C_{(\pm,\lambda\beta_1^\bullet+\mu\beta_2^\bullet)},
\end{array}
\quad (\bullet=\as,\s).
\end{align*}
\end{lemma}
Note that $(-1)^{\overline{K}_0\otimes \overline{K}_0}$ in $\mathcal{R}^\bullet$ is understood as the Koszul sign rule for swapping elements in vector superspaces $v\otimes w\mapsto (-1)^{\bar{v}\cdot \bar{w}}w\otimes v$. 

\subsection{Quantum superalgebra}
The highest weights \eqref{highest weights} satisfy $(\widetilde{\beta}_i^\bullet,\widetilde{\beta}_i^\bullet)\leq 1$ and form a positive-definite lattice $\Z\widetilde{\beta}_1^\bullet\oplus \Z\widetilde{\beta}_2^\bullet\subset \C^3(=\mathrm{span}\{x,\alpha,\alpha^\dagger\}$).
By \cite[Theorem 7.1]{L}, $\widehat{S}_1$, $\widehat{S}_2$ generate an action of the Nichols algebra $\mathfrak{N}(X^\bullet)$ associated with the braided vector space $X^\bullet=\C\widetilde{\beta}_1^\bullet\oplus \C\widetilde{\beta}_2^\bullet$, i.e. the object $X^\bullet:=\C_{(+,\beta_1)}\bigoplus \C_{(-,\beta_2)}$ in $\mathrm{Rep}^{\mathrm{wt}}_q(\mathcal{H}^\bullet)$.
By the classification of finite dimensional Nichols algebras of diagonal type \cite{He}, we associate with $\mathcal{B}^\bullet$ in \eqref{braiding matrix} the generalized Dynkin diagram
\begin{center}
\begin{tikzpicture}[x=1mm,y=1mm]
\useasboundingbox (0,0) rectangle (60,8);
\node (A) at (0,0) {};
\node (B) at (20,0) {};
\draw (A) circle (1);
\draw (B) circle (1);
\node at (0,4) {{\small$q^2$}};
\node at (20,4) {{\small$-1$}};
\node at (10,4) {{\small$q^{-2}$}};
\draw [line width=1pt, black] (A)-- (B);
\node (C) at (40,0) {};
\node (D) at (60,0) {};
\draw (C) circle (1);
\draw (D) circle (1);
\node at (40,4) {{\small$-1$}};
\node at (60,4) {{\small$-1$}};
\node at (50,4) {{\small$q^{-2}$}};
\draw [line width=1pt, black] (C)-- (D);
\end{tikzpicture}
\end{center}
named as $\mathbf{A}_2(q^2;\{1\})$ and $\mathbf{A}_2(q^{-2};\mathbb{I}_2)$ in \cite[\S 5.1.11]{AA}, respectively.
\footnote{Strictly speaking, the case $p=1$ in $\mathbf{A}_2(q^{-2};\mathbb{I}_2)$ is a degenerated case and the generalized Dynkin diagram is just a disjoint union of two nodes. Since the Nichols algebras in both cases are the same, we keep the notation $\mathbf{A}_2(q^{-2};\mathbb{I}_2)$.}
Then $\mathfrak{N}(X^\bullet)$ has a presentation by generators and relations \cite{A,AA} as 
{\small
\begin{align}\label{qSerre relations}
\begin{array}{ll}
\text{Generators}\colon &e_1, e_2,\\
\text{Relations}\colon &
\begin{cases} e_1^p=0,\ e_2^2=0,\
e_1^2e_2-(q+q^{-1})e_1e_2e_1+e_2e_1^2=0 & (\bullet=\as), \\
e_1^2=0,\quad e_2^2=0,\quad (e_1e_2+q^{-1}e_2e_1)^p=0 & (\bullet=\s).
\end{cases}
\end{array}
\end{align}}
For $\bullet=\s$, we have $(e_1e_2+q^{-1}e_2e_1)^p=(e_1e_2)^p-(e_2e_1)^p$  thanks to $e_i^2=1$.
\begin{remark}\label{asymmetric degenerate case}
\textup{For $\bullet=\as$ and $p=2$, the relation $e_1^2e_2-(q+q^{-1})e_1e_2e_1+e_2e_1^2=0$ is trivial ($0=0$) and we replace it with $(e_1e_2)^2-q^{-1}e_2e_1)^2=0$, see \cite[\S2]{ST2}.}
\end{remark}
Now, we introduce the braided Drinfeld double \cite{AA,LS} of $\mathfrak{N}(X^\bullet)$ 
$$\mathscr{U}^\bullet:=\mathrm{Drin}_{\mathcal{H}^\bullet}(\mathfrak{N}(X^\bullet),\mathfrak{N}(X^\bullet)^*),\quad (\bullet=\as,\s).$$
Following the notation \cite[Example 6.7]{CLR}, $\mathscr{U}^\bullet$ is generated by
$$ K_0,  H_i,  K_i^{\pm1}, x_i, x_i^*,\quad (i=1,2)$$
subject to the following.
\begin{enumerate}
    \item $ K_0,  H_i, K_i^{\pm1}$ generate $\mathcal{H}^\bullet$ in \eqref{Cartan subalgebra}.
    \item $\{  x_1, x_2\}$, $\{ x_1^*, x_2^*\}$ generate $\mathfrak{N}(X^\bullet)$, i.e. satisfy \eqref{qSerre relations} respectively.
    \item The weight condition holds:
    \begin{align}\label{weght condition}
\begin{array}{lll}
   {[} H_i, x_j{]}=c_{ij}^\bullet  x_j, &  K_i x_j  K_i^{-1}=q^{c_{ij}^\bullet}  x_j,  & K_0x_jK_0^{-1}=(-1)^{p(j)}x_j,\\
   {[}H_i,x_j^*{]}=-c_{ij}^\bullet x_j^*, & K_ix_j^*K_i^{-1}=q^{-c_{ij}^\bullet}x_j^*,  & K_0x_j^*K_0^{-1}=(-1)^{p(j)}x_j^*,
\end{array}
\end{align}
where $p(j)$ is the parity of $x_{j}$, $x_{j}^*$ given by the $j$-th diagonal entry of the first factor in \eqref{braiding matrix}.
    \item The linking relation holds: 
    \begin{align}\label{linking relations}
   x_i^*x_j-\mathcal{B}_{ij}^\bullet x_jx_i^*=\delta_{ij}(1-K_i^2),\quad (i,j=1,2).
\end{align}
\end{enumerate}
The algebra $\mathscr{U}^\bullet$ is naturally a Hopf algebra by the coproduct $\Delta$, counit $\varepsilon$, and the antipode $S$ is given by  
\begin{align*}
\begin{array}{lll}
    \Delta\colon&  x_i\mapsto  K_0^{p(j)}K_i\otimes x_i+x_i\otimes 1,&  x_i^*\mapsto  K_0^{p(j)}K_i\otimes x_i^*+x_i^*\otimes 1,\\
    \varepsilon\colon&  x_i, x_i^*\mapsto 0,&  \\
    S\colon &  x_i\mapsto - K_0^{p(j)}K_i^{-1}x_i,&  x_i^*\mapsto - K_0^{p(j)}K_i^{-1}x_i^*,
\end{array}
\end{align*}
and has a standard universal R-matrix, see e.g.\ \cite[\S 3]{LS}. Therefore, the category $\mathrm{Rep}^{\mathrm{wt}}(\mathscr{U}^\bullet)$ of left $q$-weight $\mathscr{U}^\bullet$-modules is a braided tensor category. 
Since $K_0$ exactly defines the super structure on $\scU^\bullet$ and their modules, we may omit $K_0$ from the definition of $\scU^\bullet$ and regard it as the Hopf superalgebra.
\begin{proposition}[\cite{ST1}] Suppose $p\geq 3$.\\
\textup{(1)} We have isomorphisms $F\colon \mathscr{U}^\as\rightarrow\mathscr{U}^\s$ and $G\colon \mathscr{U}^\s\rightarrow\mathscr{U}^\as$ of associative algebras
\renewcommand{\arraystretch}{1.3}
 \begin{align*}
  \begin{array}{lllll}
      F\colon & K_0, K_1, K_2 \mapsto K_0, K_1K_2, K_2^{-1}, &  & G\colon & K_0, K_1, K_2 \mapsto K_0, K_1K_2, K_2^{-1} \\
      &  H_1, H_2 \mapsto H_1+H_2, -H_2 & & & H_1, H_2 \mapsto H_1+H_2, -H_2\\
      &  x_1\mapsto x_1^*x_2^*+q^{-1}x_2^*x_1^* & & & x_1\mapsto x_2^*x_1^*-qx_1^*x_2^*\\
      &  \displaystyle{x_1^*\mapsto \frac{-(x_1x_2+q^{-1}x_2x_1)}{q-q^{-1}}} & & & \displaystyle{x_1^* \mapsto \frac{-(x_2x_1-qx_1x_2)}{q-q^{-1}}}\\
      & x_2, x_2^*\mapsto  x_2, -K_2^{-2}x_2^* &&& x_2, x_2^*\mapsto  x_2, -K_2^{-2}x_2^*.
  \end{array}
 \end{align*}
 \renewcommand{\arraystretch}{1}
 Moreover, they are inverse to each other. \\
 \textup{(2)} The isomorphisms $F$ and $G$ induce an equivalence of tensor categories
\begin{align*}
\mathrm{Rep}^{\mathrm{wt}}(\mathscr{U}^\as)\simeq \mathrm{Rep}^{\mathrm{wt}}(\mathscr{U}^\s).
\end{align*}
\end{proposition}
\proof
(1) Direct computation. 
(2) For $\scU^\s$-modules $M$, let $F^*M$ denote the $\scU^\as$-modules $M$ on which $\scU^\as$ acts via $F$. 
Clearly, it gives an equivalence of abelian categories
$F^*\colon \mathrm{Rep}^{\mathrm{wt}}(\scU^\s)\xrightarrow{\simeq}\mathrm{Rep}^{\mathrm{wt}}(\scU^\as)$.
Let $\delta^\bullet$ denote the coproduct for $\scU^\bullet$. Then it is straightforward to show that
\begin{align*}
    \Delta^\s (F(v))\cdot \Phi=\Phi\cdot (F\otimes F)(\Delta^\as(v)),\quad \Phi=1\otimes1-{}^\s K_0{}^\s K_2^{-1}{}^\s x_2^*\otimes{}^\s x_2
\end{align*}
for all $v\in \scU^\as$ ans thus 
$\Phi\colon F^*(M)\otimes F^*(N)\xrightarrow{\simeq} F^*(M\otimes N)$
are isomoprhisms of $\scU^\as$-modules for $\scU^\s$-modules $M,N$. Since $\Phi$ satisfies the cocycle condition
\begin{align*}
    (\Delta^\s\otimes 1)(\Phi)\cdot (\Phi\otimes 1)=(1\otimes \Delta^\s)(\Phi)\cdot (1\otimes \Phi), 
\end{align*}
$(F^*,\Phi)$ gives an equivalence of tensor categories.
\endproof

The unrolled restricted quantum supergroup $u_q^H(\sll_{2|1})$ ($q=e^{\pi\ssqrt{-1}/p}$, $p\geq2$) is the Hopf (super)algebra generated by 
\begin{align*}
   E_1, F_1, H_1, H_2, K_1^{\pm 1}, K_2^{\pm 1}\ \text{(even)},\quad E_2, F_2\ \text{(odd)}
\end{align*}
with relations 
\begin{align*}
    & H_iH_j=H_jH_i,\quad H_iK_j^{\pm1}=K_j^{\pm1}H_i,\quad K_i^{\pm1}K_j^{\pm1}=K_j^{\pm1}K_i^{\pm1},\\
    & {[} H_i, E_j{]}=c_{ij}^\as  E_j,\quad K_i E_j  E_i^{-1}=q^{c_{ij}^\as}  E_j,\\
    & {[} H_i, F_j{]}=-c_{ij}^\as  F_j,\quad K_i F_j  K_i^{-1}=q^{-c_{ij}^\as}  F_j,\\
    &e_1^p=0,\quad  e_2^2=0,\quad e_1^2e_2-(q+q^{-1})e_1e_2e_1+e_2e_1^2=0,\quad (e=E,F), \\
    &{[}E_i,F_j{]}=\delta_{i,j}\displaystyle{\frac{K_i-K_i^{-1}}{q-q^{-1}}},
\end{align*}
where we use the super commutator $[X,Y]=XY-(-1)^{p(X)p(Y)}YX$. The coproduct $\Delta$, counit $\varepsilon$, and the antipode $S$ are given by  
\renewcommand{\arraystretch}{1.3}
\begin{align*}
\begin{array}{lll}
    \Delta\colon&  H_i\mapsto  H_i\otimes 1+1\otimes  H_i,&  K_i\mapsto  K_i\otimes  K_i,\\
     &E_i\mapsto E_i\otimes 1+K_i^{-1}\otimes E_i,& F_i\mapsto F_i\otimes K_i+1\otimes F_i\\
    \varepsilon\colon&  H_i, E_i, F_i\mapsto 0,&  K_j\mapsto 1,\\
    S\colon &  H_i\mapsto - H_i,&  K_j\mapsto  K_j,\\
    & E_i\mapsto -K_iE_i, & F_i\mapsto -F_iK_i^{-1}.
\end{array}
\end{align*}
 \renewcommand{\arraystretch}{1}
We note that the word ``unrolled" corresponds to the additional generators $H_i$'s compared with the usual one $U_q(\sll_{2|1})$ and ``restricted" to the additional relations $E_1^p=F_1^p=0$. 
Th elements $E_1^p, F_1^p$ are not central in $U_q(\sll_{2|1})$, but the left ideals $(E_1^p)$, $(F_1^p)$ are still two-sided, whereas $E_1^{2p}, F_1^{2p}$ and central in $U_q(\sll_{2|1})$, see e.g. \cite{AAB,ACF,Ha}.
\begin{proposition}[\cite{ST1}]
 For $p\geq 3$, we have an isomorphism of Hopf algebras 
 \begin{align*}
    \mathscr{U}^\as\xrightarrow{\simeq} u_q^H(\sll_{2|1}).
 \end{align*}
\end{proposition}
\proof
Set elements $\widehat{E}_j,\widehat{F}_j\in \mathscr{U}^\as$ by the formulae
\begin{align*}
    x_1=\widehat{E}_1,\quad x_2=\widehat{E}_2,\quad x_1^*=(q-q^{-1})K_1\widehat{F}_1,\quad x_2^*=-(q-q^{-1})K_2\widehat{F}_2.
\end{align*}
Then we have an automorphism $\omega\colon \mathscr{U}^\as\xrightarrow{\simeq} \mathscr{U}^\as$ such that 
\begin{align*}
K_0\mapsto K_0,\quad K_i\mapsto K_i^{-1},\quad H_i\mapsto -H_i,\quad 
\widehat{E}_i\mapsto \widehat{F}_i,\quad 
\widehat{F}_i\mapsto (-1)^{\delta_{i,2}}\widehat{E}_i,\quad (i=1,2).
\end{align*}
Then, it is straightforward to show that we have an isomorphism of Hopf (super)algebras $u_q^H(\sll_{2|1}) \xrightarrow{\simeq} \scU^\as$ such that
\begin{align*}
H_i\mapsto H_i,\quad K_i\mapsto K_i,\quad 
E_i\mapsto (-1)^{\delta_{i,2}}\omega(\widehat{F}_iK_i),\quad F_i\mapsto \omega(K_i^{-1}\widehat{E}_i).
\end{align*}
This completes the proof.
\endproof
Since neither $u^H_q(\sll_{2|1})$ nor $\scU^\as$ are defined for the degenerate case $p=1$, i.e. $q=-1$, we define
$u^H_q(\sll_{2|1})$ as $\mathscr{U}^\s$ in this case.
\begin{remark}
\textup{For $p=2$, $\mathscr{U}^\as$ is related to $\overline{\mathcal{U}}^H_{\mathbf{i}}(\sll_3)$,in \cite[\S2]{CRR}.}
\end{remark}

To represent $\mathscr{U}$ for the remaining cases \eqref{extension pairs} requires the formulation of quasi-Hopf algebra \cite{GLO,CGR} as the associators of $\mathrm{Rep}^{\mathrm{wt}}(A)$ are nontrivial and thus $\mathcal{H}$ is already a quasi-Hopf algebra. Here, we present $\mathcal{H}$ as $\C$-algebra for $A=V_\Z\otimes V_{\ssqrt{p}A_1}\otimes \alpha^\dagger$ based on \eqref{Cartan subalgebra}. 
The category $\mathrm{Rep}^{\mathrm{wt}}(A)$ is semisimple with simple objects 
\begin{align*}
V_{\Z}^\pm\otimes V_{-\frac{a}{2\ssqrt{p}}\alpha+\sqrt{p}A_1}\otimes \pi^{\alpha^\dagger}_{-\frac{b}{2\ssqrt{p}}\alpha^\dagger},\quad (a\in \Z_{2p}, b\in \C).
\end{align*}
By Lemma \ref{free field equiv} (2), it corresponds to the $\mathcal{H}^\as$-module
\begin{align*}
   M_{a,b}:=\bigoplus_{n\in \Z}\C_{(\pm,\frac{a+2pn}{2}\beta^\as_1+b\gamma)},\quad \gamma:=\frac{1}{2}\beta^\as_1+\beta^\as_2.
\end{align*}
Let us introduce $\overline{H}_2=-(H_1+2H_2)$, $\overline{K}_2=(K_1K_2^2)^{-1}$. Then we find that 
$$K_1=q^a,\quad \overline{H}_2=b,\quad \overline{K}_2=q^b$$
hold as operators. We introduce the following subquotient of $\mathcal{H}^\as$:
\begin{align*}
\overline{\mathcal{H}}:=\C[\overline{H}_2,K_0, K_1^{\pm1},\overline{K}_2^{\pm}]/(K_0^2-1,K_1^{2p}-1). 
\end{align*}
Then we obtain the equivalence of abelian categories 
\begin{align*}
   \mathrm{Rep}^{\mathrm{wt}}(V_\Z\otimes V_{\ssqrt{p}A_1}\otimes \pi^{\alpha^\dagger})\simeq \mathrm{Rep}^{\mathrm{wt}}(\overline{\mathcal{H}}).
\end{align*}
We can find $\mathcal{H}$ as $\C$-algebras similarly for the cases $\FT(\sll_2)$ and $\mathcal{M}_p(\sll_2)$.
Note that $\overline{\mathcal{H}}$ is uprolled only for the direction $\C\beta^\as_1$ and still has the primitive element $\overline{H}_2$. This reflects the continuous parameter $b$ for the simple modules $M_{a,b}$. 
Through the equivalences of categories in Theorem\ref{Feigin-Semikhatov duality}, this corresponds with to the fact that $\FT(\sll_2)$ has a continuous family of simple modules $\afWb{r}{s}{b}$, see Theorem \ref{simplicity of X}.

\subsection{Logarithmic Kazhdan--Lusztig correspondence}
The following was conjectured by Semikhatov--Tipunin \cite{ST1} for $\FT(\sll_2)$. 
\begin{conjecture}[Semikhatov--Tipunin \cite{ST1}] \hspace{0mm}\\
\textup{(1)} $\mathrm{Rep}^{\mathrm{wt}}(V)$ with $V$ as in \eqref{extension pairs} has a structure of braided tensor category in the sense of Huang--Lepowsky--Zhang and, moreover,
\begin{align*}
    \mathrm{Rep}^{\mathrm{wt}}(V)\simeq {}^{\fN}_\fN\mathcal{YD}(\mathrm{Rep}^{\mathrm{wt}}(A))
\end{align*}
as braided tensor categories.\\
\textup{(2)} For $V=s\mathcal{M}_p(\sll_{2|1})$, we have, in particular,  
\begin{align*}
    \mathrm{Rep}^{\mathrm{wt}}(s\mathcal{M}_p(\sll_{2|1}))\simeq \mathrm{Rep}^{\mathrm{wt}}(u_q^H(\sll_{2|1})).
\end{align*}
\end{conjecture}
\begin{theorem} \label{log KL theorem}
The conjecture is true for $p=1$. In particular, we have an equivalence of braided tensor categories
\begin{align*}
    \mathrm{Rep}^{\mathrm{wt}}(s\mathcal{M}_1(\sll_{2|1}))\simeq \mathrm{Rep}^{\mathrm{wt}}(u_{-1}^H(\sll_{2|1})).
\end{align*}
\end{theorem}
\proof
We set $u_i=-\widetilde{\beta}_i^\s$, $(i=1,2)$ and $v=x-\alpha^\dagger$. Then the screening operators $\widehat{S}_i$ in Proposition \ref{super-side} are expressed as $\int Y(e^{-u_i},z)dz$ and thus we have the following embeddings
\begin{align*}
\xymatrix@=18pt{
s\mathcal{M}_1(\sll_{2|1}) \ar[r]^-{\simeq} & \displaystyle{\bigcap_{i=1,2}\ker \widehat{S}_i} \ar@{}[r]|*{\subset}& V_\Z\otimes \pi^{\alpha,\alpha^\dagger} \\
 \mathcal{M}(2)\otimes \mathcal{M}(2)\otimes \pi^v \ar@{}[u]|{\bigcup} \ar[r]^-{\simeq} & \displaystyle{\bigcap_{i=1,2}\ker \widehat{S}_i} \ar@{}[r]|*{\subset} & \pi^{u_1}\otimes \pi^{u_2}\otimes\pi^{v} \ar@{}[u]|{\bigcup}
}
\end{align*}
and decompositions
\begin{align}\label{description of SCE}
    s\mathcal{M}_1(\sll_{2|1})\simeq \displaystyle{\bigoplus_{n\in \Z} \mathcal{M}_n\otimes \mathcal{M}_n\otimes \pi^v_n},\ V_\Z\otimes \pi^{\alpha,\alpha^\dagger}\simeq \displaystyle{\bigoplus_{n\in\Z}\pi^{u_1}_{-nu_1}\otimes \pi^{u_2}_{-nu_2}\otimes\pi^{v}_{-nv}}
\end{align}
by Proposition \ref{super-side} and \eqref{decomposition of affine singlet for P=1}.
By \cite{CLR}, we have an equivalence
\begin{align}\label{equiv for the singlet}
    \mathrm{Rep}^{\mathrm{wt}}(\mathcal{M}(p))\xrightarrow{\simeq} \mathrm{Rep}^{\mathrm{wt}}(u^H_q(\sll_2))
\end{align}
of braided tensor categories with
\begin{align*}
    u^H_q(\sll_2)=\langle H_1, K_1^{\pm 1}, x_1, x_1^* \rangle \subset u^H_q(\sll_{2|1}),
\end{align*}
given by the braided Drinfeld double of the Nichols algebra for the braided vector space $\C \widetilde{\beta}_1^\as$ over the  Hopf algebra $\C[H_1,K_1^{\pm}]$.
By using \eqref{equiv for the singlet} with $p=2$ for each copy of $\mathcal{M}(2)$ inside $s\mathcal{M}_1(\sll_{2|1})$, we obtain 
\begin{align}\label{underlying equivalence}
    \mathrm{Rep}^{\mathrm{wt}}(\mathcal{M}(2)\otimes \mathcal{M}(2)\otimes \pi^v)\xrightarrow{\simeq} \mathrm{Rep}^{\mathrm{wt}}(\widetilde{\scU})
\end{align}
with $\widetilde{\scU}=\mathrm{Drin}_{\widetilde{\mathcal{H}}}(\mathfrak{N}(X^\s),\mathfrak{N}(X^\s)^*)$ 
over the Hopf algebra $\widetilde{\mathcal{H}}=\C[H_0,H_i,K_i^{\pm1} \mid i=1,2]$.
Here $H_0$ corresponds to $v$. 
By \eqref{description of SCE}, $s\mathcal{M}_1(\sll_{2|1})$ is a simple current extension of $\mathcal{M}(2)\otimes \mathcal{M}(2)\otimes \pi^v$ along the lattice 
$$L=\Z[u_1+u_2+v]=\Z x,$$
the equivalence \eqref{underlying equivalence} induces the desired equivalence 
\begin{align*}
    \mathrm{Rep}^{\mathrm{wt}}(s\mathcal{M}_1(\sll_{2|1}))\xrightarrow{\simeq} \mathrm{Rep}^{\mathrm{wt}}(u_{-1}^H(\sll_{2|1}))
\end{align*}
by \cite[\S 10]{CLR} since $u_{-1}^H(\sll_{2|1})$ coincides with the modification of $\widetilde{\scU}$ by replacing $\widetilde{\mathcal{H}}$ with $\mathcal{H}$ according to the uprolling along $L$. 
We can show the case $\mathcal{M}_1(\sll_2)$ by the same argument by using \eqref{decomposition of affine singlet for P=1} and then the case $\mathrm{FT}_1(\sll_2)$. This completes the proof.

\endproof

\end{document}